\documentclass{article}
\usepackage{amsmath}
\usepackage{amsfonts}
\usepackage{amsthm}
\usepackage{enumerate}

\usepackage{amssymb}
\newtheorem{thm}{Theorem}[section]
\newtheorem{cor}[thm]{Corollary}
\newtheorem{lem}[thm]{Lemma}
\newtheorem{prop}[thm]{Proposition}
\newtheorem{defn}[thm]{Definition}
\newtheorem{conj}[thm]{Conjecture}
\newtheorem*{thmnn}{Theorem}

\theoremstyle{remark}
\newtheorem{rmk}[thm]{Remark}

\begin{document}

\title{Thomae Formula for Abelian Covers of $\mathbb{CP}^{1}$}

\author{Yaacov Kopeliovich, Shaul Zemel}

\maketitle

%\begin{abstract}
%Abelian covers of $\mathbb{CP}^{1}$, with fixed Galois group $A$, are classified, as a first step, by a discrete set of parameters. Any such cover $X$, of genus $g\geq1$ say, carries a finite set of $A$-invariant divisors of degree $g-1$ on $X$ that produce non-zero theta constants on $X$. We show how to define a quotient involving a power of the theta constant on $X$ that is associated with such a divisor $\Delta$, some polynomial in the branching values, and a fixed determinant on $X$ that does not depend on $\Delta$, such that the quotient is constant on the moduli space of $A$-covers with the given discrete parameters. This generalizes the classical formula of Thomae, as well as all of its known extensions by various authors.
%\end{abstract}

\section*{Introduction}

The references \cite{[Th1]} and \cite{[Th2]} from the 19th century found a relation between the non-vanishing even theta constants on a hyperelliptic Riemann surface $X$ and certain polynomials in the branching values. More explicitly, if $X$ has genus $g$ and $\lambda_{j}$, $1 \leq j \leq 2g+2$ are the branching values, then every even theta characteristic $e$ whose theta constant does not vanish is associated with a partition of the set $\{\lambda_{j}\}_{j=1}^{2g+2}$ of branching values into two disjoint sets, say $I$ and $I^{c}$, each of cardinality $g+1$, and we define the polynomial $p_{e}$ to be $\prod_{\{i,j \in I|i<j\}}(\lambda_{i}-\lambda_{j})^{2}\prod_{\{i,j \in I^{c}|i<j\}}(\lambda_{i}-\lambda_{j})^{2}$. Choosing a canonical homology basis for $X$, say $a_{i}$ and $b_{i}$ with $1 \leq i \leq g$, we let $\tau$ be the (usual) period matrix arising from integrating the dual basis of the differentials of the first kind on $X$ along the $b_{i}$s, we define $C$ to be the matrix formed by integrating the natural basis for these differentials (which usually differs from the dual one) along the $a_{i}$s, and we set $\theta[e](0,\tau)$ to be the theta constant corresponding to the characteristic $e$. The results of \cite{[Th1]} and \cite{[Th2]} state that the quotient $\frac{\theta^{8}[e](0,\tau)}{(\det C)^{4}p_{e}}$ depends neither on $e$ nor on the branching values on $X$. These formulae are now known as \emph{Thomae formulae}, after the author of these references.

After laying dormant for about a century, the Thomae formulae returned to active research, partly due to emerging interest related to questions arising from physics. The first generalizations for cyclic covers were given by \cite{[BR1]} and \cite{[BR2]} for the non-singular case, using methods from conformal field theory. These results were later established rigorously in \cite{[Na]}. Numerous authors considered fully ramified cyclic covers in various degrees of generality---see \cite{[EiF]}, \cite{[EbF]}, \cite{[FZ]}, and the more general references \cite{[Z]} and \cite{[Ko2]} (we also mention the contributions of \cite{[M]} and \cite{[MT]}, as well as the application to Young tableaux in \cite{[Ko1]}). The preprint \cite{[Ko3]} contains the first non-cyclic cases of Thomae formulae (still using the tools from \cite{[Na]}). A more abstract and algebraic approach is taken by \cite{[SB]}, which eventually expresses an $N$-torsion theta constant as the theta constant of characteristic 0 times a quotient of fractional powers of Weil functions on Jacobians as principally polarized abelian varieties. These functions are also related to determinants of sections of certain line bundles on the corresponding algebraic curves. Note that \cite{[SB]} considers arbitrary compact Riemann surfaces (as complex algebraic curves), without the Galois condition over $\mathbb{CP}^{1}$, but evaluating the final results there requires some extra calculations, after unfolding all the definitions.

In this paper we show how to state and prove Thomae type formulae for general abelian covers of $\mathbb{CP}^{1}$ by combining this machinery with the formalism that we established in the previous paper \cite{[KZ]}. More precisely, we show how to generalize the methods from \cite{[Ko1]}, which are based on \cite{[Na]}, to the case of a general abelian cover of $\mathbb{CP}^{1}$. The result expresses the theta constants in very explicit terms, up to a global constant and a determinant of a specific matrix of integrals.

We remark that \cite{[Th2]} contains an additional formula, involving theta derivatives at characteristics lying on the generic locus of the theta divisor. Recently the authors found, with V. Enolskii, how to generalize this formula as well for $Z_{3}$ curves (see \cite{[EKZ]}). We expect to be able to extend these formulae as well to more general families of Riemann surfaces in the future.

\smallskip

The search for Thomae-like formulae is natural in the context of investigating special values of theta functions on the moduli space of algebraic covers, and such insights into the moduli space should have broader applications. Some of the applications include generalizations of the Picard $\lambda$-function to arbitrary abelian covers, the investigation of equations satisfied by period matrices of specific nature, and the connection to the representation theory of the monodromy groups of the underlying moduli spaces (see \cite{[Mu]} and \cite{[Ko1]} for examples of this type). Thomae formulae were also applied in \cite{[dJ]} in order to obtain explicit forms of the Mumford isomorphism, with an application to Diophantine problems for hyperelliptic curves (see also \cite{[JK]}). For other directions, like obtaining bounds on modular forms describing discriminants of hyperelliptic curves and physics, see \cite{[Lo]}, \cite{[vK]}, and \cite{[EG1]} (among others). Note that this type of Thomae formulae depends on choosing non-special divisors (either positive of degree $g$ or non-positive of degree $g-1$) that are invariant under the action of the abelian Galois group of $X$ over $\mathbb{CP}^{1}$, and such divisors need not always exist (see, e.g., \cite{[GDT]} for cyclic cases where this happens). We find some additional abelian covers carrying no such divisors (for simpler reasons), hence for which no effective Thomae formula is defined.

\smallskip

These types of applications indicate why it is worthwhile to have Thomae type formulae for more general algebraic curves. The method that we use, following \cite{[Na]} and \cite{[EG2]}, is to obtain explicit algebraic expressions for analytical objects such as the Szeg\H{o} kernel function (with our specific characteristics) and the canonical differential on $X \times X$. A general relation from \cite{[Fa]} yields equalities involving derivatives of theta functions, and the Rauch Variational Formula from \cite{[Ra]} allows us to transform these equalities into a simple differential equation for the theta constant (as a function of the branching values). The final result is then established by integrating this equation.

For stating the final formula, let $A$ be the Galois group of the cover, whose order and exponent are $n$ and $m$ respectively, and for every $Id_{X}\neq\sigma \in A$ we let $\lambda_{\sigma,j}$, $1 \leq j \leq r_{\sigma}$ be the $r_{\sigma}$ branching images in $\mathbb{C}\subseteq\mathbb{CP}^{1}$ that are associated to $\sigma$ as in \cite{[KZ]} (see Proposition \ref{liftcirc} for the explicit meaning of this association). We denote the order of the element $\sigma \in A$ by $o(\sigma)$, set $c_{\sigma,\rho}=\frac{(o(\sigma)-1)(o(\rho)-1)}{4}$, and after fixing a canonical homology basis for $X$, we let $\tau$ denote the resulting period matrix, and set $C$ to be the matrix of $a_{i}$-integrals of an appropriate non-normalized basis for the differentials of the first kind on $X$. We write $\theta[e](0,\tau)$ for the (non-vanishing) theta constant that is associated with the characteristic $e$ coming from a non-special $A$-invariant divisor, and the main result of this paper, stated in Theorem \ref{relpert}, then reads as follows.
\begin{thmnn}
The equality
\[\theta[e]^{4m}(0,\tau)=\alpha_{e}(\det C)^{2m}\prod_{(\sigma,j)<(\rho,i)}(\lambda_{\sigma,j}-\lambda_{\rho,i})^{\frac{2mn}{o(\sigma)o(\rho)}[2\phi_{h+d\mathbb{Z}}(\xi)+\phi_{h+d\mathbb{Z}}(0)+c_{\sigma,\rho}]}\] holds, where $\alpha_{e}$ is a constant that is independent of the branching values, and the $\phi$-terms are generalized Dedekind sums arising from $\sigma$, $\rho$, $j$, $i$, and $e$.
\end{thmnn}
The occurrence of these generalized Dedekind sums in our formulae suggests that the resulting rational functions may be connected to other interesting objects. We also investigate the dependence of the constant $\alpha_{e}$ on the characteristic $e$, namely we conjecture that an appropriate power of $\alpha_{e}$ would give the same value for every $e$. Since the form of the expressions we encounter in the proofs may be a bit difficult to grasp in general, we conclude with some examples, where all the expressions can be written explicitly in simpler forms.

\smallskip

This paper is divided into 8 sections. Section \ref{Divs} presents the required results from \cite{[KZ]}, including the form of the invariant divisors defining our characteristics. Section \ref{Chars} gathers some properties of the theta constants, the Abel--Jacobi map, and the divisors appearing in our analysis. Section \ref{Szego} constructs the Szeg\H{o} kernel using algebraic expressions, and evaluates its expansion near the diagonal. Section \ref{CanDif} investigates the form of the canonical differential in our case, including the Bergman projective connection arising from its expansion near the diagonal. In Section \ref{BrPts} we consider the expansions around branch points, and use the formula from \cite{[Fa]} to connect the expressions from the previous two sections. Section \ref{Thomae} then establishes a differential equation for the theta constants, and proves the main result. In Section \ref{Moduli} we consider the Thomae constants $\alpha_{e}$ on coarser and coarser moduli spaces, and pose the conjecture about powers of $\alpha_{e}$ being independent of $e$. Finally, Section \ref{Examples} contains some examples from the literature, shows why some non-cyclic abelian covers do not carry any invariant divisors of the required type, and gives some more explicit formulae for the Szeg\H{o} kernel and the canonical differential in very special cases.

\smallskip

We thank S. Grushevsky and H. Farkas for continued support and interest in this topic, as well as M. Fried, A. Rojas, M. Carvacho, and R. Livn\'{e} for interesting discussions. We also express our deep gratitude to the anonymous referees for many valuable suggestions to improve the presentation of this paper. Their contribution is gratefully acknowledged.

\section*{List of Notation}

$A$ --- An abelian group, typically of automorphisms of a Riemann surface $X$.

\noindent
$\widehat{A}$ --- The dual group $\mathrm{Hom}(A,\mathbb{C}^{\times})$ of $A$, consisting of characters on $A$.

\noindent
$a_{i}$ --- An element of a canonical homology basis on $X$, with $1 \leq i \leq g$.

\noindent
$b_{i}$ --- An element of a canonical homology basis on $X$, with $1 \leq i \leq g$.

\noindent
$\beta_{\sigma,j}^{\Delta}$ --- The coefficient of $z^{-1}(\lambda_{\sigma,j})$ in the divisor $\Delta$ from Equation \eqref{normdiv}.

\noindent
$C$ --- The non-normalized period matrix arising from Proposition \ref{holdif}.

\noindent
$\mathbb{CP}^{1}$ --- The complex projective line, namely the Riemann sphere $\mathbb{C}\cup\{\infty\}$.

\noindent
$\mathrm{Div}(X)$ --- The divisor group of the Riemann surface $X$.

\noindent
$\mathrm{div}(f)$, $\mathrm{div}(\omega)$ --- The divisor of the meromorphic function $f$ on $X$ or of the meromorphic differential $\omega$ on $X$.

\noindent
$\deg p$, $\deg\Delta$ --- The degree of the polynomial $p$ or of the divisor $\Delta$.

\noindent
$\Delta$ --- An $A$-invariant divisor on $X$, typically non-special of degree $g-1$.

\noindent
$\delta$ --- The second part of a theta characteristic, typically lies in $\mathbb{Q}^{g}$.

\noindent
$\delta_{\xi,\eta}$ --- The Kronecker $\delta$-symbol, which equals 1 when $\xi=\eta$ and 0 otherwise.

\noindent
$\mathbf{e}(w)$ --- A shorthand for $e^{2\pi iw}$ for complex $w$.

\noindent
$e$ --- A theta characteristic in $\mathbb{C}^{g}$, also written as $\tau\frac{\varepsilon}{2}+I\frac{\delta}{2}$.

\noindent
$\varepsilon$ --- The first part of a theta characteristic, typically lies in $\mathbb{Q}^{g}$.

\noindent
$g$ --- The genus of the Riemann surface $X$.

\noindent
$G_{B}$ --- The Bergman projective connection.

\noindent
$i(\Delta)$ --- The index of specialty of the divisor $\Delta$.

\noindent
$J(X)$ --- The Jacobian of the Riemann surface $X$.

\noindent
$K_{R}$ --- The vector of Riemann constants associated with the base point $R$.

\noindent
$\lambda_{\sigma,j}$ --- A branching value of $z$ that is associated with the element $\sigma \in A$ via Proposition \ref{liftcirc}.

\noindent
$m$ --- The exponent of the abelian group $A$.

\noindent
$\mathcal{M}_{A,\vec{r}}$, $\mathcal{M}_{A,\vec{r}}^{Tei}$, $\mathcal{M}_{A,\vec{r}}^{Tei,z}$, $\mathcal{M}_{A,\vec{r}}^{Tei,z,ord}$, $\mathcal{M}_{A,\vec{r}}^{z,ord}$ --- Moduli spaces (see Definition \ref{moduli}).

\noindent
$n$ --- The order of the abelian group $A$, as well as of its dual $\widehat{A}$.

\noindent
$N$ --- The operator on normalized $A$-invariant divisors that inverts the characteristics, as defined in Corollary \ref{NonAinv}.

\noindent
$o(\sigma)$ --- The order of the element $\sigma$ in $A$.

\noindent
$\mathcal{P}_{\leq d}(z)$ --- The space of polynomials of degree $\leq d$ in $z$, for an integer $d\geq-1$.

\noindent
$p_{\Delta,\chi}$ --- The polynomial required for normalizing $\chi\Delta$, defined in Equation \eqref{pDeltachi}.

\noindent
$\phi_{h+d\mathbb{Z}}(\xi)$ --- A Dedekind sum defined in Equation \eqref{phihd}.

\noindent
$\psi_{\chi}$ --- A meromorphic differential on $X$ on which $A$ acts by $\chi$, whose divisor is normalized and is described in Proposition \ref{decomgen}.

\noindent
$r_{\sigma}$ --- The number of values in $\mathbb{CP}^{1}$ to which Proposition \ref{liftcirc} associates the non-trivial element $\sigma \in A$.

\noindent
$r(-\Delta)$ --- The dimension of $\{f\in\mathbb{C}(X)|\mathrm{div}(f)\geq-\Delta\}$.

\noindent
$\rho$ --- An element of the abelian group $A$.

\noindent
$S^{1}$ --- The multiplicative group of complex numbers of absolute value 1.

\noindent
$S[e]$ --- The Szeg\H{o} kernel on $X$ that is associated with the characteristic $e$.

\noindent
$\mathcal{S}\{z,t\}$ --- The Schwarzian derivative of $z$ with respect to $t$.

\noindent
$\sigma$ --- An element of the abelian group $A$.

\noindent
$t_{\chi}$ --- The combination $\sum_{\sigma \neq Id_{X}}\frac{r_{\sigma}u_{\chi,\sigma}}{o(\sigma)}$, which is an integer by Proposition \ref{decomgen}.

\noindent
$\tau$ --- The period matrix of a Riemann surface with a canonical homology basis.

\noindent
$\theta[e]$, $\theta\big[{\textstyle{\varepsilon \atop \delta}}\big]$ --- The theta function with characteristic $e$ or $\big[{\textstyle{\varepsilon \atop \delta}}\big]$.

\noindent
$u_{R}$, $u$ --- The Abel--Jacobi map with base point $R$, and its restriction to divisors of degree 0 (which is independent of $R$).

\noindent
$u_{\chi,\sigma}$ --- The integer in $\big[0,o(\sigma)\big)$ describing $\chi(\sigma)$ as in Equation \eqref{uchisigma}.

\noindent
$v_{s}$ --- An differential of the first kind that belongs to a basis that is dual to a canonical homology basis.

\noindent
$\omega$ --- The canonical differential on $X \times X$.

\noindent
$X$ --- A compact Riemann surface, of genus $g$, typically an $A$-cover of $\mathbb{CP}^{1}$.

\noindent
$\chi$ --- A character on $A$, namely an element of $\widehat{A}$.

\noindent
$\lfloor x \rfloor$ --- The integral part of the real number $x$, namely $\max\{n\in\mathbb{Z}|n \leq x\}$.

\noindent
$\{x\}$ --- The fractional part $x-\lfloor x \rfloor$ of the real number $x$.

\noindent
$y_{\chi}$ --- A meromorphic function on $X$ on which $A$ acts by $\chi$, whose divisor is normalized and is described in Proposition \ref{decomgen}.

\noindent
$\eta$ --- A character on $A$, namely an element of $\widehat{A}$.

\noindent
$z$ --- A meromorphic function on $X$, as well as a map from $X$ to $\mathbb{CP}^{1}$, typically Galois with abelian Galois group $A$.

\noindent
$\zeta$ --- The $\mathbb{C}^{g}$-variable of theta functions.

\section{Non Special Divisors on Abelian Covers \label{Divs}}

In this section we give, following \cite{[KZ]}, the normalized form of the non-special invariant divisors of degree $g-1$ on a genus $g$ abelian cover of $\mathbb{CP}^{1}$.

\smallskip

Let $z:X\mapsto\mathbb{CP}^{1}$ be an abelian cover of the Riemann sphere, where we denote the Galois group by $A$, its order by $n$, and its trivial element (naturally) by $Id_{X}$. The order of an element $\sigma \in A$ is denoted by $o(\sigma)$. Then Section 1 of \cite{[KZ]} contains the following result, refining the signature of of the Galois map $z$.
\begin{prop}
Consider a simple closed oriented loop around a given point $\lambda\in\mathbb{CP}^{1}$, and lift it to $X$. Then the end point of the lift is the image of the starting point under an element $\sigma=\psi(\lambda) \in A$ that depends only on $\lambda$, and which generates the stabilizer in $A$ of any of the $\frac{n}{o(\sigma)}$ points $P \in X$ with $z(P)=\lambda$. The order of $z-\lambda$ (or $\frac{1}{z}$ in case $\lambda=\infty$) at any such point is $o(\sigma)$. \label{liftcirc}
\end{prop}
A point $\lambda\in\mathbb{CP}^{1}$ is associated via Proposition \ref{liftcirc} to a non-trivial element of $A$ if and only if it is a branching value of $z$. Therefore for every $Id_{X}\neq\sigma \in A$ the set of elements $\lambda\in\mathbb{CP}^{1}$ that are associated with $\sigma$ is finite, and we denote its cardinality by $r_{\sigma}$. We shall henceforth assume that $\infty$ is not a branching value of $z$, and set for $Id_{X}\neq\sigma \in A$ the notation \[\big\{\lambda\in\mathbb{C}\big|\psi(\lambda)=\sigma\big\}=\big\{\lambda_{\sigma,j}\big\}_{j=1}^{r_{\sigma}},\quad\mathrm{as\ well\ as}\quad z^{-1}(\lambda_{\sigma,j})=\sum_{\{P \in X|z(P)=\lambda_{\sigma,j}\}}P\] (the latter being a divisor of degree $\frac{n}{o(\sigma)}$ by Proposition \ref{liftcirc}). Similarly, we let $z^{-1}(\infty)$ be the divisor of poles of $z$, which consists of $n$ distinct points by our assumption that there is no branching over $\infty$. The divisors of $z-\lambda_{\sigma,j}$ with $Id_{X}\neq\sigma \in A$ and $1 \leq j \leq r_{\sigma}$ and of the differential $dz$ are therefore
\begin{equation}
o(\sigma)z^{-1}(\lambda_{\sigma,j})-z^{-1}(\infty)\mathrm{\ and\ }\sum_{\sigma \neq Id_{X}}\sum_{j=1}^{r_{\sigma}}\big(o(\sigma)-1\big)z^{-1}(\lambda_{\sigma,j})-2z^{-1}(\infty) \label{simpledivs}
\end{equation}
respectively, where we emphasize that we always consider expressions and divisors on $X$ (and never on $\mathbb{CP}^{1}$). The Riemann--Hurwitz formula then produces the following expression.
\begin{prop}
The genus $g$ of $X$ is $1-n+\sum_{\sigma \neq Id_{X}}\frac{nr_{\sigma}}{2o(\sigma)}\big(o(\sigma)-1\big)$. \label{genus}
\end{prop}

Denote the character group (or dual group) $\mathrm{Hom}(A,\mathbb{C}^{\times})$ of $A$, whose order is also $n$, by $\widehat{A}$, and its trivial element by $\mathbf{1}$. To every character $\chi\in\widehat{A}$ and every $\sigma \in A$ we determine a number $u_{\chi,\sigma}$ by the conditions
\begin{equation}
0 \leq u_{\chi,\sigma}<o(\sigma)\quad\mathrm{and}\quad\chi(\sigma)=\mathbf{e}\big(\tfrac{u_{\chi,\sigma}}{o(\sigma)}\big),\quad\mathrm{where}\quad\mathbf{e}(w)=e^{2\pi iw}\mathrm{\ for\ }w\in\mathbb{C}. \label{uchisigma}
\end{equation}
Sections 2 and 5 of \cite{[KZ]} then prove the following result.
\begin{prop}
The number $t_{\chi}=\sum_{\sigma \neq Id_{X}}\frac{r_{\sigma}u_{\chi,\sigma}}{o(\sigma)}$ is a non-negative integer for every $\chi\in\widehat{A}$, which vanishes if and only if $\chi=\mathbf{1}$. There is a meromorphic function $y_{\chi}$ and a meromorphic differential $\psi_{\chi}$ on $X$, both unique up to scalar multiples and on both of which $A$ operates via $\chi$, such that \[\mathrm{div}(y_{\chi})=\sum_{\sigma \neq Id_{X}}\sum_{j=1}^{r_{\sigma}}u_{\chi,\sigma}z^{-1}(\lambda_{\sigma,j})-t_{\chi}z^{-1}(\infty)\] and \[\mathrm{div}(\psi_{\chi})=\sum_{\sigma \neq Id_{X}}\sum_{j=1}^{r_{\sigma}}\big(o(\sigma)-1-u_{\overline{\chi},\sigma})z^{-1}(\lambda_{\sigma,j})+(t_{\overline{\chi}}-2)z^{-1}(\infty).\] \label{decomgen}
\end{prop}
The functions $y_{\chi}$ and the differentials $\psi_{\chi}$ from Proposition \ref{decomgen} were denoted by $h_{\chi}$ and $\omega_{\chi}$ respectively in \cite{[KZ]} (but here the letters $h$ and $\omega$ are used for different objects, whence the change in notation), and one can take $\psi_{\chi}=\frac{dz}{y_{\overline{\chi}}}$ for every $\chi\in\widehat{A}$. For an integer $d\geq-1$ we set $\mathcal{P}_{\leq d}(z)$ to be the space of polynomials of degree not exceeding $d$ in $z$, of dimension $d+1$ (so that the zero space is $\mathcal{P}_{\leq-1}(z)$), and then the space of differentials of the first kind on $X$ is described by Corollary 5.6 of \cite{[KZ]} and the paragraph following it as follows.
\begin{prop}
The space of differentials of the first kind on $X$ decomposes as $\bigoplus_{\mathbf{1}\neq\chi\in\widehat{A}}\mathcal{P}_{\leq t_{\overline{\chi}}-2}(z)\psi_{\chi}$. \label{holdif}
\end{prop}
The numbers $u_{\chi,\sigma}$ and $t_{\chi}$ and the functions $y_{\chi}$ satisfy a few useful relations.
\begin{lem}
For any $\sigma \in A$ and character $\chi\in\widehat{A}$ we have \[u_{\chi,\sigma}+u_{\overline{\chi},\sigma}=\left\{\begin{array}{ll}0 & \mathrm{If\ }\chi(\sigma)=1 \\ o(\sigma) & \mathrm{otherwise}\end{array}\right.\quad\mathrm{and}\quad t_{\chi}+t_{\overline{\chi}}=\textstyle{\sum_{\sigma\not\in\ker\chi}r_{\sigma}}.\] In addition, for any $\chi\in\widehat{A}$ there is the equality
\[\mathrm{div}(y_{\chi}y_{\overline{\chi}})=\mathrm{div}\big(\textstyle{\prod_{\sigma\not\in\ker\chi}\prod_{j=1}^{r_{\sigma}}(z-\lambda_{\sigma,j})}\big).\] \label{chicomp}
\end{lem}

\smallskip

In order to see these expressions more explicitly, recall from the structure theorem for finite abelian groups that if
\[A=\prod_{l=1}^{q}H_{l},\quad\mathrm{where\ }H_{l}\mathrm{\ is \ cyclic\ of \ order\ }m_{l}\mathrm{\ with\ generator\ }\rho_{l}\mathrm{\ and\ }\prod_{l=1}^{q}m_{l}=n,\] then $X$ is the fibered product of the $Z_{m_{l}}$-curves
\begin{equation}
Y_{l}=\big\{(z,w_{l})\in\mathbb{CP}^{1}\times\mathbb{CP}^{1}\big|w_{l}^{m_{l}}=F_{l}(z)\big\}\mathrm{\ for\ some\ monic\ }F_{l}\in\mathbb{C}(z)^{\times},\ 1 \leq l \leq q \label{Zmlcurve}
\end{equation}
(where a monic rational function is the quotient of monic polynomials), namely
\[X=\big\{(z,w_{1},\ldots,w_{q})\in(\mathbb{CP}^{1})^{q+1}\big|w_{l}^{m_{l}}=F_{l}(z),\ 1 \leq l \leq q\big\}\] (with some conditions on the $F_{l}$s for the irreducibility of the $Y_{l}$s and of $X$---see Section 1 of \cite{[KZ]} for more details). For $1 \leq l \leq q$ we define $\rho_{l} \in H_{l} \subseteq A$ by
\begin{equation}
\rho_{l}w_{l}=\mathbf{e}\big(\tfrac{1}{m_{l}}\big)w_{l}\quad\mathrm{and}\quad \rho_{l}z=z\quad(\mathrm{as\ well\ as}\quad\rho_{l}w_{k}=w_{k}\quad\mathrm{for\ }k \neq l\mathrm{\ on\ }X). \label{basisprod}
\end{equation}
Assume that \[\big\{\lambda\in\mathbb{C}\big|F_{l}(\lambda)=0\mathrm{\ or\ }F_{l}(\lambda)=\infty\mathrm{\ for\ some\ }1 \leq l \leq q\big\}\subseteq\{\lambda_{i}|1 \leq i \leq s\},\] so that for every $1 \leq l \leq q$ we have an expression of the form
\begin{equation}
\textstyle{F_{l}(z)=\prod_{j=1}^{s}(z-\lambda_{i})^{\alpha_{li}},\quad\mathrm{with}\quad\alpha_{li}\in\mathbb{Z}\quad\mathrm{such\ that}\quad m_{l}\big|\sum_{i=1}^{s}\alpha_{li}} \label{Flprod}
\end{equation}
(to avoid branching over $\infty$). Recall that
\begin{equation}
\mathrm{for\ }\chi\in\widehat{A}\mathrm{\ and\ }1 \leq l \leq q\mathrm{\ there\ is\ }e_{l}\in\mathbb{Z}\mathrm{\ such\ that \ } \chi(\rho_{l})=\mathbf{e}\big(\tfrac{e_{l}}{m_{l}}\big), \label{charprod}
\end{equation}
that $e_{l}$ is unique up to $m_{l}\mathbb{Z}$, and that every choice of $e_{l}$s yields an element of $\widehat{A}$. More explicitly we define, dually to Equation \eqref{basisprod}, the elements
\begin{equation}
\eta_{l}\in\widehat{A}\quad\mathrm{by}\quad\eta_{l}(\rho_{l})=\mathbf{e}\big(\tfrac{1}{m_{l}}\big)\quad\mathrm{and}\quad\eta_{l}(\rho_{k})=1\mathrm{\ when\ }k \neq l \label{basischar}
\end{equation}
(so that the character $\chi$ from Equation \eqref{charprod} is $\prod_{l=1}^{q}\eta_{l}^{e_{l}}$). In addition, we denote the \emph{integral part} of the real number $x$ (namely the largest integer that does not exceed $x$) by $\lfloor x \rfloor$, and the \emph{fractional part} $x-\lfloor x \rfloor$ of $x$ by $\{x\}$. Then Proposition 1.3 of \cite{[KZ]} and the discussion surrounding it yields the following description of the expressions from Propositions \ref{liftcirc} and \ref{decomgen}.
\begin{prop}
With these normalizations, Proposition \ref{liftcirc} associates to the point $\lambda_{i}$ the element $\prod_{l=1}^{q}\rho_{l}^{\alpha_{li}}$ of $A=\prod_{l=1}^{q}H_{l}$, namely the points $\lambda_{\sigma,j}$ for $\sigma=\prod_{l=1}^{q}\rho_{l}^{d_{l}}$ are those points $\lambda_{i}$ for which $\alpha_{li} \in d_{l}+m_{l}\mathbb{Z}$ for every $1 \leq l \leq q$. Any point in $\mathbb{CP}^{1}$ that is not $\lambda_{i}$ for some $1 \leq i \leq s$ is sent to $Id_{X}$. If $\chi\in\widehat{A}$ and $\{e_{l}\}_{l=1}^{q}\in\mathbb{Z}^{q}$ satisfy Equation \eqref{charprod} then we can take \[y_{\chi}=\prod_{l=1}^{q}y_{l}^{e_{l}}\Bigg/\prod_{i=1}^{s}(z-\lambda_{i})^{\lfloor\sum_{l=1}^{q}e_{l}\alpha_{li}/m_{l}\rfloor}.\] If $P \in X$ satisfies $z(P)=\lambda_{i}$ then $y_{\chi}$ has the same order at $P$ as the local function $(z-\lambda_{i})^{\{\sum_{l=1}^{q}e_{l}\alpha_{li}/m_{l}\}}$, and when $\lambda_{i}$ is $\lambda_{\sigma,j}$ for some $\sigma \in A$ and $1 \leq j \leq r_{\sigma}$ then this order is $o(\sigma)\big\{\sum_{l=1}^{q}\frac{e_{l}\alpha_{li}}{m_{l}}\big\}$. The number $t_{\chi}$ equals $\sum_{i=1}^{s}\big\{\sum_{l=1}^{q}\frac{e_{l}\alpha_{li}}{m_{l}}\big\}$. \label{fibprod}
\end{prop}
Note that with the choices from Proposition \ref{fibprod} the equality involving $y_{\chi}y_{\overline{\chi}}$ in Lemma \ref{chicomp} holds also as an equality of functions (not only of divisors).

\begin{rmk}
It is convenient to normalize the $Z_{m_{l}}$-equation for $Y_{l}$ in Equation \eqref{Zmlcurve}, like in \cite{[FZ]}, such that $F_{l}$ is a (monic) polynomial containing no $\lambda_{i}$ to a power exceeding $m_{l}$, i.e., the powers $\alpha_{li}$ from Equation \eqref{Flprod} satisfy $0\leq\alpha_{li}<m_{l}$ for every $1 \leq i \leq s$ (and the condition of divisibility by $m_{l}$). Then Proposition \ref{fibprod} shows that $y_{l}$ coincides with $y_{\eta_{l}}$ for $\eta_{l}$ from Equation \eqref{basischar}, and for $e\in\mathbb{Z}$ the function $y_{\eta_{l}^{e}}$ is \[\mathrm{the\ normalized\ }e\mathrm{th\ power\ }\textstyle{y_{l}^{(e)}=y_{l}^{e}\big/\prod_{i=1}^{s}(z-\lambda_{i})^{e\alpha_{li}-m_{l}\lfloor e\alpha_{li}/m_{l} \rfloor}}\mathrm{\ of\ }y_{l}.\] The expression for $y_{l}^{(e)}$ remains invariant when $e$ is changed by a multiple of $m_{l}$. In general, when $\chi=\prod_{l=1}^{q}\eta_{l}^{e_{l}}$ as in Equations \eqref{charprod} and \eqref{basischar} the function $y_{\chi}$ is $\prod_{l=1}^{q}y_{l}^{(e_{l})}$ divided by a polynomial. On the other hand, \cite{[Ko3]} considers a situation where the branching values of the $Y_{l}$s are mutually disjoint, namely for every $1 \leq i \leq s$ the exponent $\alpha_{li}$ can be positive for at most one index $l$. Then for such $\chi$ we have \[\textstyle{y_{\chi}=\prod_{l=1}^{q}y_{l}^{(e_{l})},\qquad\psi_{\overline{\chi}}=\frac{dz}{y_{\chi}}=\frac{dz}{\prod_{l=1}^{q}y_{l}^{(e_{l})}},\qquad\mathrm{and}\qquad t_{\chi}=\sum_{l=1}^{q}\big\{\frac{e_{l}\alpha_{li}}{m_{l}}\big\}},\] and Proposition \ref{holdif} implies that a basis for the differentials of the first kind on $X$ is given by
\[\Bigg\{\frac{z^{k}dz}{\prod_{l=1}^{q}y_{l}^{(e_{l})}}\bigg|0 \leq e_{l}<m_{l},\ \sum_{l=1}^{q}e_{l}>0,\ 0 \leq k \leq t_{\chi}-2\Bigg\}.\] \label{normFl}
\end{rmk}
The quotient $\frac{dz}{y_{l}^{(e)}}$, where $y_{l}^{(e)}$ is the normalized $e$th power from Remark \ref{normFl}, is the differential denoted by $\omega_{k}$ in the cyclic case considered in \cite{[Z]}. While the fibered product structure is not necessary for proving any of the following results, it may be useful for some extreme cases (see Remark \ref{eoford2} below).

\smallskip

We shall need $A$-invariant divisors $\Delta$ of degree $g-1$ on $X$ that are not linearly equivalent to any positive divisor (this condition is expressed as $r(-\Delta)=0$ in \cite{[KZ]} and others). Such divisors are called \emph{non-special} (as their index of specialty vanishes, by the Riemann--Roch Theorem). It is easy to see, via Equation \eqref{simpledivs}, that any $A$-invariant divisor on $X$ can be written uniquely as the divisor of a rational function of $z$ plus a divisor of the form
\begin{equation}
\Delta=\sum_{\sigma \neq Id_{X}}\sum_{j=1}^{r_{\sigma}}\beta_{\sigma,j}^{\Delta}z^{-1}(\lambda_{\sigma,j})-h^{\Delta}z^{-1}(\infty),\quad0\leq\beta_{\sigma,j}^{\Delta}<o(\sigma),\quad h^{\Delta}\in\mathbb{Z} \label{normdiv}
\end{equation}
(see Proposition 2.4 and Corollary 2.8 of \cite{[KZ]} for the more general setting). An $A$-invariant divisor that can be expressed as in Equation \eqref{normdiv} is called \emph{normalized}. Adding a divisor of the form $\mathrm{div}(y_{\chi})$ from Proposition \ref{decomgen} to $\Delta$ clearly yields another $A$-invariant divisor that is linearly equivalent to $\Delta$, and Lemma 2.1 of \cite{[KZ]} implies that the $A$-invariant divisors that are linearly equivalent to $\Delta$ are precisely those that can be obtained from $\Delta+\mathrm{div}(y_{\chi})$ for the various $\chi\in\widehat{A}$ by adding divisors of rational functions of $z$. There is precisely one such normalized divisor for every $\chi$, but note that if $\Delta$ itself is normalized, $\Delta+\mathrm{div}(y_{\chi})$ need no longer be normalized. It would be useful for us to find its normalized form.

In order do so explicitly we write our divisor $\Delta$ as in Equation \eqref{normdiv}, and for a character $\chi\in\widehat{A}$, an element $Id_{X}\neq\sigma \in A$, and an index $1 \leq j \leq r_{\sigma}$ we define
\begin{equation}
\varepsilon_{\Delta,\sigma,j,\chi}=\left\{\begin{array}{ll}1 & \mathrm{If\ }\beta_{\sigma,j}^{\Delta} \geq o(\sigma)-u_{\chi,\sigma} \\ 0 & \mathrm{otherwise,}\end{array}\right.\mathrm{and\ }\beta_{\sigma,j}^{\chi\Delta}=\beta_{\sigma,j}^{\Delta}+u_{\chi,\sigma}-o(\sigma)\varepsilon_{\Delta,\sigma,j,\chi}. \label{epsDelsigjchi}
\end{equation}
Note that we have $0\leq\beta_{\sigma,j}^{\chi\Delta}<o(\sigma)$ in Equation \eqref{epsDelsigjchi} by the definition of $\varepsilon_{\Delta,\sigma,j,\chi}$. For our $\Delta$ and $\chi$ we also set
\begin{equation}
p_{\Delta,\chi}(z)=\prod_{\sigma \neq Id_{X}}\prod_{j=1}^{r_{\sigma}}(z-\lambda_{\sigma,j})^{\varepsilon_{\Delta,\sigma,j,\chi}}=\prod_{\sigma \neq Id_{X}}\prod_{\{j|\beta_{\sigma,j}^{\Delta} \geq o(\sigma)-u_{\chi,\sigma}\}}(z-\lambda_{\sigma,j}), \label{pDeltachi}
\end{equation}
and deduce the following lemma.
\begin{lem}
If $\Delta$ is a normalized $A$-invariant divisor and $\chi$ is in $\widehat{A}$ then the normalized divisor $\chi\Delta$ arising from $\Delta+\mathrm{div}(y_{\chi})$ is
\[\Delta+\mathrm{div}(y_{\chi})-\mathrm{div}(p_{\Delta,\chi})=\sum_{\sigma \neq Id_{X}}\sum_{j=1}^{r_{\sigma}}\beta_{\sigma,j}^{\chi\Delta}z^{-1}(\lambda_{\sigma,j})-(h^{\Delta}+t_{\chi}-\deg p_{\Delta,\chi})z^{-1}(\infty).\] The map sending $\chi$ and $\Delta$ to $\chi\Delta$ defines a free action of $\widehat{A}$ on the set of normalized $A$-invariant divisors, such that two such divisors are linearly equivalent if and only if they lie in the same $\widehat{A}$-orbit. \label{actAdual}
\end{lem}

\begin{proof}
The explicit expression for the combination defining $\chi\Delta$ follows from the formulae for $\Delta$ in Equation \eqref{normdiv}, for $\mathrm{div}(y_{\chi})$ in Proposition \ref{decomgen}, and for $p_{\Delta,\chi}$ in Equation \eqref{pDeltachi}, together with the definition of $\beta_{\sigma,j}^{\chi\Delta}$ in Equation \eqref{epsDelsigjchi}. The fact that this divisor is normalized (hence indeed equals $\chi\Delta$) follows from the bounds on the latter coefficients. The formula for $\chi\Delta$ and the uniqueness of normalizations via rational functions of $z$ imply that $(\chi,\Delta)\mapsto\chi\Delta$ is indeed an action of $\widehat{A}$, and the relation to linear equivalence was shown to follow from Lemma 2.1 of \cite{[KZ]}. This proves the lemma.
\end{proof}
The normalization of $\chi\Delta$ in Lemma \ref{actAdual} justifies a fortiori the notation $\beta_{\sigma,j}^{\chi\Delta}$ in Equation \eqref{epsDelsigjchi}.

Recall that we are interested in those normalized $A$-invariant divisors that have degree $g-1$ and are non-special. These divisors are now described in Theorem 4.5 of \cite{[KZ]}. As the set $B_{\sigma,i}$ from that theorem is $\{j|\beta_{\sigma,j}^{\Delta}=o(\sigma)-1-i\}$, its statement reads as follows.
\begin{thm}
The divisor $\Delta$ from Equation \eqref{normdiv} satisfies $\deg\Delta=g-1$ and $r(-\Delta)=0$ if and only if $h^{\Delta}=1$ and the equality \[\sum_{\sigma \neq Id_{X}}\big|\{j|\beta_{\sigma,j}^{\Delta} \geq o(\sigma)-u_{\chi,\sigma}\}\big|=t_{\chi}\] holds for every $\chi\in\widehat{A}$. \label{nonspdiv}
\end{thm}

An simple consequence of Lemma \ref{actAdual} and Theorem \ref{nonspdiv} is the following.
\begin{cor}
If the divisor $\Delta$ from Equation \eqref{normdiv} satisfies the equalities from Theorem \ref{nonspdiv} then so does the divisor $\chi\Delta$ from Lemma \ref{actAdual} for any $\chi\in\widehat{A}$. Hence the set of non-special normalized $A$-invariant divisors of degree $g-1$ on $X$ is a union of finitely many $\widehat{A}$-orbits. \label{chiDeltanonsp}
\end{cor}

\begin{proof}
As $\Delta$ and $\chi\Delta$ are linearly equivalent in Lemma \ref{actAdual}, they have the same degree and they are non-special together. Hence the first assertion follows immediately from Theorem \ref{nonspdiv}. Since the set of possible choices for the numbers $\beta_{\sigma,j}^{\Delta}$ in Equation \eqref{normdiv} is finite, the equality $h^{\Delta}=1$ in Theorem \ref{nonspdiv} implies the finiteness of the set of divisors in question. The second assertion now follows from the first one, with this finiteness property. This proves the corollary.
\end{proof}
The part about the invariance of the equality $h^{\Delta}=1$ in the first assertion of Corollary \ref{chiDeltanonsp} under the action of $\widehat{A}$ can also be seen in Lemma \ref{actAdual}, since Theorem \ref{nonspdiv} shows that the degree of the polynomial $p_{\Delta,\chi}$ from Equation \eqref{pDeltachi} is $t_{\chi}$ in this case. For the finiteness property in a more general setting, see Corollary 3.3 of \cite{[KZ]}.

\section{Additive Inverses of Torsion Characteristics \label{Chars}}

This section presents the basic properties of theta functions on abelian covers of $\mathbb{CP}^{1}$, whose characteristics are associated with the divisors from Theorem \ref{nonspdiv}.

\smallskip

Consider a (compact) Riemann surface $X$, with a canonical homology basis $a_{i}$ and $b_{i}$ with $1 \leq i \leq g$. Let $v_{s}$, $1 \leq s \leq g$ be the basis for the differentials on the first kind on $X$ that is dual to the cycles $a_{i}$, $1 \leq i \leq g$ (namely $\int_{a_{i}}v_{s}=\delta_{i,s}$), and let $\tau$ be the resulting period matrix with entries $\int_{b_{i}}v_{s}$, which is symmetric with positive definite imaginary part. Then any vector $e\in\mathbb{C}^{g}$ has a unique presentation as $\tau\frac{\varepsilon}{2}+I\frac{\delta}{2}$ for real vectors $\varepsilon$ and $\delta$, and one defines, for $\tau$ as above and $\zeta\in\mathbb{C}^{g}$, the \emph{theta function with characteristics} by
\[\theta[e](\zeta,\tau)=\theta\big[{\textstyle{\varepsilon \atop \delta}}\big](\zeta,\tau)=\sum_{\mu\in\mathbb{Z}^{g}}\mathbf{e}\big[\big(\mu+\tfrac{\varepsilon}{2}\big)^{t}\tfrac{\tau}{2}\big(\mu+\tfrac{\varepsilon}{2}\big)
+\big(\mu+\tfrac{\varepsilon}{2}\big)^{t}\big(\zeta+\tfrac{\delta}{2}\big)\big].\] All of these functions satisfy the \emph{heat equation}
\begin{equation}
\frac{\partial^{2}\theta}{\partial\zeta_{r}\partial\zeta_{s}}=\left\{\begin{array}{ll}4\pi i\frac{\partial\theta}{\partial\tau_{rs}} & \mathrm{if\ }r=s \\ 2\pi i\frac{\partial\theta}{\partial\tau_{rs}} & \mathrm{if\ }r \neq s\end{array}\right.=2\pi i(1+\delta_{r,s})\frac{\partial\theta}{\partial\tau_{rs}} \label{heat}
\end{equation}
for every $r$ and $s$. For the basic properties of these theta functions see Chapter 6 of \cite{[FK]}, Section 1.3 of \cite{[FZ]}, and some of the references therein.

Let $\Lambda_{\tau}$ be the lattice in $\mathbb{C}^{g}$ that is generated by the columns of the identity matrix $I$ and the columns of $\tau$, and we identify the Jacobian $J(X)$ of $X$ with $\mathbb{C}^{g}/\Lambda_{\tau}$ as usual. We denote the Abel--Jacobi map on divisors of degree 0 on $X$ (expressed via the basis $\{v_{s}\}_{s=1}^{g}$ for the differentials of the first kind on $X$) by $u$, and given a base point $R \in X$ we denote the Abel--Jacobi map with base point $R$, which takes any divisor on $X$ to $u(\Delta-\deg\Delta \cdot R)$, by $u_{R}$. We recall that if $K_{R}$ is the vector of Riemann constants associated with the base point $R$ and $\Delta$ is a divisor of degree $g-1$ then \[u(\Delta)+K=u_{R}(\Delta)+K_{R}=u\big(\Delta-(g-1)R\big)+K_{R}\mathrm{\ is\ independent\ of\ }R,\] whence the notation $u(\Delta)+K$ not including $R$ (for a proof of this independence see, e.g., Theorem 1.12 of \cite{[FZ]}, though this result is of course much older).

We now state the Riemann Vanishing Theorem, which is proved as Theorem 1.9 of \cite{[FZ]} (among others, including much older references).
\begin{thm}
We have $\theta[0](\zeta,\tau)=0$ with $\zeta\in\mathbb{C}^{g}$ precisely when the element $\zeta+\Lambda_{\tau} \in J(X)$ equals $u(\Delta)+K$ for a positive divisor $\Delta$ of degree $g-1$. \label{RiemVan}
\end{thm}
A related result, which is a consequence Theorem \ref{RiemVan}, is the following.
\begin{prop}
For $e\in\mathbb{C}^{g}$ and $P \in X$ we have a dichotomy: Either the expression $\theta[e]\big(u(P-Q)\big)$ vanishes identically for $Q \in X$, or the divisor $\Xi_{e,P}$ of its zeros is non-special of degree $g$, and it is characterized by the equality \[u(\Xi_{e,P}-P)+K=u_{P}(\Xi_{e,P})+K_{P}=e+\Lambda_{\tau}.\] \label{thetaonX}
\end{prop}
We recall that the expression from Proposition \ref{thetaonX} is not a function on $X$, but rather a section of a non-trivial line bundle on that Riemann surface. By the evenness of theta functions with respect to characteristics, considering $\theta[e]\big(u(P-Q)\big)$ as a function of $P$ (with $Q$ fixed) yields either identically zero or an expression with divisor $\Xi_{-e,Q}$ of zeros. We shall also be needing the following result, appearing, among others, as Proposition 2.1 of \cite{[Z]}.
\begin{prop}
For $R \in X$ and a non-zero meromorphic differential $\omega$ on $X$ we have $\deg\mathrm{div}(\omega)=2g-2$ and \[u_{R}\big(\mathrm{div}(\omega)\big)=u\big(\mathrm{div}(\omega)-(2g-2)R\big)=-2K_{R}.\] \label{candivAJ}
\end{prop}

\smallskip

For any point $Q$ on a compact Riemann surface $X$ and any non-special positive divisor $\Xi$ of degree $g$ not containing $Q$ in its support, Theorem 2.3 of \cite{[Z]} defines another divisor $N_{Q}(\Xi)$ with the same properties such that the elements $u_{Q}(\Xi)+K_{Q}$ and $u_{Q}\big(N_{Q}(\Xi)\big)+K_{Q}$ are inverses in $J(X)$. This operation is expressed in the language of divisors of degree $g-1$ in the following immediate consequence of Proposition \ref{candivAJ} and the Riemann--Roch Theorem.
\begin{lem}
If $\Delta$ is a divisor of degree $g-1$ on $X$ with $r(-\Delta)=0$ and $\omega$ is any non-zero meromorphic differential on $X$ then $\Gamma=\mathrm{div}(\omega)-\Delta$ also has degree $g-1$ and satisfies $r(-\Gamma)=0$, and the elements $u(\Delta)+K$ and $u(\Gamma)+K$ of $J(X)$ are additive inverses. If $Q$ is any point on $X$ then there is a unique positive divisor $\Xi_{\Delta,Q}$ of degree $g$, depending only on the linear equivalence class of $\Delta$, that is linearly equivalent to $\Delta+Q$, and this divisor does not contain $Q$ in its support. Moreover, $\Xi_{\Gamma,Q}$ is defined similarly, it is independent of the choice of $\omega$, and it equals $N_{Q}(\Xi_{\Delta,Q})$. \label{negation}
\end{lem}
The divisor $\Xi_{e,P}$ from Proposition \ref{thetaonX} is precisely $\Xi_{\Delta,P}$ from Lemma \ref{negation}, when $e$ and $\Delta$ are related via the equality $e=u(\Delta)+K$. Its $N_{P}$-image is therefore $\Xi_{-e,P}$.

\smallskip

We need the operator from Lemma \ref{negation} for our case, where $X$ is an abelian cover of the sphere via the map $z$, with Galois group $A$. When $\Delta$ is a normalized $A$-invariant divisor, Lemma \ref{negation} simplifies, via Equation \eqref{simpledivs}, as follows.
\begin{cor}
If $\Delta$ is a normalized $A$-invariant divisor satisfying the conditions of Theorem \ref{nonspdiv} then $N\Delta=\mathrm{div}(dz)-\Delta$ also has these properties, with $\beta_{\sigma,j}^{N\Delta}$ being $o(\sigma)-1-\beta_{\sigma,j}^{\Delta}$ for every $Id_{X}\neq\sigma \in A$ and $1 \leq j \leq r_{\sigma}$. The images of these divisors under $u+K$ are additive inverses in $J(X)$. \label{NonAinv}
\end{cor}
Equation \eqref{simpledivs} also explains (via Theorem \ref{nonspdiv}) why the coefficient of $z^{-1}(\infty)$ in $N\Delta$ is $-1$ as well. Corollary \ref{NonAinv} lifts the natural involution $\xi \mapsto K_{X}-\xi$ on $\mathrm{Pic}^{g-1}(X)$ to a map on divisors of degree $g-1$ (using $dz$ as the ``canonical'' choice of differential on $X$), and shows that the lifted involution preserves the set of (normalized) divisors from Theorem \ref{nonspdiv}. Combining Equation \eqref{pDeltachi} with Lemma \ref{chicomp} and Corollary \ref{NonAinv} shows that the product $p_{\Delta,\chi}p_{N\Delta,\overline{\chi}}$ equals the polynomial from that lemma (whose divisor on $X$ coincides with $\mathrm{div}(y_{\chi}y_{\overline{\chi}})$), which yields the following consequence.
\begin{cor}
The actions of $N$ and $\widehat{A}$ on the set of divisors $\Delta$ from Theorem \ref{nonspdiv}, given in Corollary \ref{NonAinv} and Lemma \ref{actAdual} respectively, satisfy the equality $N\chi\Delta=\overline{\chi}N\Delta$ for every such $\chi$ and $\Delta$. Therefore $N$ and $\widehat{A}$ generate a generalized dihedral group of operators acting on this set of divisors. \label{dihedral}
\end{cor}
Corollary \ref{dihedral} generalizes Lemma 5.2 of \cite{[Z]} from the cyclic case (including our notation $N$). On characteristics (i.e., images in $J(X)$ under $u+K$) this action reduces to the involution on $\mathrm{Pic}^{g-1}(X)$ mentioned above, with $N$ representing the non-trivial action. In the language of positive divisors of degree $g$, the fact that at the end of Section 3 of \cite{[Z]} we required an operator $N_{\beta}$ with the index $\beta$ is a reminiscent of the fact that the divisors come in classes that are $\widehat{A}$-orbits.

\section{An Expression for the Szeg\H{o} Kernel \label{Szego}}

Given a compact Riemann surface $X$ (with $\tau$ and $u$ as above) and a point $e\in\mathbb{C}^{g}$ with $\theta[e](0,\tau)\neq0$, the \emph{Szeg\H{o} kernel associated with $e$} is defined to be \[S[e](P,Q)=\frac{\theta[e]\big(u(P-Q),\tau\big)}{\theta[e](0,\tau)E(P,Q)},\quad\mathrm{with\ }P\mathrm{\ and\ }Q\mathrm{\ in\ }X.\] Here $E(\cdot,\cdot)$ is an anti-symmetric holomorphic $\big(-\frac{1}{2},-\frac{1}{2}\big)$-form on $X \times X$, called the \emph{prime form}, that vanishes only along the diagonal $P=Q$, and whose expansion near the diagonal in some coordinate chart $z$ is \[E(P,Q)=\frac{z(P)-z(Q)}{\sqrt{dz(P)}\sqrt{dz(Q)}}\Big[1+O\Big(\big(z(P)-z(Q)\big)^{2}\Big)\Big].\] In this section we shall give, when $X$ is an abelian cover of $\mathbb{CP}^{1}$ and $e$ arises from one of the divisors from Theorem \ref{nonspdiv}, an explicit expression for $S[e]$, and deduce some consequences from its expansion near the diagonal. Most of the proofs are based on Taylor expansions of explicit expressions.

Some useful properties of $S[e]$ (still in the general case) are given in the following proposition, a proof of which can be found in \cite{[Fa]} (see pages 19 and 123 there), \cite{[EG1]} (see Sections 3 and 4 of that reference, in particular the proof of Theorem 4.7), or \cite{[Na]}. It can also be established using direct calculations and the considerations from the previous section.
\begin{prop}
$S[e]$ is a $\big(\frac{1}{2},\frac{1}{2}\big)$-form that depends only on the image of $e$ in $J(X)$. The divisor of $Q \mapsto S[e](P,Q)$ for fixed $P$ is $\Xi_{e,P}-P$, where $\Xi_{e,P}$ is defined in Proposition \ref{thetaonX}, while for fixed $Q$ the divisor of $P \mapsto S[e](P,Q)$ is $\Xi_{-e,Q}-Q$. The former (resp. latter) function is the section of the unitary line bundle of degree $g-1$ corresponding to $e$ (resp. $-e$). In particular, $S[e]$ is holomorphic on $X \times X$ except for a simple pole along the diagonal, the expansion around which in a coordinate chart $z$ as above is of the form
\[\frac{\sqrt{dz(P)}\sqrt{dz(Q)}}{z(P)-z(Q)}\bigg[1+\sum_{s=1}^{g}\frac{\partial\ln\theta[e]}{\partial\zeta_{s}}\bigg|_{\zeta=0}\frac{v_{s}(Q)}{dz(Q)}\big(z(P)-z(Q)\big)+O\Big(\big(z(P)-z(Q)\big)^{2}\Big)\!\bigg]\!.\]
Finally, these properties (even with just $1+O\big(z(P)-z(Q)\big)$ in the brackets above) determine the $\left(\frac{1}{2},\frac{1}{2}\right)$-form $S[e]$ uniquely. \label{Szegoprop}
\end{prop}
By the \emph{unitary line bundle} appearing in Proposition \ref{Szegoprop} we mean the realization that transforms according to the unique associated unitary character of the homology of $X$, i.e., in which adding a cycle $a_{i}$ or $b_{i}$ to $P$ (resp. $Q$) multiplies the value of $S[e](P,Q)$ by $\mathbf{e}\big(\frac{\varepsilon_{i}}{2})$ or $\mathbf{e}\big(-\frac{\delta_{i}}{2})$ (resp. their inverses) when $e$ is written as $\tau\frac{\varepsilon}{2}+I\frac{\delta}{2}$. The expansion from Proposition \ref{Szegoprop} is known to be invariant under a change of the coordinate.

The only part of Proposition \ref{Szegoprop} that does not follow directly from the basic properties of the theta functions or of the prime form is the last assertion. Its proof appears implicitly in Theorem 1.1 of \cite{[Na]}, Theorem 7.2 of \cite{[Ko2]}, and Theorem 6.2 of \cite{[Ko3]}, but it holds equally well for every Riemann surface $X$ (regardless of a Galois cover structure). The idea is that the line bundle of which $S[e]$ is a section is the tensor product of pull-backs of fixed line bundles on the two copies of $X$, both of which have no non-zero global holomorphic sections. Hence the difference between two sections of this line bundle on $X \times X$ that have the same singularities is a holomorphic section of a bundle with no non-zero holomorphic sections, which yields the desired uniqueness of $S[e]$.

\smallskip

Assume now that $z:X\to\mathbb{CP}^{1}$ is an abelian cover with Galois group $A$ as above, that $\Delta$ is one of the divisors from Theorem \ref{nonspdiv} (written as in Equation \eqref{normdiv}), and that the characteristic $e$ is $u(\Delta)+K$. Following \cite{[Na]} we wish to produce an algebraic formula for $S[e](P,Q)$, for which we consider the expression
\begin{equation}
f_{\Delta}(P)=\prod_{\sigma \neq Id_{X}}\prod_{j=1}^{r_{\sigma}}\big(z(P)-\lambda_{\sigma,j}\big)^{\frac{\beta_{\sigma,j}^{\Delta}}{o(\sigma)}-\frac{o(\sigma)-1}{2o(\sigma)}}\cdot\sqrt{dz(P)},\qquad\mathrm{with\ }P \in X. \label{halfdif}
\end{equation}
We now prove a generalization of Proposition 4 of \cite{[Na]}, which is implicitly used also in \cite{[Ko2]} and \cite{[Ko3]}.
\begin{prop}
The expression $f_{\Delta}$ from Equation \eqref{halfdif} is a well-defined meromorphic section, with divisor $\Delta$, of the unitary line bundle of degree $g-1$ on $X$ that is associated with $e=u(\Delta)+K$. \label{mersect}
\end{prop}

\begin{proof}
Equation \eqref{simpledivs} shows that the difference between $\mathrm{div}(f_{\Delta})$ and $\Delta$ can be supported only at the poles of $z$. By the same equation, the order of $f_{\Delta}$ at any pole of $z$ on $X$ is
\begin{equation}
-1-\sum_{\sigma \neq Id_{X}}\sum_{j=1}^{r_{\sigma}}\bigg(\frac{\beta_{\sigma,j}}{o(\sigma)}-\frac{o(\sigma)-1}{2o(\sigma)}\bigg)=-1-\bigg[\frac{\deg\Delta+n}{n}-\frac{g+n-1}{n}\bigg], \label{average0}
\end{equation}
where the equality in Equation \eqref{average0} follows from the expression for $\Delta$ in Equation \eqref{normdiv} (with $h^{\Delta}=1$) and the formula for $g$ appearing in Proposition \ref{genus}. As $\deg\Delta=g-1$ by Theorem \ref{nonspdiv}, this order is $-1$, and we deduce that $\mathrm{div}(f_{\Delta})=\Delta$. Since $f_{\Delta}$ is defined in Equation \eqref{halfdif} via fractional powers of well-defined expressions on $X$, and its divisor has integral coefficients, it is a section of the unitary realization of a line bundle, which must be the asserted one since it has a section with divisor $\Delta$. This proves the proposition.
\end{proof}
We henceforth denote the exponent of $A$ by $m$, and obtain the following consequence of Proposition \ref{mersect}.
\begin{cor}
The expression $u(\Delta)+K$, with $\Delta$ as above, is torsion of order dividing $2m$ in $J(X)$. \label{torsion}
\end{cor}

\begin{proof}
Write $u(\Delta)+K$ as $u_{P}(\Delta)+K_{P}$ for some $P \in X$, and multiply it by $2m$. By Proposition \ref{candivAJ} and the proof of Proposition \ref{mersect}, the resulting expression is $u_{P}\big(\mathrm{div}(f_{\Delta}^{2m})\big)-m \cdot u_{P}\big(\mathrm{div}(dz)\big)$. But all the exponents in $f_{\Delta}/\sqrt{dz}$ from Equation \eqref{halfdif} become integral after multiplying by $2m$, so that the latter expression vanishes as the $u_{P}$-image of $\frac{f_{\Delta}^{2m}}{(dz)^{m}}\in\mathbb{C}(z)\subseteq\mathbb{C}(X)$. This proves the corollary.
\end{proof}
We remark that using direct tools, one could have obtained that $u(\Delta)+K$ is torsion of order dividing $2\mathrm{lcm}\big\{m,\frac{n}{m}\big\}$ for \emph{any} $A$-invariant divisor of degree $g-1$, an argument that holds also for non-abelian covers of $\mathbb{CP}^{1}$. The same argument applies also for Galois covers $f:X \to S$ of more general Riemann surfaces $S$, where the torsion property is established in the Prym variety $P(X/S)$ complementing the image of $J(S)$ inside $J(X)$.

\smallskip

Recall the polynomials $p_{\Delta,\chi}$ from Equation \eqref{pDeltachi} and the expressions for $\chi\Delta$ and for $N\Delta$ in Lemma \ref{actAdual} and Corollary \ref{NonAinv} respectively. A simple evaluation of divisors combines with the proof of Proposition \ref{mersect} to yield the following result.
\begin{lem}
Given $f_{\Delta}$ and $\chi\in\widehat{A}$, the product $\frac{y_{\chi}}{p_{\Delta,\chi}}f_{\Delta}$ is a scalar multiple of $f_{\chi\Delta}$, and $\frac{dz}{f_{\Delta}}$ is a scalar multiple of $f_{N\Delta}$. \label{fDeltarels}
\end{lem}

We now present our generalization of the algebraic construction of the Szeg\H{o} kernel appearing in \cite{[Na]}, \cite{[EG2]}, \cite{[Ko2]}, and \cite{[Ko3]}. This requires a certain normalization of these expressions (a point that was overlooked in these references), as will be evident in the proof below.

\begin{thm}
Assume that $e=u(\Delta)+K$ for some of the divisors $\Delta$ from Theorem \ref{nonspdiv}, and define $f_{\Delta}$ as in Equation \eqref{halfdif}. Given explicit choices of the functions $y_{\chi}$ with $\chi\in\widehat{A}$, set $f_{\chi\Delta}^{+}$ to be the product $\frac{y_{\chi}}{p_{\Delta,\chi}}f_{\Delta}$ from Lemma \ref{fDeltarels}, and define $f_{N\chi\Delta}^{+}$ to be $\frac{dz}{f_{\chi\Delta}^{+}}$. Then the expression \[F_{e}(P,Q)=\frac{1}{n}\frac{\sum_{\chi\in\widehat{A}}f_{\chi\Delta}^{+}(P)f_{N\chi\Delta}^{+}(Q)}{z(P)-z(Q)},\quad\mathrm{for\ }P\mathrm{\ and\ }Q\mathrm{\ in\ }X,\] is invariant under replacing $\Delta$ by another divisor mapping to $e$ and under scalar multiplications of the $y_{\chi}$s, and it coincides with $S[e](P,Q)$. \label{Szegoalg}
\end{thm}

\begin{proof}
The invariance from the last assertion is immediate from our normalizations, thus justifying the notation $F_{e}$. Proposition \ref{mersect} and Lemma \ref{fDeltarels} imply that $F_{e}$ is a section of the same line bundle as $S[e]$, and examining the divisors of all the terms (via Equation \eqref{simpledivs} and the proof of Proposition \ref{mersect}) yields the holomorphicity at every point $(P,Q)$ with $z(P) \neq z(Q)$ (also when one of these $z$-values is infinite). Recall the invariance of the expansion from Proposition \ref{Szegoprop} under coordinate changes, and that $z:X\to\mathbb{CP}^{1}$ is Galois. Since it is enough to establish holomorphicity and the expansion on a dense open set, it suffices to consider $P$ in the neighborhood of $\sigma Q$ for $\sigma \in A$ and $Q \in X$ with $\lambda=z(Q)$ finite and such that $\psi(\lambda)=Id_{X}$ in Proposition \ref{liftcirc} (so that $z$ itself is a good coordinate chart around $Q$). The divisor considerations make it now clear that
\begin{equation}
F_{e}(P,Q)=\frac{\sqrt{dz(P)}\sqrt{dz(Q)}}{z(P)-z(Q)}\big[a_{Q,\sigma}+O\big(z(P)-z(Q)\big)\big], \label{Feinitexp}
\end{equation}
and it remains, by the uniqueness in Proposition \ref{Szegoprop}, to prove that the constant $a_{Q,\sigma}$ from Equation \eqref{Feinitexp} is the Kronecker $\delta$-symbol $\delta_{\sigma,Id_{X}}$ for every such $Q$.

In order to do so we recall the normalization and set $f_{N\Delta}=\frac{dz}{f_{\Delta}}$, and
\[\mathrm{from\ }f_{\chi\Delta}^{+}f_{N\chi\Delta}^{+}=dz=f_{\Delta}f_{N\Delta}\mathrm{\ and\ }\tfrac{f_{\chi\Delta}^{+}}{f_{\Delta}}=\tfrac{y_{\chi}}{p_{\Delta,\chi}}\mathrm{\ we\ get\ }\tfrac{f_{N\chi\Delta}^{+}}{f_{N\Delta}}=\tfrac{p_{\Delta,\chi}}{y_{\chi}}\] (with no multiplying scalar). Substituting into the definition, we obtain that
\begin{equation}
F_{e}(P,Q)=\frac{1}{n}\bigg[\sum_{\chi\in\widehat{A}}\frac{y_{\chi}(P)}{y_{\chi}(Q)}\frac{p_{\Delta,\chi}\big(z(Q)\big)}{p_{\Delta,\chi}\big(z(P)\big)}\bigg]\frac{f_{\Delta}(P)f_{N\Delta}(Q)}{z(P)-z(Q)}, \label{Febefexp}
\end{equation}
and we expand all the functions of $P$ around $P=\sigma Q$ using $z$ as the local coordinate, recalling that $z(\sigma Q)=z(Q)$. To do this we recall from Proposition \ref{decomgen} that $A$ acts on $y_{\chi}$ via $\chi$, that $p_{\Delta,\chi}\big(z(Q)\big)\neq0$ by our assumption on $Q$, and that $f_{\Delta}$ is defined in Equation \eqref{halfdif} in terms of $z$ and $dz$ alone. It follows that
\[y_{\chi}(P)=\chi(\sigma)y_{\chi}(Q),\ \tfrac{1}{p_{\Delta,\chi}(z(P))}=\tfrac{1}{p_{\Delta,\chi}(z(Q))},\mathrm{\ and\ }f_{\Delta}(P)=\tfrac{f_{\Delta}(Q)}{\sqrt{dz(Q)}}\sqrt{dz(P)},\] all up to error terms of the sort $O\big(z(P)-z(Q)\big)$. Substituting these into Equation \eqref{Febefexp}, and recalling that $f_{\Delta}f_{N\Delta}=dz$, we indeed obtain Equation \eqref{Feinitexp}, with the coefficient $a_{Q,\sigma}$ being just $\sum_{\chi\in\widehat{A}}\frac{\chi(\sigma)}{n}$. As this sum is known to be $\delta_{\sigma,Id_{X}}$ by the classical property of characters (here we need the normalization---without it we would get an arbitrary sum $\sum_{\chi\in\widehat{A}}c_{\chi}\chi(\sigma)$, which we cannot evaluate), this completes the proof of the theorem.
\end{proof}

For an explicit expression for $S[e]=F_{e}$ in a case where $A$ is a Klein 4-group, see Equation \eqref{Szegoexp} in Section \ref{Examples} below.
\begin{rmk}
Note that $f_{\Delta}^{+}$ itself may not equal $f_{\Delta}$, since $y_{\mathbf{1}}$ can be an arbitrary non-zero scalar. Another way to overcome the normalization issue is to express $X$ as a fibered product and take $y_{\chi}$ as in Proposition \ref{fibprod} and Remark \ref{normFl}, with which the action of $\widehat{A}$ and $N$ (hence of the generalized dihedral group from Corollary \ref{dihedral}) lifts to an operation on the expressions $f_{\Delta}$ themselves. While in general there is no ambiguity in the definition of the expression $F_{e}(P,Q)$ in Theorem \ref{Szegoalg} (since we start with $\Delta$ and define all the $f_{\chi\Delta}^{+}$s and later all the $f_{N\chi\Delta}^{+}$s uniquely), when $e$ has order 2 (or 1) in $J(X)$ the action of the dihedral group from Corollary \ref{dihedral} is not free. The choices from the fibered product structure can be used to verify the well-definedness also in this case. \label{eoford2}
\end{rmk}

\smallskip

We shall also use higher order terms in the expansion of $F_{e}$ around $P=Q$. For a divisor $\Delta$ as above, two non-trivial elements $\sigma$ and $\rho$ from $A$, and two indices $1 \leq j \leq r_{\sigma}$ and $1 \leq i \leq r_{\rho}$, we define
\begin{equation}
q_{\Delta}(\sigma,j;\rho,i)=\bigg(\frac{\beta_{\sigma,j}^{\Delta}}{o(\sigma)}-\frac{o(\sigma)-1}{2o(\sigma)}\bigg)\bigg(\frac{\beta_{\rho,i}^{\Delta}}{o(\rho)}-\frac{o(\rho)-1}{2o(\rho)}\bigg). \label{qDelta}
\end{equation}
Recalling from Lemma \ref{actAdual} that the set of divisors from Theorem \ref{nonspdiv} that have the same image in $J(X)$ forms an $\widehat{A}$-orbit, we set \[q_{e}(\sigma,j;\rho,i)=\sum_{\{\Delta|u(\Delta)+K=e\}}q_{\Delta}(\sigma,j;\rho,i)=\sum_{\chi\in\widehat{A}}q_{\chi\Delta}(\sigma,j;\rho,i)\mathrm{\ if\ }u(\Delta)+K=e.\] The required expansion is evaluated in the following lemma.
\begin{lem}
If $z(Q)$ is finite and does not belong to $\bigcup_{Id_{X}\neq\sigma \in A}\{\lambda_{\sigma,j}\}_{j=1}^{r_{\sigma}}$ then the expansion of $\frac{F_{e}(P,Q)}{\sqrt{dz(P)}\sqrt{dz(Q)}}$ in $P$ around $P=Q$ is
\[\frac{1}{z(P)-z(Q)}\Bigg[1+\frac{1}{2n}\sum_{\sigma,j,\rho,i}\frac{q_{e}(\sigma,j;\rho,i)\big(z(P)-z(Q)\big)^{2}}{\big(z(Q)-\lambda_{\sigma,j}\big)\big(z(Q)-\lambda_{\rho,i}\big)}+O\Big(\big(z(P)-z(Q)\big)^{3}\Big)\Bigg].\] \label{highordFe}
\end{lem}

\begin{proof}
As the terms $f_{\chi\Delta}^{+}$ and $f_{N\chi\Delta}^{+}$ from Theorem \ref{Szegoalg} are holomorphic and non-vanishing around such a point $Q$, the normalization allows us to write \[\frac{F_{e}(P,Q)}{\sqrt{dz(P)}\sqrt{dz(Q)}}=\frac{1}{n\big(z(P)-z(Q)\big)}\sum_{\chi\in\widehat{A}}\Bigg[\frac{f_{\chi\Delta}^{+}(P)}{\sqrt{dz(P)}}\Bigg/\frac{f_{\chi\Delta}^{+}(Q)}{\sqrt{dz(Q)}}\Bigg].\] We expand $\frac{f_{\chi\Delta}^{+}(P)}{\sqrt{dz(P)}}$ around $P=Q$ up to an error of $O\big[\big(z(P)-z(Q)\big)^{3}\big]$, divide by $\frac{f_{\chi\Delta}^{+}(Q)}{\sqrt{dz(Q)}}$, sum over $\chi\in\widehat{A}$, and divide by $n$. The constant term is clearly 1, and the coefficient of the linear term simplifies to
\[\frac{1}{n}\sum_{\chi\in\widehat{A}}\frac{d}{dP}\ln\frac{f_{\chi\Delta}^{+}}{\sqrt{dz}}\bigg|_{P=Q}=\frac{1}{n}\sum_{\chi\in\widehat{A}}\frac{d}{dP}\ln\frac{f_{\chi\Delta}}{\sqrt{dz}}\bigg|_{P=Q}\] (the equality follows from the fact that $f_{\chi\Delta}$ is a scalar multiple of $f_{\chi\Delta}^{+}$, and the values are defined by our assumption on $z(Q)$). For the second derivative, observe that we are interested in the expansion of an expression of the form $\frac{\varphi(z)}{\varphi(w)}$ (with $\varphi=\frac{f_{\chi\Delta}}{\sqrt{dz}}$, $z=z(P)$, and $w=z(Q)$), and we can write $\frac{\varphi''(z)}{2\varphi(w)}$ at $z=w$ as $\frac{1}{2}\big[\frac{d^{2}\ln\varphi}{dz^{2}}+\big(\frac{d\ln\varphi}{dz})^{2}\big]_{z=w}$ (this is easily verified by evaluating the second derivative of $\ln\varphi$).

We thus consider the expression for $\frac{f_{\chi\Delta}}{\sqrt{dz}}$ given in Equation \eqref{halfdif}, using the formula for $\chi\Delta$ in Lemma \ref{actAdual}, and differentiate its logarithm. For the first derivative, we note that for fixed $\sigma$ the values of $u_{\chi,\sigma}$ are evenly distributed for $\chi\in\widehat{A}$ in the integers between 0 and $o(\sigma)-1$ (consider the values $\chi(\sigma)$ for $\chi\in\widehat{A}$), hence so are the values of $\beta_{\sigma,j}^{\chi\Delta}$ from Equation \eqref{epsDelsigjchi}. Hence $\frac{1}{n}$ times the sum over $\chi$ of this derivative equals
\begin{equation}
\frac{1}{n}\sum_{\chi\in\widehat{A}}\sum_{\sigma \neq Id_{X}}\sum_{j=1}^{r_{\sigma}}\frac{\beta_{\sigma,j}^{\chi\Delta}-\frac{o(\sigma)-1}{2}}{o(\sigma)(z(Q)-\lambda_{\sigma,j})}=\sum_{\sigma \neq Id_{X}}\sum_{j=1}^{r_{\sigma}}\frac{\frac{n}{o(\sigma)}\sum_{k=0}^{o(\sigma)-1}\big[k-\frac{o(\sigma)-1}{2}\big]}{no(\sigma)(z(Q)-\lambda_{\sigma,j})}=0, \label{vansum}
\end{equation}
as desired. The terms involving $\frac{d^{2}\ln\varphi}{dz^{2}}$ look like Equation \eqref{vansum} (with the denominator $n\big(z(Q)-\lambda_{\sigma,j}\big)$ replaced by $-2n\big(z(Q)-\lambda_{\sigma,j}\big)^{2}$), so that their sum over $\chi$ vanishes as well. On the other hand, the remaining expression $\big(\frac{d\ln\varphi}{dz})^{2}$ with $\varphi=\frac{f_{\chi\Delta}}{\sqrt{dz}}$ is the sum of products of two summands from the left hand side of Equation \eqref{vansum}, which yields $\sum_{\sigma,j,\rho,i}\frac{q_{\chi\Delta}(\sigma,j;\rho,i)}{(z(Q)-\lambda_{\sigma,j})(z(Q)-\lambda_{\rho,i})}$ by Equation \eqref{qDelta}. Summing over $\chi\in\widehat{A}$ and dividing by $2n$ thus yields the asserted term. This proves the lemma.
\end{proof}

The first consequence that we deduce from Lemma \ref{highordFe} (along the lines of \cite{[Na]}, \cite{[Ko2]}, and \cite{[Ko3]}), is the following.
\begin{cor}
For any characteristic $e$ appearing in Theorem \ref{Szegoalg}, all the derivatives $\frac{\partial\theta[e]}{\partial\zeta_{s}}\big|_{\zeta=0}$ with $1 \leq s \leq g$ vanish. \label{thetader0}
\end{cor}

\begin{proof}
Proposition \ref{Szegoprop} and Theorem \ref{Szegoalg} imply that when $Q$ satisfies the conditions of Lemma \ref{highordFe} we have the equality
\[\sum_{s=1}^{g}\frac{\partial\ln\theta[e]}{\partial\zeta_{s}}\bigg|_{\zeta=0}\frac{v_{s}(Q)}{dz(Q)}=\lim_{P \to Q}\bigg[\frac{F_{e}(P,Q)}{\sqrt{dz(P)}\sqrt{dz(Q)}}-\frac{1}{z(P)-z(Q)}\bigg].\] Note that the left hand side is well-defined, since Theorem \ref{RiemVan} and the non-specialty of $\Delta$ imply that $\theta[e](0,\tau)\neq0$. But as the expansion from Lemma \ref{highordFe} has no linear term, the right hand side vanishes. Hence so does the left hand side, which we then multiply by $\theta[e](0,\tau)$ and by $dz(Q)$ to get the vanishing of the differential $\sum_{s=1}^{g}\frac{\partial\theta[e]}{\partial\zeta_{s}}\big|_{\zeta=0}v_{s}$ at every such $Q$. This differential therefore vanishes identically, and the linear independence of the basis $\{v_{s}\}_{s=1}^{g}$ yields the vanishing of the desired coefficients. This proves the corollary.
\end{proof}

Another statement from \cite{[Na]}, \cite{[Ko2]}, and \cite{[Ko3]} takes, in our general abelian setting, the following form.
\begin{cor}
For $Q$ as in Lemma \ref{highordFe} the product $S[e](P,Q)S[-e](P,Q)$ expands, with $P$ in the neighborhood of $Q$, as
\[\frac{dz(P)dz(Q)}{\big(z(P)-z(Q)\big)^{2}}
\Bigg[1+\frac{1}{n}\sum_{\sigma,j,\rho,i}\frac{q_{e}(\sigma,j;\rho,i)\big(z(P)-z(Q)\big)^{2}}{\big(z(Q)-\lambda_{\sigma,j}\big)\big(z(Q)-\lambda_{\rho,i}\big)}+O\Big(\big(z(P)-z(Q)\big)^{3}\Big)\Bigg].\] \label{prodSzego}
\end{cor}

\begin{proof}
Theorem \ref{Szegoalg} allows us to use the expansions of $F_{e}(P,Q)$ and $F_{-e}(P,Q)$. But Corollaries \ref{NonAinv} and \ref{dihedral} imply that for every $\sigma$, $j$, $\rho$, and $i$, replacing $\Delta$ and $\chi$ by $N\Delta$ and $\overline{\chi}$ takes the summand from Equation \eqref{vansum} to its additive inverse, so that in Equation \eqref{qDelta} we get \[q_{\overline{\chi}N\Delta}(\sigma,j;\rho,i)=q_{\chi\Delta}(\sigma,j;\rho,i),\quad\mathrm{which\ implies}\quad q_{-e}(\sigma,j;\rho,i)=q_{e}(\sigma,j;\rho,i)\] by summing over $\chi$. Hence the expansions of $F_{e}(P,Q)$ and $F_{-e}(P,Q)$ in Lemma \ref{highordFe} coincide, and evaluating their product proves the corollary.
\end{proof}
In fact, one has $S[e](P,Q)=S[-e](Q,P)$ for any characteristic on any Riemann surface, which yields an alternative proof for Corollary \ref{prodSzego}.

\section{Constructing the Canonical Differential \label{CanDif}}

For a general compact Riemann surface $X$, one defines the \emph{canonical differential} $\omega$ on $X \times X$. This section considers the decomposition of $\omega$ in case $X$ is an abelian cover of $\mathbb{CP}^{1}$, and uses it to prove a formula for the Bergman projective connection appearing in the expansion of $\omega$. Also here we expand explicit expressions near the diagonal.

For any Riemann surface $X$, the differential $\omega$ is characterized as follows.
\begin{prop}
There exists a unique meromorphic $(1,1)$-form $\omega$ on $X \times X$ with the following properties: It is symmetric; It is holomorphic when $P \neq Q$ and has a singularity of the form \[\omega(P,Q)=\Big[\tfrac{1}{(z(P)-z(Q))^{2}}+\tfrac{G_{B}(z(Q))}{6dz(Q)^{2}}+O\big(z(P)-z(Q)\big)\Big]dz(P)dz(Q)\] along the diagonal $P=Q$ in every coordinate $z$ around a diagonal point; And its integral (in $P$ say, for fixed $Q$) along any cycle $a_{i}$ with $1 \leq i \leq g$ vanishes. \label{candifprop}
\end{prop}
Here $G_{B}$ is the \emph{Bergman projective connection}, and it is clear that replacing $P$ by $Q$ in the term involving it gives the same value up to $O\big(z(P)-z(Q)\big)$. The symmetry in Proposition \ref{candifprop}, namely the condition $\omega(P,Q)=\omega(Q,P)$ for every $P$ and $Q$ in $X$, implies that the $a_{i}$-integrals of $\omega(P,Q)$ in $Q$ (for fixed $P$) vanish as well. As with $E(P,Q)$ and $S[e](P,Q)$, the form of the singularity of $\omega(P,Q)$ is invariant under coordinate changes (see also Lemma \ref{Schwarz} below).

Back in the case where $z:X\to\mathbb{CP}^{1}$ is abelian with Galois group $A$, it turns out useful to decompose $\omega$ according to the action of $A$ on the two variables. The simplest analogue of this argument is the following immediate consequence of Proposition \ref{holdif}, of which we shall make use later.
\begin{cor}
The holomorphic $(1,1)$-forms on $X \times X$ are of the form \[\textstyle{\sum_{\mathbf{1}\neq\chi\in\widehat{A}}\sum_{\mathbf{1}\neq\eta\in\widehat{A}}q_{\chi}^{\eta}\big(z(P),z(Q)\big)\psi_{\overline{\chi}}(P)\psi_{\eta}(Q)},\] where $q_{\chi}^{\eta}$ is a polynomial whose degree in the first (resp. second) variable does not exceed $t_{\chi}-2$ (resp. $t_{\overline{\eta}}-2$). \label{hol11forms}
\end{cor}
In particular, the degree bound in Corollary \ref{hol11forms} implies that only characters $\chi$ and $\eta$ with $t_{\chi}\geq2$ and $t_{\overline{\eta}}\geq2$ can appear (as was the case in Proposition \ref{holdif}). The normalization with the subscript $\chi$ corresponding to $\psi_{\overline{\chi}}(P)$ will turn out more convenient below. In order to obtain a similar decomposition of the non-holomorphic differential $\omega$, we shall first need to verify the existence of certain polynomials. While this was relatively simple in the particular cases considered in \cite{[Na]}, \cite{[Ko2]}, \cite{[Ko3]} and others, in general we shall need the following lemma.
\begin{lem}
Let two degrees $d\geq1$ and $e\geq1$ be given, and take two polynomials $f_{0}$ and $f_{1}$ in one variable, of degrees $e+d$ and $e+d-1$ respectively. A necessary and sufficient condition for the existence of polynomials $f_{l}$, $2 \leq l \leq d$, each of degree $d+e-l$, such that the degree of $\sum_{l=0}^{d}f_{l}(w)(z-w)^{l}$ in $w$ will not exceed $e$, is that the leading coefficient of $f_{1}$ is $d$ times the leading coefficient of $f_{0}$. \label{polexist}
\end{lem}

\begin{proof}
Write $f_{l}(w)=\sum_{k=0}^{d+e-l}a_{l,k}(-w)^{k}$, and apply the binomial theorem to expand the sum $\sum_{l=0}^{d}f_{l}(w)(z-w)^{l}$. This yields \[\textstyle{\sum_{l=0}^{d}\sum_{k=0}^{d+e-l}\sum_{i=0}^{l}a_{l,k}\binom{l}{i}z^{i}(-w)^{k+l-i}\!=\!\sum_{i=0}^{d}\sum_{l=i}^{d}\sum_{j=l-i}^{d+e-i}a_{l,j-l+i}\binom{l}{i}z^{i}(-w)^{j}}\!,\] where we have changed the variable to $j=k+l-i$ (so that $k=j-l+i$). The degree bound in $w$ is thus satisfied if and only if \[\textstyle{\sum_{l=i}^{\min\{d,i+j\}}\binom{l}{i}a_{l,j-l+i}=0}\quad\mathrm{for\ every}\quad0 \leq i \leq d\quad\mathrm{and}\quad j>e.\] We may separate these equations according to the sum $i+j$, which we write as $d+e-h$ for some $h$, and this index satisfies $0 \leq h<d$ because we assume that $\deg f_{l} \leq d+e-l$ and $j>e$. By fixing $h$ and recalling the inequality $j=d+e-h-i>e$, we obtain $d-h$ linear equations, namely
\[\textstyle{\sum_{l=i}^{\min\{d,d+e-h\}}\binom{l}{i}a_{l,d+e-h-l}=0}\quad\mathrm{for}\quad0 \leq i<d-h,\] in the variables $\{a_{l,d+e-h-l}\}_{l=0}^{\min\{d,d+e-h\}}$ (and there are $\min\{d+1,d+e-h+1\}$ such variables).

First take $h=0$ and consider $a_{0,d+e}$ as a parameter. This yields a non-homogenous system of $d$ linear equations in the $d$ variables $\{a_{l,d+e-l}\}_{l=1}^{d}$, where we denote the matrix of coefficients (whose $il$-entry is $\binom{l}{i-1}$) by $M$, and the non-homogenous part is $-a_{0,d+e}$ times the first standard column vector $\epsilon_{1}$. If $T$ is the matrix with entries $\binom{l-1}{i-1}$ and $J$ is the one having 1s on the main diagonal and the one below it and 0 anywhere else, then the relation $\binom{l}{i-1}=\binom{l-1}{i-1}+\binom{l-1}{i-2}$ shows that $M=JT$. Both $J$ and $T$ are triangular and invertible, so that these equations have a unique solution for the $a_{l,d+e-l}$s in terms of $a_{0,d+e}$. As the linear equations arising from $h>0$ are represented by the matrix consisting of the first $d-h$ rows of $M$, we deduce that by fixing the coefficients $a_{l,d+e-l-h}$ with $0 \leq l \leq h$, the remaining coefficients, $a_{l,d+e-l-h}$ with larger $l$, are determined uniquely by the ones that we fixed.

But our assumption is that the coefficients with $l=0$ and with $l=1$ are given, and we are free to choose the rest. Hence the polynomials $f_{l}$ with $2 \leq l \leq d$ with the required property exist if and only if the value of $a_{1,d+e-1}$ coincides with the solution for that coefficient using the equations with $h=0$. But this solution is the first coordinate in the product of $M^{-1}=T^{-1}J^{-1}$ with the column vector $-a_{0,d+e}\epsilon_{1}$. Now, $J^{-1}$ is lower triangular with $(-1)^{k}$ along the $k$th lower diagonal, and multiplying the row with coordinates $(-1)^{k}$ by $T$ from the right yields the row vector $\epsilon_{1}^{t}$. Hence this row is the upper row of $T^{-1}$, so that the upper left entry of $M^{-1}$ is $d$ and the required relation for the existence of our polynomials is $a_{1,d+e-1}=-da_{0,d+e}$. Since $a_{l,k}$ was the coefficient of $(-w)^{k}$ in $f_{l}$, this is indeed the asserted relation. This proves the lemma.
\end{proof}
Note that the case $d=1$ in Lemma \ref{polexist} simply says that when $\deg f_{0}=e+1$ and $\deg f_{1}=e$, the degree of $f_{0}(w)+f_{1}(w)(z-w)$ in $w$ is at most $e$ if and only if the leading coefficients of $f_{0}$ and $f_{1}$ coincide, and we do not choose any additional polynomials.

\smallskip

We shall now construct a meromorphic $(1,1)$-form on $X \times X$ having the same singularities as $\omega$. Considering the divisor of $y_{\chi}$ in Proposition \ref{decomgen} shows, via Equation \eqref{simpledivs}, that there is a non-zero complex number $c_{\chi}$ such that
\begin{equation}
y_{\chi}^{n}=c_{\chi}\!\prod_{\sigma \neq Id_{X}}\prod_{j=1}^{r_{\sigma}}(z-\lambda_{\sigma,j})^{nu_{\chi,\sigma}/o(\sigma)}\!,\mathrm{\ and\ hence\ }\frac{d}{dz}\ln y_{\chi}=\!\sum_{\sigma \neq Id_{X}}\sum_{j=1}^{r_{\sigma}}\frac{u_{\chi,\sigma}/o(\sigma)}{z-\lambda_{\sigma,j}}. \label{ddzlnychi}
\end{equation}
We can therefore construct the following polynomial.
\begin{cor}
For $\chi\in\widehat{A}$ set $\widetilde{f}_{\chi,0}^{\chi}=y_{\chi}y_{\overline{\chi}}$, and set $\widetilde{f}_{\chi,1}^{\chi}=\widetilde{f}_{\chi,0}^{\chi}\cdot\frac{d}{dz}\ln y_{\chi}$. Then we can choose for every $2 \leq l \leq t_{\chi}$ a polynomial $\widetilde{f}_{\chi,l}^{\chi}$ of degree at most $t_{\chi}+t_{\overline{\chi}}-l$ such that by setting $\widetilde{p}_{\chi}^{\chi}(z,w)=\sum_{l=0}^{t_{\chi}}\widetilde{f}_{\chi,l}^{\chi}(w)(z-w)^{l}$, the degree of $\widetilde{p}_{\chi}^{\chi}$ in $w$ does not exceed $t_{\overline{\chi}}$. \label{pchichi}
\end{cor}
It is also clear from the definition that the degree of $\widetilde{p}_{\chi}^{\chi}$ in $z$ is at most $t_{\chi}$.

\begin{proof}
Lemma \ref{chicomp} shows that $\widetilde{f}_{\chi,0}^{\chi}=y_{\chi}y_{\overline{\chi}}$ is a polynomial of degree $t_{\chi}+t_{\overline{\chi}}$ in $z$, whose roots coincide with those of $y_{\chi}^{n}$ hence with the simple poles of $\frac{d}{dz}\ln y_{\chi}$ as expressed in Equation \eqref{ddzlnychi}. Therefore $\widetilde{f}_{\chi,1}^{\chi}$ is also a polynomial in $z$, of degree one less. Moreover, the leading coefficient of $\widetilde{f}_{\chi,1}^{\chi}$ equals that of $\widetilde{f}_{\chi,0}^{\chi}$ multiplied by $\sum_{\sigma \neq Id_{X}}\frac{r_{\sigma}u_{\chi,\sigma}}{o(\sigma)}$, and Proposition \ref{decomgen} identifies this multiplier as $t_{\chi}$. The existence of the appropriate polynomials $\widetilde{f}_{\chi,l}^{\chi}$ with $2 \leq l \leq t_{\chi}$ thus follows from Lemma \ref{polexist}, with $d=t_{\chi}$ and $e=t_{\overline{\chi}}$. This proves the corollary.
\end{proof}
The polynomial $\widetilde{p}_{\chi}^{\chi}$ from Corollary \ref{pchichi} is not unique, since Lemma \ref{polexist} allows some degree of freedom in taking the $\widetilde{f}_{\chi,l}^{\chi}$s with $l\geq2$. When $t_{\chi}=1$ it is unique though. The case with $t_{\chi}=0$ occurs only for $\chi=\mathbf{1}$, where the polynomials $\widetilde{f}_{\mathbf{1},0}^{\mathbf{1}}$ and $\widetilde{p}_{\mathbf{1}}^{\mathbf{1}}$ from Corollary \ref{pchichi} are just the constant polynomial $y_{\mathbf{1}}^{2}$, and $\widetilde{f}_{\mathbf{1},1}^{\mathbf{1}}=0$. The double indexation of $\chi$ will become clearer in Corollary \ref{omegaxi} below.

\smallskip

The shall need several consequences of the following expansion.
\begin{lem}
Let $Q \in X$ be such that $z(Q)$ is in $\mathbb{C}\setminus\bigcup_{Id_{X}\neq\sigma \in A}\{\lambda_{\sigma,j}\}_{j=1}^{r_{\sigma}}$, and take $\chi\in\widehat{A}$ and $\sigma \in A$. Then $\frac{1}{y_{\chi}(P)}$ expands around $P=\sigma Q$ as $\frac{\overline{\chi(\sigma)}}{y_{\chi}(Q)}$ times \[1-\sum_{\sigma \neq Id_{X}}\tfrac{u_{\chi,\sigma}}{o(\sigma)}\sum_{j=1}^{r_{\sigma}}\Big[\tfrac{z(P)-z(Q)}{z(Q)-\lambda_{\sigma,j}}-\tfrac{(z(P)-z(Q))^{2}}{2(z(Q)-\lambda_{\sigma,j})^{2}}\Big]+
\sum_{\sigma,j,\rho,i}\tfrac{u_{\chi,\sigma}u_{\chi,\rho}(z(P)-z(Q))^{2}/o(\sigma)o(\rho)}{2(z(Q)-\lambda_{\sigma,j})(z(Q)-\lambda_{\rho,i})},\] up to an error term of $O\big[\big(z(P)-z(Q)\big)^{3}\big]$. \label{ychiexp}
\end{lem}

\begin{proof}
The proof of Lemma \ref{highordFe} allows us to write $\frac{1}{y_{\chi}(P)}$ up to $O\big[\big(z(P)-z(Q)\big)^{3}\big]$ as $\frac{1}{y_{\chi}(\sigma Q)}$ times an expansion involving only derivatives of $\ln\frac{1}{y_{\chi}}$ at $P=\sigma Q$. The external multiplier was seen to be $\frac{\overline{\chi(\sigma)}}{y_{\chi}(Q)}$, and the remaining expressions can be evaluated using Equation \eqref{ddzlnychi}. Since the latter equation involves only a function of $z$, the rest of the expansion attains the same value at $Q$ and at $\sigma Q$. The result thus follows from a straightforward calculation. This proves the lemma.
\end{proof}

For every $\chi\in\widehat{A}$ we define, following \cite{[Na]}, \cite{[Ko2]}, and \cite{[Ko3]}, the $(1,1)$-form
\begin{equation}
\xi_{\chi}(P,Q)=\frac{\widetilde{p}_{\chi}^{\chi}\big(z(P),z(Q)\big)}{\big(z(P)-z(Q)\big)^{2}}\psi_{\overline{\chi}}(P)\psi_{\chi}(Q)=
\frac{\widetilde{p}_{\chi}^{\chi}\big(z(P),z(Q)\big)dz(P)dz(Q)}{y_{\chi}(P)y_{\overline{\chi}}(Q)\big(z(P)-z(Q)\big)^{2}} \label{xichidef}
\end{equation}
(so that in particular $\xi_{\mathbf{1}}(P,Q)=\frac{dz(P)dz(Q)}{(z(P)-z(Q))^{2}}$), and set $\xi=\frac{1}{n}\sum_{\chi\in\widehat{A}}\xi_{\chi}$. Note that also here we have an implicit normalization, since the polynomials $\widetilde{f}_{\chi,0}^{\chi}$ and $\widetilde{f}_{\chi,1}^{\chi}$ we used for defining $\widetilde{p}_{\chi}^{\chi}$ in Corollary \ref{pchichi} are based on $y_{\chi}y_{\overline{\chi}}$ rather than the polynomial from Lemma \ref{chicomp}. This normalization makes $\xi_{\chi}$ invariant under rescalings of $y_{\chi}$ and $y_{\overline{\chi}}$, and is also used in the following result.

\begin{prop}
The difference between the $(1,1)$-form $\xi$ and the canonical form $\omega$ is holomorphic on $X \times X$. \label{xising}
\end{prop}

\begin{proof}
By Proposition \ref{candifprop} we need to show that $\xi$ is holomorphic outside the diagonal in $X \times X$, and has the desired expansion near the diagonal. It is clear that if $z(P) \neq z(Q)$ then all the $\xi_{\chi}$s are holomorphic at $(P,Q)$ (when one of $z(P)$ or $z(Q)$ is infinite, this follows from substituting the divisors from Proposition \ref{decomgen} and Equation \eqref{simpledivs} and the bounds on the degrees of $\widetilde{p}_{\chi}^{\chi}$ in the two variables), hence so is $\xi$. As in the proof of Theorem \ref{Szegoalg}, it suffices to take $Q$ like in Lemma \ref{highordFe} and $P$ in the neighborhood of $\sigma Q$ for some $\sigma \in A$, and to show that $\xi(P,Q)$ expands as $\big(\frac{\delta_{\sigma,Id_{X}}}{(z(P)-z(Q))^{2}}+O(1)\big)dz(P)dz(Q)$. But Corollary \ref{pchichi} evaluates the numerator $\widetilde{p}_{\chi}^{\chi}\big(z(P),z(Q)\big)$ of $\frac{\xi_{\chi}(P,Q)}{dz(P)dz(Q)}$ for $P$ near $\sigma Q$ as
\begin{equation}
\textstyle{y_{\chi}(Q)y_{\overline{\chi}}(Q)\big[1+\sum_{\sigma \neq Id_{X}}\sum_{j=1}^{r_{\sigma}}\frac{u_{\chi,\sigma}/o(\sigma)}{z(Q)-\lambda_{\sigma,j}}\big(z(P)-z(Q)\big)\big]+O\big[\big(z(P)-z(Q)\big)^{2}\big]}. \label{numxiexp}
\end{equation}
Multiplying the expression from Equation \eqref{numxiexp} by the expansion of $\frac{1}{y_{\chi}(P)}$ from Lemma \ref{ychiexp} and dividing by $y_{\overline{\chi}}(Q)\big(z(P)-z(Q)\big)^{2}$ shows that \[\textstyle{\frac{\xi_{\chi}(P,Q)}{dz(P)dz(Q)}=\frac{\overline{\chi(\sigma)}+O[(z(P)-z(Q))^{2}]}{(z(P)-z(Q))^{2}},\mathrm{\ and\ hence\ }\frac{\xi(P,Q)}{dz(P)dz(Q)}=\frac{\delta_{\sigma,Id_{X}}}{(z(P)-z(Q))^{2}}+O(1)}\]
as desired (sums of characters again). This proves the proposition.
\end{proof}

We therefore deduce the following corollary.
\begin{cor}
For $\chi$ and $\eta$ in $\widehat{A}$, set $f_{\chi,l}^{\eta}$ for $l\in\{0,1\}$ to be $\widetilde{f}_{\chi,l}^{\chi}$ when $\chi=\eta$ and just 0 otherwise. Then there are polynomials $f_{\chi,l}^{\eta}$ with $l\geq2$, whose degree is bounded by $t_{\chi}+t_{\overline{\eta}}-l-2+2\delta_{\chi,\eta}$, such that by setting \[\textstyle{p_{\chi}^{\eta}(z,w)=\sum_{l=2-2\delta_{\chi,\eta}}^{t_{\chi}}f_{\chi,l}^{\eta}(w)(z-w)^{l-2+2\delta_{\chi,\eta}}},\] the total degree of $p_{\chi,\eta}(z,w)$ in $w$ is bounded by $t_{\overline{\eta}}-2+2\delta_{\chi,\eta}$, and we get \[\omega(P,Q)=\sum_{\chi\in\widehat{A}}\frac{p_{\chi}^{\chi}\big(z(P),z(Q)\big)dz(P)dz(Q)}{ny_{\chi}(P)y_{\overline{\chi}}(Q)\big(z(P)-z(Q)\big)^{2}}+
\sum_{\chi\neq\eta}\frac{p_{\chi}^{\eta}\big(z(P),z(Q)\big)dz(P)dz(Q)}{ny_{\chi}(P)y_{\overline{\eta}}(Q)}.\] \label{omegaxi}
\end{cor}

\begin{proof}
Corollary \ref{hol11forms} allows us to write the difference $\omega(P,Q)-\xi(P,Q)$ from Proposition \ref{xising} as $\frac{1}{n}\sum_{\chi,\eta}\widehat{p}_{\chi}^{\eta}\big(z(P),z(Q)\big)\psi_{\overline{\chi}}(P)\psi_{\eta}(Q)$, where $\widehat{p}_{\chi}^{\eta}$ is a polynomial with the degree bounds from the corollary (the summands with trivial $\chi$ or $\eta$ are with $\widehat{p}_{\chi}^{\eta}=0$). Set $p_{\chi}^{\eta}(z,w)$ to be $\widetilde{p}_{\chi}^{\chi}(z,w)+\widehat{p}_{\chi}^{\chi}(z,w)(z-w)^{2}$ with $\widetilde{p}_{\chi}^{\chi}$ from Corollary \ref{pchichi} when $\chi=\eta$, and just $\widehat{p}_{\chi}^{\eta}(z,w)$ otherwise. Our expression for $\omega-\xi$ then combines with the definition of $\xi$ in Equation \eqref{xichidef} to yield the asserted formula for $\omega$. Moreover, an argument similar to Lemma \ref{polexist} allows us to write $\widehat{p}_{\chi}^{\eta}(z,w)$ as $\sum_{l=2}^{t_{\chi}}\widehat{f}_{\chi,l}^{\eta}(w)(z-w)^{l-2}$, with the degree of $\widehat{f}_{\chi,l}^{\eta}$ not exceeding $t_{\chi}+t_{\overline{\eta}}-l-2$ (the reason for this choice of index will soon become apparent), and with $\widehat{f}_{\chi,l}^{\eta}=0$ for $l\geq2$ wherever $\chi$ or $\eta$ equal $\mathbf{1}$. We set $\widehat{f}_{\chi,l}^{\eta}=0$ for any $\chi$ and $\eta$ wherever $l\leq1$, and define $f_{\chi,l}^{\eta}$ to be $\widetilde{f}_{\chi,l}^{\eta}+\widehat{f}_{\chi,l}^{\eta}$ if $\eta=\chi$ and just $\widehat{f}_{\chi,l}^{\eta}$ in case $\chi\neq\eta$. Then the expression for $p_{\chi}^{\eta}$ in terms of the $f_{\chi,l}^{\eta}$s is indeed the asserted one both when $\chi=\eta$ and when $\chi\neq\eta$. This proves the corollary.
\end{proof}

Note that only elements with $\chi=\eta$ contribute to the singular part of $\omega$, since the diagonal action of $A$ on $X \times X$ must preserve this singularity. A more explicit expression for $\omega$ in case $X$ is a genus 3 fibered product of two elliptic curves is given in Equation \eqref{difexp} in Section \ref{Examples} below.

The choices of indices in Corollary \ref{omegaxi} decomposes $\omega$ as
\begin{equation}
\omega=\sum_{\eta\in\widehat{A}}\frac{\omega_{\eta}}{n}\mathrm{\ with\ }\omega_{\eta}(P,Q)=\sum_{\chi\in\widehat{A}}\sum_{l=0}^{t_{\chi}}\frac{f_{\chi,l}^{\eta}\big(z(Q)\big)\big(z(P)-z(Q)\big)^{l-2}dz(P)dz(Q)}{y_{\chi}(P)y_{\overline{\eta}}(Q)}. \label{omegaPQexp}
\end{equation}
We shall need the following property of the components $\omega_{\eta}$ from Equation \eqref{omegaPQexp}.
\begin{lem}
For any $\eta\in\widehat{A}$ and any $P$ and $Q$ in $X$ the $(1,1)$-form $\omega_{\eta}(P,Q)$ from Equation \eqref{omegaPQexp} equals $\sum_{\sigma \in A}\overline{\eta(\sigma)}\omega(P,\sigma Q)$. For every such $\eta$, every $Q \in X$, and every $1 \leq i \leq g$ the integral $\int_{a_{i}}\omega_{\eta}(P,Q)$ with respect to $P$ vanishes. \label{intomegaeta}
\end{lem}

\begin{proof}
Checking the action of $A$ on all the terms appearing in the definition of $\omega_{\eta}$ in Equation \eqref{omegaPQexp} shows that the equality $\omega_{\eta}(P,\sigma Q)=\eta(\sigma)\omega_{\eta}(P,Q)$ holds for every $\sigma \in A$ and $P$ and $Q$ in $X$. By expanding $\omega(P,\sigma Q)$ as in Equation \eqref{omegaPQexp} we therefore get, for any $\chi\in\widehat{A}$, that
\[\sum_{\sigma \in A}\overline{\chi(\sigma)}\omega(P,\sigma Q)=\sum_{\sigma \in A}\sum_{\eta\in\widehat{A}}\tfrac{\overline{\chi(\sigma)}}{n}\omega_{\eta}(P,\sigma Q)=
\sum_{\eta\in\widehat{A}}\bigg[\sum_{\sigma \in A}\tfrac{\overline{\chi(\sigma)}\eta(\sigma)}{n}\bigg]\omega_{\eta}(P,Q),\] and the sum over $\sigma$ is known to be $\delta_{\chi,\eta}$. This establishes the first assertion, and the second one follows from the vanishing of the integral of each of the summands in that expression along each $a_{i}$. This proves the lemma.
\end{proof}

\smallskip

Another consequence of Lemma \ref{ychiexp} that we shall need is the following expression for the Bergman projective connection $G_{B}$ from Proposition \ref{candifprop}.
\begin{cor}
If $z(Q)\in\mathbb{C}\setminus\bigcup_{Id_{X}\neq\sigma \in A}\{\lambda_{\sigma,j}\}_{j=1}^{r_{\sigma}}$ then $G_{B}\big(z(Q)\big)/6$ equals \[\Bigg[\sum_{\chi\in\widehat{A}}\sum_{\eta\in\widehat{A}}\frac{f_{\chi,2}^{\eta}\big(z(Q)\big)}{ny_{\chi}(Q)y_{\overline{\eta}}(Q)}-\frac{1}{2}\sum_{\sigma \in A}\sum_{\rho \in A}\sum_{j=1}^{r_{\sigma}}\sum_{i=1}^{r_{\rho}}\frac{\gamma_{\sigma,\rho}-\delta_{\sigma,\rho}\delta_{i,j}\frac{o(\sigma)-1}{2o(\sigma)}}{\big(z(Q)-\lambda_{\sigma,j}\big)\big(z(Q)-\lambda_{\rho,i}\big)}\Bigg]\!dz(Q)^{2},\] where $\gamma_{\sigma,\rho}$ denotes the sum $\frac{1}{n}\sum_{\chi\in\widehat{A}}\frac{u_{\chi,\sigma}u_{\chi,\rho}}{o(\sigma)o(\rho)}$. \label{Bergman}
\end{cor}

\begin{proof}
We express $\omega(P,Q)$ as in Equation \eqref{omegaPQexp}, and then $\frac{\omega_{\eta}(P,Q)}{dz(P)dz(Q)}$ equals \[\frac{f_{\eta,0}^{\eta}\big(z(Q)\big)}{y_{\eta}(P)y_{\overline{\eta}}(Q)\big(z(P)-z(Q)\big)^{2}}+\frac{f_{\eta,1}^{\eta}\big(z(Q)\big)}{y_{\eta}(P)y_{\overline{\eta}}(Q)\big(z(P)-z(Q)\big)}+
\sum_{\chi\in\widehat{A}}\frac{f_{\chi,2}^{\eta}\big(z(Q)\big)}{y_{\chi}(P)y_{\overline{\eta}}(Q)},\] plus an error term of $O\big((z(P)-z(Q)\big)$ that vanishes at the limit $P \to Q$ (recall that $f_{\chi,0}^{\eta}=f_{\chi,1}^{\eta}=0$ for $\chi\neq\eta$ in Corollary \ref{omegaxi}). Taking this limit in the sum over $\chi$, and then summing over $\eta$ and dividing by $n$ (as in Equation \eqref{omegaPQexp}) yields the first asserted term. Now, Corollary \ref{pchichi} shows that the two remaining terms yield, as in Equation \eqref{numxiexp} with $\chi=\eta$, the expression
\[\frac{y_{\eta}(Q)}{y_{\eta}(P)\big(z(P)-z(Q)\big)^{2}}\cdot\textstyle{\big[1+\sum_{\sigma \neq Id_{X}}\sum_{j=1}^{r_{\sigma}}\frac{u_{\eta,\sigma}/o(\sigma)}{z(Q)-\lambda_{\sigma,j}}\big(z(P)-z(Q)\big)\big]}\] (the expression $y_{\overline{\eta}}(Q)$ cancels). We may substituting the expansion of $\frac{1}{y_{\chi}(P)}$ from Lemma \ref{ychiexp}, with $\chi=\eta$ and $\sigma=Id_{X}$, with the error term $O\big[\big(z(P)-z(Q)\big)^{3}\big]$ vanishing at the limit $P \to Q$ also when divided by $\big(z(P)-z(Q)\big)^{2}$. This replaces the first multiplier by $1\Big/\big(z(P)-z(Q)\big)^{2}$ times \[1-\sum_{\sigma \neq Id_{X}}\tfrac{u_{\eta,\sigma}}{o(\sigma)}\sum_{j=1}^{r_{\sigma}}\Big[\tfrac{z(P)-z(Q)}{z(Q)-\lambda_{\sigma,j}}-\tfrac{(z(P)-z(Q))^{2}}{2(z(Q)-\lambda_{\sigma,j})^{2}}\Big]+
\sum_{\sigma,j,\rho,i}\tfrac{u_{\eta,\sigma}u_{\eta,\rho}(z(P)-z(Q))^{2}/o(\sigma)o(\rho)}{2(z(Q)-\lambda_{\sigma,j})(z(Q)-\lambda_{\rho,i})},\] and multiplying by the second one we get \[\textstyle{\frac{1}{(z(P)-z(Q))^{2}}+\sum_{\sigma \neq Id_{X}}\sum_{j=1}^{r_{\sigma}}\frac{u_{\eta,\sigma}/o(\sigma)}{2(z(Q)-\lambda_{\sigma,j})^{2}}-\sum_{\sigma,j,\rho,i}\frac{u_{\eta,\sigma}u_{\eta,\rho}/o(\sigma)o(\rho)}{2(z(Q)-\lambda_{\sigma,j})(z(Q)-\lambda_{\rho,i})}},\] Summing over $\eta$ and dividing by $n$, we indeed get the singularity from Proposition \ref{candifprop}, and the asserted terms involving $\gamma_{\sigma,\rho}$. In the remaining terms we only have denominators of $\big(z(Q)-\lambda_{\sigma,j}\big)^{2}$ (represented in the required expression by the product of the Kronecker $\delta$-symbols), and the respective numerator is evaluated in the rightmost equality in Equation \eqref{vansum} to be the desired one. This completes the proof of the corollary.
\end{proof}

\section{Expanding Around Branch Points \label{BrPts}}

The expressions from Corollaries \ref{prodSzego} and \ref{Bergman} are related via Corollary 2.12 of \cite{[Fa]}, given in Proposition \ref{Fayident} below. The goal of this section is to consider this relation near a branch point, and deduce a useful equality. This is done by changing the coordinate to the natural one near a branch point. We also relate the coefficients from these corollaries to generalized Dedekind sums.

We now assume that $z(Q)$ is close to $\lambda_{\sigma,j}$ for some $\sigma$ and $j$, so that the local coordinate $t$ in that neighborhood satisfies $z-\lambda_{\sigma,j}=t^{o(\sigma)}$. Recall the following result, from, e.g., the discussion around Equation (27) of \cite{[Fa]}, about coordinate changes (this reference considers the case of $G_{B}$ from Proposition \ref{candifprop}, but the idea works in general), in which the \emph{Schwarzian derivative} $\mathcal{S}\{z,t\}$ of a coordinate $z$ with respect to another coordinate $t$ is $\frac{z'''(t)}{z'(t)}-\frac{3}{2}\cdot\frac{z''(t)^{2}}{z'(t)^{2}}$.
\begin{lem}
The expression $\frac{dz(P)dz(Q)}{(z(P)-z(Q))^{2}}$ takes, in the coordinate $t$, the form \[\Big[1+\tfrac{\mathcal{S}\{z,t\}(t(Q))}{6}\big(t(P)-t(Q)\big)^{2}+O\Big(\big(t(P)-t(Q)\big)^{3}\Big)\Big]\tfrac{dt(P)dt(Q)}{(t(P)-t(Q))^{2}}\] for $P$ near $Q$. When $z-\lambda_{\sigma,j}=t^{o(\sigma)}$ we have $\mathcal{S}\{z,t\}=-\frac{(o(\sigma)-1)(o(\sigma)+1)}{2t^{2}}$. \label{Schwarz}
\end{lem}
The second assertion in Lemma \ref{Schwarz} is easy, and we consider the expressions from Corollaries \ref{prodSzego} and \ref{Bergman} and from Equation \eqref{omegaPQexp} in the coordinate $t$.
\begin{lem}
Let $e$ be one of the characteristics appearing in Theorem \ref{Szegoalg}, and assume that $Q$ and $P$ lie in a neighborhood of a pre-image of $\lambda_{\sigma,j}$ described by the coordinate $t$ from above. Then $S[e](P,Q)S[-e](P,Q)$ expands as \[1+\frac{2o(\sigma)^{2}}{n}\sum_{(\rho,i)\neq(\sigma,j)}\frac{q_{e}(\sigma,j;\rho,i)}{\lambda_{\sigma,j}-\lambda_{\rho,i}}t(Q)^{o(\sigma)-2}\big(t(P)-t(Q)\big)^{2}+O\Big(\big(t(P)-t(Q)\big)^{3}\Big)\] times $\frac{dt(P)dt(Q)}{(t(P)-t(Q))^{2}}$, plus an error term of $O\big(t(Q)^{2o(\sigma)-2}\big)$. \label{expSzego}
\end{lem}

\begin{proof}
The expression $\frac{1}{z(Q)-\lambda_{\rho,i}}$ from Corollary \ref{prodSzego} equals $t^{-o(\sigma)}$ in case $\rho=\sigma$ and $i=j$ and $\frac{1+O(t^{o(\sigma)})}{\lambda_{\sigma,j}-\lambda_{\rho,i}}$ otherwise, where $t=t(Q)$. We separate the fourfold sum from Corollary \ref{prodSzego} into three terms, where $S_{0}$ consists of the summands in which neither of the branching values involved coincides with $\lambda_{\sigma,j}$, $S_{1}$ is the sum over the summands where one branching value is $\lambda_{\sigma,j}$ and the other one is not (by symmetry we may take only those in which the first one is $\lambda_{\sigma,j}$ and multiply by 2), and $S_{2}$ is the unique summand where both branching values are $\lambda_{\sigma,j}$. By writing $t(P)=t+\Delta t$, applying Lemma \ref{Schwarz} for the multiplier $\frac{dz(P)dz(Q)}{(z(P)-z(Q))^{2}}$, and noting that $\big(z(P)-z(Q)\big)^{2}$ and the error term $O\big[\big(z(P)-z(Q)\big)^{3}\big]$ combine to $\big(o(\sigma)\big)^{2}t^{2o(\sigma)-2}(\Delta t)^{2}+O\big((\Delta t)^{3}\big)$, the expression from Corollary \ref{prodSzego} becomes
\begin{equation}
1+\big[\big(o(\sigma)\big)^{2}t^{2o(\sigma)-2}(S_{0}+S_{1}+S_{2})-\tfrac{(o(\sigma)-1)(o(\sigma)+1)}{12t^{2}}\big](\Delta t)^{2}+O\big((\Delta t)^{3}\big) \label{sepsum}
\end{equation}
times $dt(P)dt(Q)/(\Delta t)^{2}$.

Now, $S_{1}$ is the sum over $(\rho,i)\neq(\sigma,j)$ of terms of the form $\frac{2q_{e}(\sigma,j;\rho,i)+O(t^{o(\sigma)})}{n(\lambda_{\sigma,j}-\lambda_{\rho,i})t^{o(\sigma)}}$, which produces the desired coefficient in front of $(\Delta t)^{2}$. Moreover, all the error terms go into the error term $O\big(t(Q)^{2o(\sigma)-2}\big)$, together with $S_{0}$ since it is bounded as $t\to0$. In addition, by combining the definition, Equation \eqref{qDelta}, and the argument producing Equation \eqref{vansum}, we find that the term involving $S_{2}$ equals
\[\textstyle{\frac{o(\sigma)^{2}q_{e}(\sigma,j;\sigma,j)}{nt^{2}}=\frac{o(\sigma)}{t^{2}}\sum_{u=0}^{o(\sigma)-1}\big(\frac{u}{o(\sigma)}-\frac{o(\sigma)-1}{2o(\sigma)}\big)^{2}=
\frac{1}{t^{2}}\sum_{u=0}^{o(\sigma)-1}u\big(\frac{u}{o(\sigma)}-\frac{o(\sigma)-1}{2o(\sigma)}\big)}.\] But the two resulting sums give $\frac{(o(\sigma)-1)(2o(\sigma)-1)}{6t^{2}}$ and $\frac{(o(\sigma)-1)^{2}}{4t^{2}}$, the difference of which cancels with remaining term $\frac{(o(\sigma)-1)(o(\sigma)+1)}{12t^{2}}$ from Equation \eqref{sepsum}. This completes the proof of the lemma.
\end{proof}

\begin{lem}
Let $Q$ and $t=t(Q)$ be as in Lemma \ref{expSzego}, and take some $P \in X$ with $z=z(P)$. Then the differential $\omega_{\eta}(P,Q)$ from Equation \eqref{omegaPQexp} expands as \[(1-\delta_{\eta(\sigma),1})\tfrac{o(\sigma)}{c_{\eta}}\bigg[\tfrac{u_{\eta,\sigma}}{o(\sigma)}\tfrac{(f_{\eta,0}^{\eta})'(\lambda_{\sigma,j})}{y_{\eta}(P)(z-\lambda_{\sigma,j})}+
\sum_{\chi\in\widehat{A}}\sum_{l=2}^{t_{\chi}}\tfrac{f_{\chi,l}^{\eta}(\lambda_{\sigma,j})(z-\lambda_{\sigma,j})^{l-2}}{y_{\chi}(P)}\bigg]t^{o(\sigma)-1-u_{\overline{\eta},\sigma}}dzdt\]
for some non-zero constant $c_{\eta}$, up to an error term of $O(t^{o(\sigma)-1})$. In addition, $G_{B}(t)=G_{B}\big(t(Q)\big)$ expands, up to an error term of $O(t^{o(\sigma)})$, as $6$ times \[\bigg[\sum_{\chi\in\widehat{A}}\sum_{\eta\in\widehat{A}}\frac{o(\sigma)^{2}}{nc_{\overline{\chi}}c_{\eta}}f_{\chi,2}^{\eta}(\lambda_{\sigma,j})t^{2o(\sigma)-2-u_{\chi,\sigma}-u_{\overline{\eta},\sigma}}-
\sum_{(\rho,i)\neq(\sigma,j)}\frac{o(\sigma)^{2}\gamma_{\sigma,\rho}}{\lambda_{\sigma,j}-\lambda_{\rho,i}}t^{o(\sigma)-2}\bigg]dt^{2}.\] \label{expsint}
\end{lem}

\begin{proof}
Write $y_{\overline{\eta}}(Q)$ as $t^{u_{\overline{\eta},\sigma}}$ times a local function $\varphi_{\eta}=\varphi_{\eta}\big(z(Q)\big)$ (involving fractional powers of $z(Q)-\lambda_{\rho,i}$ with $\lambda_{\rho,i}\neq\lambda_{\sigma,j}$), with $c_{\eta}=\varphi_{\eta}(\lambda_{\sigma,j})\neq0$. Substituting $z(Q)=\lambda_{\sigma,j}+t^{o(\sigma)}$ and $dz(Q)=o(\sigma)t^{o(\sigma)-1}dt$ in the remaining parts of the formula for $\omega_{\eta}(P,Q)$ in Equation \eqref{omegaPQexp} (including $\varphi_{\eta}$) shows that
\[\omega_{\eta}(P,Q)=\frac{o(\sigma)}{c_{\eta}}\Bigg[\sum_{\chi\in\widehat{A}}\sum_{l=0}^{t_{\chi}}\frac{f_{\chi,l}^{\eta}(\lambda_{\sigma,j})(z-\lambda_{\sigma,j})^{l-2}}{y_{\chi}(P)}+O(t^{o(\sigma)})\Bigg]
t^{o(\sigma)-1-u_{\overline{\eta},\sigma}}dzdt,\] with $z$ and $dz$ being those of $P$. If $\eta(\sigma)=1$ then $1-\delta_{\eta(\sigma),1}=0$ and the whole expression is indeed $O(t^{o(\sigma)-1})$, so we consider characters $\eta$ with $\eta(\sigma)\neq1$. Now, the polynomials $f_{\chi,l}^{\eta}$ from Corollary \ref{omegaxi} vanish for $l\leq1$ when $\chi\neq\eta$, and $f_{\eta,0}^{\eta}$ vanishes at $\lambda_{\sigma,j}$ when $\eta(\sigma)\neq1$ (see Lemma \ref{chicomp}). In addition, $f_{\eta,1}^{\eta}(w)$ takes the same value as $\frac{o(\sigma)}{c_{\eta}}\cdot\frac{f_{\eta,0}^{\eta}(w)}{w-\lambda_{\sigma,j}}$ at $w=\lambda_{\sigma,j}$, and this value is $\frac{u_{\eta,\sigma}}{o(\sigma)}\big(f_{\eta,0}^{\eta}\big)'(\lambda_{\sigma,j})$. This proves the first assertion.

For the second one, recall from Lemma \ref{Schwarz} that $G_{B}(t)/6$ is obtained by changing the variable $z(Q)$ in the expression for $G_{B}\big(z(Q)\big)/6$ from Corollary \ref{Bergman} to $t$, but also subtracting $\frac{(o(\sigma)-1)(o(\sigma)+1)}{12t^{2}}dt^{2}$. Following the proof of Lemma \ref{expSzego}, the sum over the characters in Corollary \ref{Bergman} yields the first asserted sum, up to error terms of the form $O(t^{3o(\sigma)-2-u_{\chi,\sigma}-u_{\overline{\eta},\sigma}})$, all of which are $O(t^{o(\sigma)})$ by the bound on the $u$-coefficients from Equation \eqref{uchisigma}. The fourfold sum in that corollary splits as in Equation \eqref{sepsum}, with the symmetry in the corresponding $S_{1}$ (hence a factor of 2). The explicit terms in $S_{1}$ again yield the second desired expression, while the error terms in $S_{1}$ and the full expression $S_{0}$ are $O(t^{2o(\sigma)-2})$, which is $O(t^{o(\sigma)})$ since $o(\sigma)\geq2$ for $\sigma \neq Id_{X}$. The term corresponding to $S_{2}$ here is $\frac{o(\sigma)^{2}\gamma_{\sigma,\sigma}-o(\sigma)(o(\sigma)-1)/2}{2t^{2}}$ by definition, and since \[\textstyle{o(\sigma)^{2}\gamma_{\sigma,\sigma}=\frac{1}{n}\sum_{\chi\in\widehat{A}}u_{\chi,\sigma}^{2}=\frac{1}{o(\sigma)}\sum_{u=0}^{o(\sigma)-1}u=\frac{(o(\sigma)-1)(2o(\sigma)-1)}{6}},\] it cancels with the modification arising from $\mathcal{S}\{z,t\}$ also in this case. This completes the proof of the lemma.
\end{proof}

\smallskip

We now turn to the coefficients $q_{e}(\sigma,j;\rho,i)$ and $\gamma_{\sigma,\rho}$. For an integer $d$ and a class $h\in\mathbb{Z}/d\mathbb{Z}$ that is co-prime to $d$ we define the function $\phi_{h+d\mathbb{Z}}:\mathbb{Z}\to\mathbb{C}$ by
\begin{equation}
\phi_{h+d\mathbb{Z}}(\xi)=\sum_{0 \neq k\in\mathbb{Z}/d\mathbb{Z}}\frac{\mathbf{e}(k\xi/d)}{\big(1-\mathbf{e}(kh/d)\big)\big(1-\mathbf{e}(-k/d)\big)}. \label{phihd}
\end{equation}
The value $\phi_{h+d\mathbb{Z}}(0)$ is essentially a Dedekind sum (see Lemma \ref{Dedekind} below), so that $\phi_{h+d\mathbb{Z}}(\xi)$ is a generalized Dedekind sum, which depends only on the image of $\xi$ in $\mathbb{Z}/d\mathbb{Z}$. In addition, given two elements $\sigma$ and $\rho$ of $A$, generating the cyclic subgroups $\langle\sigma\rangle$ and $\langle\rho\rangle$ respectively,
\begin{equation}
\mathrm{set}\quad d=\big|\langle\sigma\rangle\cap\langle\rho\rangle\big|,\quad\mathrm{and\ define}\quad h\in\mathbb{Z}/d\mathbb{Z}\quad\mathrm{by}\quad\rho^{o(\rho)/d}=(\sigma^{o(\sigma)/d})^{h}. \label{dhdef}
\end{equation}
Note that $d$ from Equation \eqref{dhdef} is a divisor of both $o(\sigma)$ and $o(\rho)$, and since either $\sigma^{o(\sigma)/d}$ or $\rho^{o(\rho)/d}$ generate $\langle\sigma\rangle\cap\langle\rho\rangle$, the value of $h$ in that equation exists and is unique, and it is co-prime to $d$. We now get the following relation.
\begin{prop}
Take a divisor $\Delta$ as in Theorem \ref{nonspdiv}, written as in Equation \eqref{normdiv}, and set $e=u(\Delta)+K$. Then for $\sigma$ and $\rho$ in $A$ let $d$ and $h$ be as in Equation \eqref{dhdef}, and given $1 \leq j \leq r_{\sigma}$ and $1 \leq i \leq r_{\rho}$ we have
\[q_{e}(\sigma,j;\rho,i)\!=\!\frac{n\phi_{h+d\mathbb{Z}}(\beta_{\rho,i}^{\Delta}-h\beta_{\sigma,j}^{\Delta})}{o(\sigma)o(\rho)}\mathrm{\ and\ }\gamma_{\sigma,\rho}\!=\!\frac{\phi_{h+d\mathbb{Z}}(0)}{o(\sigma)o(\rho)}+\frac{(o(\sigma)-1)(o(\rho)-1)}{4o(\sigma)o(\rho)}.\] \label{expval}
\end{prop}

\begin{proof}
For every $\chi$ the expression $q_{\chi\Delta}(\sigma,j;\rho,i)$ from Equation \eqref{qDelta} is the product of two terms, which we write as $\alpha_{\sigma,j}^{\Delta\chi}$ and $\alpha_{\rho,i}^{\Delta\chi}$, and Equation \eqref{vansum} shows that $\sum_{\chi\in\widehat{A}}\alpha_{\sigma,j}^{\chi\Delta}=\sum_{\chi\in\widehat{A}}\alpha_{\rho,i}^{\chi\Delta}=0$. The sum $\sum_{\chi\in\widehat{A}}\alpha_{\sigma,j}^{\chi\Delta}\alpha_{\rho,i}^{\chi\Delta}$ thus equals \[\sum_{\chi\in\widehat{A}}(\alpha_{\sigma,j}^{\chi\Delta}-\alpha_{\sigma,j}^{\Delta})\alpha_{\rho,i}^{\chi\Delta}+\alpha_{\sigma,j}^{\Delta}\cdot0=\sum_{\chi\in\widehat{A}}
(\alpha_{\sigma,j}^{\chi\Delta}-\alpha_{\sigma,j}^{\Delta})(\alpha_{\rho,i}^{\chi\Delta}-\alpha_{\rho,i}^{\Delta})+\alpha_{\rho,i}^{\Delta}\bigg[0-\sum_{\chi\in\widehat{A}}\alpha_{\sigma,j}^{\Delta}\bigg].\] Evaluating the difference $\alpha_{\sigma,j}^{\chi\Delta}-\alpha_{\sigma,j}^{\Delta}=\frac{\beta_{\sigma,j}^{\chi\Delta}-\beta_{\sigma,j}^{\Delta}}{o(\sigma)}$ (and similarly for $\rho$ and $i$) via Equation \eqref{epsDelsigjchi}, and reconstructing $q_{\Delta}(\sigma,j;\rho,i)$ again, we find that
\begin{equation}
q_{e}(\sigma,j;\rho,i)=\textstyle{\sum_{\chi\in\widehat{A}}\big(\frac{u_{\chi,\sigma}}{o(\sigma)}-\varepsilon_{\Delta,\sigma,j,\chi}\big)\big(\frac{u_{\chi,\rho}}{o(\rho)}-\varepsilon_{\Delta,\rho,i,\chi}\big)
-nq_{\Delta}(\sigma,j;\rho,i)}. \label{qeeval}
\end{equation}
Note that for $\beta_{\sigma,j}=\beta_{\rho,i}=0$ the two $\varepsilon$ terms vanish for all $\chi$, so that we can evaluate $n\gamma_{\sigma,\rho}$ by the same calculation.

Now, Equation \eqref{uchisigma} and the definition of $d$ and $h$ in Equation \eqref{dhdef} yield
\[\mathbf{e}\big(\tfrac{u_{\chi,\rho}}{d}\big)=\chi(\rho^{o(\rho)/d})=\chi(\sigma^{o(\sigma)/d})^{h}=\mathbf{e}\big(\tfrac{hu_{\chi,\sigma}}{d}\big)\mathrm{\ and\ hence\ }u_{\chi,\rho} \equiv hu_{\chi,\sigma}(\mathrm{mod\ }d)\] for every $\chi\in\widehat{A}$. The number of pairs $0 \leq u<o(\sigma)$ and $0 \leq v<o(\rho)$ satisfying such a congruence is $\frac{o(\sigma)o(\rho)}{d}$, like the order of the group $\langle\sigma,\rho\rangle$, which implies that every such pair is $u=u_{\chi,\sigma}$ and $v=u_{\chi,\rho}$ for $\frac{nd}{o(\sigma)o(\rho)}$ many characters $\chi$. As the sum $\frac{1}{d}\sum_{k\in\mathbb{Z}/d\mathbb{Z}}\mathbf{e}\big(\frac{k(hu-v)}{d}\big)$ equals 1 in case $u$ and $v$ satisfy our congruence and 0 otherwise, the sum over $\chi$ in Equation \eqref{qeeval} can be written as
\[\textstyle{\frac{nd}{o(\sigma)o(\rho)}\!\sum_{u=0}^{o(\sigma)-1}\sum_{v=0}^{o(\rho)-1}\big(\frac{u}{o(\sigma)}-\varepsilon_{\Delta,\sigma,j,u}\big)\big(\frac{v}{o(\rho)}-\varepsilon_{\Delta,\rho,i,v}\big) \frac{1}{d}\sum_{k\in\mathbb{Z}/d\mathbb{Z}}\mathbf{e}\big(\frac{k(hu-v)}{d}\big)}.\] Changing the summation order thus yields $\frac{n}{o(\sigma)o(\rho)}$ times
\[\textstyle{\sum_{k\in\mathbb{Z}/d\mathbb{Z}}\big[\sum_{u=0}^{o(\sigma)-1}\big(\frac{u}{o(\sigma)}-\varepsilon_{\Delta,\sigma,j,u}\big)\mathbf{e}\big(\frac{khu}{d}\big)\big]
\big[\sum_{v=0}^{o(\rho)-1}\big(\frac{v}{o(\rho)}-\varepsilon_{\Delta,\rho,i,v}\big)\mathbf{e}\big(\frac{-kv}{d}\big)\big]},\] and the sums involving the symbols $\varepsilon_{\Delta,\sigma,j,u}$ and $\varepsilon_{\Delta,\rho,i,v}$ from Equation \eqref{epsDelsigjchi} can be simplified to $\sum_{u=o(\sigma)-\beta_{\sigma,j}}^{o(\sigma)-1}\mathbf{e}\big(\frac{khu}{d}\big)$ and $\sum_{v=o(\rho)-\beta_{\rho,i}}^{o(\rho)-1}\mathbf{e}\big(\frac{-kv}{d}\big)$ respectively.

When $k=0$ the exponents equal 1, so that the simplified sums from the $\varepsilon$-symbols yield $\beta_{\sigma,j}$ and $\beta_{\rho,i}$, and the other sums over $u$ and $v$ produce $\frac{o(\sigma)-1}{2}$ and $\frac{o(\rho)-1}{2}$ respectively. Recalling the external coefficient $\frac{n}{o(\sigma)o(\rho)}$, Equation \eqref{qDelta} shows that these terms cancel with the term $nq_{\Delta}(\sigma,j;\rho,i)$ from Equation \eqref{qeeval}, and when $\beta_{\sigma,j}=\beta_{\rho,i}=0$ it produces the second term in the asserted formula for $\gamma_{\sigma,\rho}$. On the other hand, for $k\neq0$ the (geometric) sums from the $\varepsilon$-symbols yield $\frac{1-\mathbf{e}(-kh\beta_{\sigma,j}/d)}{1-\mathbf{e}(kh/d)}$ and $\frac{1-\mathbf{e}(k\beta_{\rho,i}/d)}{1-\mathbf{e}(-k/d)}$ respectively, and for the other sums we recall from Lemma 6.5 of \cite{[KZ]} that when $y$ is a non-trivial $r$th root of unity, the sum $\sum_{l=0}^{r-1}\frac{l}{r}y^{l}$ equals $\frac{-1}{1-y}$. We set $r=o(\sigma)$ and $y=\mathbf{e}\big(\frac{kh}{d}\big)$ (resp. $r=o(\rho)$ and $y=\mathbf{e}\big(\frac{-k}{d}\big)$) and deduce that the other sum over $u$ (resp. $v$) equals $-\frac{1}{1-\mathbf{e}(kh/d)}$ (resp. $-\frac{1}{1-\mathbf{e}(-k/d)}$), so that the total multipliers in the summand associated with $k$ are $\frac{-\mathbf{e}(-kh\beta_{\sigma,j}/d)}{1-\mathbf{e}(kh/d)}$ and $\frac{-\mathbf{e}(k\beta_{\rho,i}/d)}{1-\mathbf{e}(-k/d)}$. Multiplying by the coefficient $\frac{n}{o(\sigma)o(\rho)}$, and recalling that for $\gamma_{\sigma,\rho}$ (with $\beta_{\sigma,j}=\beta_{\rho,i}=0$) we have an extra coefficient of $\frac{1}{n}$, we find that both the required expressions now follow from Equation \eqref{phihd}. This completes the proof of the proposition.
\end{proof}
Note that $q_{e}(\sigma,j;\rho,i)$ depends only on $e$, but we evaluated it in Proposition \ref{expval} using the term $\beta_{\rho,i}^{\Delta}-h\beta_{\sigma,j}^{\Delta}$ arising from a divisor $\Delta$ with $u(\Delta)=e$. Equation \eqref{epsDelsigjchi} shows that replacing $\Delta$ by $\chi\Delta$ adds $u_{\chi,\rho}-hu_{\chi,\sigma}$ (which was seen in the proof of Proposition \ref{expval} to be a multiple of $d$) to that argument, plus multiples of $o(\sigma)$ and $o(\rho)$, both of which are also divisible by $d$. Since $\phi_{h+d\mathbb{Z}}(\xi)$ depends only on the image of $\xi$ in $\mathbb{Z}/d\mathbb{Z}$, we deduce that $\phi_{h+d\mathbb{Z}}(\beta_{\rho,i}^{\Delta}-h\beta_{\sigma,j}^{\Delta})$ indeed depends only on $e$ and not on the choice of the representing divisor $\Delta$.

\smallskip

Proposition \ref{expval} shows that $\phi_{h+d\mathbb{Z}}(\xi)\in\mathbb{Q}$ for every $d$, $h$, and $\xi$. We shall later need to bound the denominators universally.
\begin{lem}
If $d$ is co-prime to 6 then $\phi_{h+d\mathbb{Z}}(\xi)\in\mathbb{Z}$ for every $\xi\in\mathbb{Z}$. In case $d$ is odd but divisible by 3 we have $\phi_{h+d\mathbb{Z}}(\xi)\in\frac{-h}{3}+\mathbb{Z}$ (which is not integral since 3 does not divide $h$). For even $d$ not divisible by 3 the number $\phi_{h+d\mathbb{Z}}(\xi)$ lies in $\frac{1+2\xi}{4}+\mathbb{Z}$. Finally, if $d$ is divisible by 6 then $\phi_{h+d\mathbb{Z}}(\xi)$ belongs to $\frac{1+2\xi}{4}-\frac{h}{3}+\mathbb{Z}$. \label{Dedekind}
\end{lem}

\begin{proof}
Given $d$, $h$, and $\xi$ consider a cyclic group of order $d$ with generator $\sigma$, take $\rho=\sigma^{h}$, and assume that in the divisor $\Delta$ on some $Z_{d}$-curve $X$ (with positive $r_{\sigma}$ and $r_{\rho}$) we have indices $j$ and $i$ such that $\beta_{\sigma,j}^{\Delta}=0$ and $\beta_{\rho,i}^{\Delta}=\xi$. Then the dual group is cyclic as well, with a generator $\eta$ that sends $\sigma$ to $\mathbf{e}\big(\frac{1}{d}\big)$. Equations \eqref{uchisigma} and \eqref{epsDelsigjchi} then imply that \[\mathrm{for}\quad\chi=\eta^{u}\quad\mathrm{we \ have}\quad\beta_{\sigma,j}^{\chi\Delta}=u_{\chi,\Delta}=u\quad\mathrm{and}\quad\beta_{\rho,i}^{\chi\Delta}=d\big\{\tfrac{hu+\xi}{d}\big\},\] and as $\eta$ generates the dual group, Proposition \ref{expval} and Equation \eqref{qDelta} yield \[\textstyle{\phi_{h+d\mathbb{Z}}(\xi)=dq_{e}(\sigma,j;\rho,i)=d\sum_{u=0}^{d-1}\big(\frac{u}{d}-\frac{d-1}{2d}\big)\big(\big\{\frac{hu+\xi}{d}\big\}-\frac{d-1}{2d}\big)}.\] As this is now clearly a (generalized) Dedekind sum, we adapt the argument from Section 3 of \cite{[Rd]}. We separate the first multiplier, and recall from Equation \eqref{vansum} that then the constant $\frac{d-1}{2d}$ multiplies a vanishing sum. Using the definition of the fractional part, the remaining part of our expression for $\phi_{h+d\mathbb{Z}}(\xi)$ becomes
\begin{equation}
\textstyle{\sum_{u=0}^{d-1}u\big(\frac{hu}{d}+\frac{2\xi-(d-1)}{2d}-\big\lfloor\frac{hu+\xi}{d}\big\rfloor\big)\in\frac{h(d-1)(2d-1)}{6}+\frac{(d-1)(2\xi-d+1)}{4}+\mathbb{Z}}. \label{phidenom}
\end{equation}
Now, when $d$ is odd the second term in Equation \eqref{phidenom} is integral, and the first one is $\frac{1}{3}$ times the product of $h$, $\frac{d-1}{2}$, and $2d-1$. If $d$ is not divisible by 3 then 3 divides one of the latter numbers, while otherwise these numbers have residues 1 and 2 modulo 3. This proves the first two assertions. On the other hand, take $d=2k$ to be even, so that $h$ is odd and the second term in Equation \eqref{phidenom} becomes $\frac{2\xi-1}{4}$ plus the integer $(k-1)(\xi-k)$. When $d$ (or equivalently $k$) is not divisible by 3, the first term is the product of the odd numbers $h$, $2k-1$, and $4k-1$ divided by 6, and as one of the latter multipliers is divisible by 3 the product lies in $\frac{1}{2}+\mathbb{Z}$. Adding $\frac{2\xi-1}{4}$ from the second term yields the third assertion. Otherwise 3 divides $k$, and the first term is $\frac{h}{6}$ plus an integer. Writing it as $\frac{h}{2}-\frac{h}{3}\in\frac{1}{2}-\frac{h}{3}+\mathbb{Z}$ establishes the fourth assertion. This proves the lemma.
\end{proof}

At this point we invoke (again, as in \cite{[Na]}, \cite{[Ko2]}, and \cite{[Ko3]}) Corollary 2.12 of \cite{[Fa]}. Observing that $\theta[e](\zeta,\tau)$ differs from $\theta[0](\zeta+e,\tau)$ by the exponential of a linear function of $\{\zeta_{s}\}_{s=1}^{g}$, and the effect of this exponential disappears after taking a logarithm and differentiating with respect to two $\zeta$-variables, this corollary reads as follows.
\begin{prop}
Let $X$ be a compact Riemann surface, and let $e$ be a point in $J(X)$ such that $\theta[e](0,\tau)\neq0$. Then one has the equality
\[\textstyle{S[e](P,Q)S[-e](P,Q)=\omega(P,Q)+\sum_{r=1}^{g}\sum_{s=1}^{g}\frac{\partial^{2}\ln\theta[e]}{\partial\zeta_{r}\partial\zeta_{s}}\Big|_{\zeta=0}v_{r}(P)v_{s}(Q)}.\] \label{Fayident}
\end{prop}

In our setting, we recall from Proposition \ref{liftcirc} that for $Id\neq\sigma \in A$ and an index $1 \leq j \leq r_{\sigma}$ there are $\frac{n}{o(\sigma)}$ points $R \in X$ with $z(R)=\lambda_{\sigma,j}$, with respective coordinates $t_{R}$ satisfying $t_{R}^{o(\sigma)}=z-\lambda_{\sigma,j}$. We now deduce the following equality.
\begin{cor}
For every characteristic $e=u(\Delta)+K$ considered in Theorems \ref{nonspdiv} and \ref{Szegoalg}, every $Id_{X}\neq\sigma \in A$, and every $1\leq j\leq r_{\sigma}$, the expression
\begin{equation}
n\sum_{(\rho,i)\neq(\sigma,j)}\frac{2\phi_{h+d\mathbb{Z}}(\xi)+\phi_{h+d\mathbb{Z}}(0)+\frac{(o(\sigma)-1)(o(\rho)-1)}{4}}
{o(\sigma)o(\rho)(\lambda_{\sigma,j}-\lambda_{\rho,i})}-\sum_{\{\chi\in\widehat{A}|\chi(\sigma)\neq1\}}\frac{f_{\chi,2}^{\chi}(\lambda_{\sigma,j})}{(f_{\chi,0}^{\chi})'(\lambda_{\sigma,j})}, \label{termswithphi}
\end{equation}
with $d$ and $h$ as in Equation \eqref{dhdef} and where $\xi=\beta_{\rho,i}^{\Delta}-h\beta_{\sigma,j}^{\Delta}$, equals
\begin{equation}
\frac{1}{o(\sigma)\big(o(\sigma)-2\big)!}\sum_{r=1}^{g}\sum_{s=1}^{g}\frac{\partial^{2}\ln\theta[e]}{\partial\zeta_{r}\partial\zeta_{s}}\bigg|_{\zeta=0}
\sum_{\{R|z(R)=\lambda_{\sigma,j}\}}\frac{d^{o(\sigma)-2}}{dt_{R}^{o(\sigma)-2}}\bigg(\frac{v_{r}}{dt_{R}}\cdot\frac{v_{s}}{dt_{R}}\bigg)\bigg|_{t_{R}=0}. \label{dervsvr}
\end{equation} \label{firstrel}
\end{cor}

\begin{proof}
Proposition \ref{Fayident} yields via a simple summation that the difference
\begin{equation}
\textstyle{\sum_{\rho \in A}S[e](\rho P,\rho Q)S[-e](\rho P,\rho Q)-\sum_{\rho \in A}\omega(\rho P,\rho Q)} \label{sumSzegoomega}
\end{equation}
equals
\begin{equation}
\textstyle{\sum_{\rho \in A}\sum_{r=1}^{g}\sum_{s=1}^{g}\frac{\partial^{2}\ln\theta[e]}{\partial\zeta_{r}\partial\zeta_{s}}\Big|_{\zeta=0}v_{r}(\rho P)v_{s}(\rho Q)}. \label{thetavrho}
\end{equation}
We assume that $Q$ is in the neighborhood of some pre-image $R$ of $\lambda_{\sigma,j}$, with the coordinate $t=t_{R}$, and that $P$ is close to $Q$, and expand the terms from the equality in $t(Q)$ and $t(P)-t(Q)$. Since the expansion of $S[e](P,Q)S[-e](P,Q)$ in Corollary \ref{prodSzego} is only in terms of $z$-values, the first term in Equation \eqref{sumSzegoomega} can be written as $n$ times the expansion in $t$ appearing in Lemma \ref{expSzego}. Next, we write $\omega$ as in Proposition \ref{candifprop}, and in the expression for $G_{B}/6$ from Corollary \ref{Bergman} the action of $\rho$ again leaves the term involving $z$ invariant, but the factor $y_{\chi}(Q)y_{\overline{\eta}}(Q)$ is multiplied by $\chi(\rho)\overline{\eta}(\rho)$. Changing the coordinate to $t$, the sum over $\rho$ thus gives $n$ times the value of $G_{B}(t)/6$ given in Lemma \ref{expsint}, but with only the terms in which $\eta=\chi$ remaining in the sum over the characters there. The resulting expansion of Equation \eqref{sumSzegoomega} is therefore $o(\sigma)^{2}dt(P)dt(Q)$ times
\begin{equation}
\sum_{(\rho,i)\neq(\sigma,j)}\frac{2q_{e}(\sigma,j;\rho,i)+n\gamma_{\sigma,\rho}}{\lambda_{\sigma,j}-\lambda_{\rho,i}}t(Q)^{o(\sigma)-2}-
\sum_{\chi\in\widehat{A}}\frac{f_{\chi,2}^{\chi}(\lambda_{\sigma,j})}{c_{\overline{\chi}}c_{\chi}}t(Q)^{2o(\sigma)-2-u_{\chi,\sigma}-u_{\overline{\chi},\sigma}}, \label{expofdif}
\end{equation}
up to $O\big(t(P)-t(Q)\big)$ and another error term of the form $O\big(t(Q)^{o(\sigma)}\big)$. Lemma \ref{chicomp} shows that the exponent in the term associated with $\chi$ in Equation \eqref{expofdif} is $2o(\sigma)-2$ when $\chi(\sigma)=1$ and $o(\sigma)-2$ otherwise, so that the former elements can be included in the error term $O\big(t(Q)^{o(\sigma)}\big)$. Moreover, when $\chi(\sigma)\neq1$ the product $c_{\chi}c_{\overline{\chi}}$ in the denominator is the value of $\frac{y_{\chi}y_{\overline{\chi}}}{z-\lambda_{\sigma,j}}$ at a pre-image of $\lambda_{\sigma,j}$, which was evaluated as $\big(f_{\chi,0}^{\chi}\big)'(\lambda_{\sigma,j})$ (which is non-zero since $f_{\chi,0}^{\chi}$ has no multiple roots). We evaluate $2q_{e}(\sigma,j;\rho,i)$ and $n\gamma_{\sigma,\rho}$ using Proposition \ref{expval}, and deduce that Equation \eqref{termswithphi} gives the coefficient of $t(Q)^{o(\sigma)-2}$ in Equation \eqref{expofdif}.

Since Proposition \ref{Fayident} compares Equations \eqref{sumSzegoomega} and \eqref{thetavrho}, and we have shown that multiplying Equation \eqref{termswithphi} by $o(\sigma)^{2}dt(P)dt(Q)$ yields the coefficient of $t(Q)^{o(\sigma)-2}$ in the expansion of the former in terms of $t(Q)$ and $t(P)-t(Q)$, we must therefore expand the expression from Equation \eqref{thetavrho} in the same variables, divide by $o(\sigma)^{2}dt(P)dt(Q)$, and see that the coefficient of $t(Q)^{o(\sigma)-2}$ is the one from Equation \eqref{dervsvr}. The theta derivatives are independent of $P$ and $Q$, and since we consider the ``constant term in powers of $t(P)-t(Q)$'', we may substitute $P=Q$ in $v_{r}$ in Equation \eqref{thetavrho}. For the contribution of our point $Q$, near the pre-image $R$ of $\lambda_{\sigma,j}$ (at which $t_{R}$ vanishes), we may therefore evaluate $\frac{v_{r}}{dt_{R}}\cdot\frac{v_{s}}{dt_{R}}$ using its Taylor expansion (divided also by $o(\sigma)^{2}$). Doing so with $\rho Q$ for each $\rho \in A$, we find that each one of the $\frac{n}{o(\sigma)}$ pre-images of $\lambda_{\sigma,j}$ is obtained in this way by $o(\sigma)$ elements $\rho \in A$ (these are the different cosets of the stabilizer of $R$ in $A$ from Proposition \ref{liftcirc}), and different such elements will give the same coordinate $t$ but multiplied by a different root of unity of order $o(\sigma)$. Because of the division by $dt^{2}$, replacing $t$ by $\alpha t$ with such a root of unity $\alpha$ multiplies the coefficient of $t^{k}$ of $\alpha^{k-2}$, so that in the summation over the stabilizer of $R$ in $A$ the coefficients of $t^{k}$ for $k\in-2+o(\sigma)\mathbb{Z}$ (that are invariant) are multiplied by $o(\sigma)$, and the others vanish. Summing over the full group $A$ collects the resulting expansions from all the pre-images of $\lambda_{\sigma,j}$, and as the exponent $o(\sigma)-2$ in which we are interested is such that its Taylor coefficient survives, substituting it indeed yields Equation \eqref{dervsvr}. This completes the proof of the corollary.
\end{proof}

\begin{rmk}
The references \cite{[Na]} and \cite{[Ko2]} did not apply the averaging included in the proof of Corollary \ref{firstrel} (the sum $\sum_{\alpha=0}^{o(\sigma)-2}\binom{o(\sigma)-2}{\alpha}v_{r}^{(\alpha)}(Q)v_{s}^{(o(\sigma)-2-\alpha)}(Q)$ from these references is just the expansion of the derivative of the product $\frac{v_{r}}{dt_{R}}\cdot\frac{v_{s}}{dt_{R}}$). We carried it out in order to leave only the terms with $\chi=\eta$ in the relevant sum in Lemma \ref{expsint} (for using Proposition \ref{fchichi2} below), which also improves the error term to $O(t^{2o(\sigma)-2})$ (in correspondence with the expressions appearing there being $A$-invariant and involving $dt(P)dt(Q)$), as well as to get the sum over all the pre-images of $\lambda_{\sigma,j}$ in Corollary \ref{firstrel} (as in our application of Theorem \ref{variational} below). When $A$ is cyclic and $\sigma$ is a generator, the value of $u_{\chi,\sigma}$ determines $\chi$, there is only one pre-image of $\lambda_{\sigma,j}$ in $X$, and the only contribution to the coefficient of $t^{n-2}dt^{2}$ in the sum over $\chi$ and $\eta$ arises from terms with $\chi=\eta\neq\mathbf{1}$. This is why these references could avoid the averaging. Note that none of these statements hold if $\langle\sigma\rangle \neq A$, so that in any case that is more general than the one considered in \cite{[Ko2]}, this averaging must be executed. \label{NaKoexp}
\end{rmk}

\section{The Thomae Formulae \label{Thomae}}

In this section we replace several terms in Corollary \ref{firstrel} by equivalent expressions, in order to produce a differential equation whose solution will yield the Thomae formula that we seek. The combination of the polynomials is related to the derivative of a certain determinant, and the term involving derivatives of theta functions will become a derivative with respect to a branching value. Since we shall employ integration of some meromorphic differentials with non-trivial residues, the notation $a_{i}$ and $b_{i}$ with $1 \leq i \leq g$ here will stand for fixed smooth paths on our $A$-cover $X$ of $\mathbb{CP}^{1}$ (representing the corresponding homology classes, of course). We assume that these paths do not pass through any branch point or through any pole of $z$. We shall consider the numbers $r_{\rho}$ with $Id_{X}\neq\rho \in A$ as fixed, and with them the genus $g$ (by Proposition \ref{genus}). Hence we may view $X$ as a constant real surface (together with the paths $a_{i}$ and $b_{i}$ with $1 \leq i \leq g$), and let the complex structure vary with the choice of the branching values $\lambda_{\rho,i}$ with $Id_{X}\neq\rho \in A$ and $1 \leq i \leq r_{\rho}$, or more explicitly with respect to perturbing one particular value $\lambda_{\sigma,j}$.

Denote by $C$ the transition matrix between the dual basis $v_{s}$, $1 \leq s \leq g$ and the basis $\bigcup_{\mathbf{1}\neq\chi\in\widehat{A}}\big\{z^{k}\psi_{\chi}\big|0 \leq k \leq t_{\overline{\chi}}-2\big\}$ from Proposition \ref{holdif}. Its entries are of the form $\int_{a_{i}}z^{k}\psi_{\chi}$ (by the definition of the former basis), where $i$ is the row index and $\chi$ and $k$ represent the column. As in \cite{[Na]}, \cite{[Ko2]}, and \cite{[Ko3]}, we shall require the evaluation of the derivative $\frac{\partial\ln\det C}{\partial\lambda_{\sigma,j}}$, using the following lemma.
\begin{lem}
Let $B_{\sigma,j}$ denote the matrix with entries $\int_{a_{i}}(z-\lambda_{\sigma,j})^{k}\psi_{\chi}$. For any $\chi\in\widehat{A}$ with $t_{\chi}\geq2$ we define $B_{\sigma,j}^{\chi}$ to be the matrix $B_{\sigma,j}$ but in which the column corresponding to the differential $\psi_{\overline{\chi}}=\frac{dz}{y_{\chi}}$ (with $k=0$) is replaced by the integrals of $\frac{u_{\chi,\sigma}}{o(\sigma)}\cdot\frac{\psi_{\overline{\chi}}}{z-\lambda_{\sigma,j}}$. Then we have the equality \[\textstyle{\frac{\partial\ln\det C}{\partial\lambda_{\sigma,j}}=\sum_{\{\chi\in\widehat{A}|t_{\chi}\geq2,\chi(\sigma)\neq1\}}\det B_{\sigma,j}^{\chi}\big/\det B_{\sigma,j}}.\] \label{lndetCder}
\end{lem}

\begin{proof}
The expansion of $(z-\lambda_{\sigma,j})^{k}$ binomially yields $\det B_{\sigma,j}=\det C$, so that the required assertion is equivalent to \[\textstyle{\frac{\partial\det B_{\sigma,j}}{\partial\lambda_{\sigma,j}}=\sum_{\{\chi\in\widehat{A}|t_{\chi}\geq2,\chi(\sigma)\neq1\}}\det B_{\sigma,j}^{\chi}}\qquad\mathrm{(without\ the\ logarithm)}.\] But if $H$ is a matrix of functions of $x$ and $H_{l}$ is obtained from $H$ by replacing the $l$th column by its derivative then $\frac{d(\det H)}{dx}$ equals $\sum_{l}\det H_{l}$ (just write $H(x+\varepsilon)$ as $H(x)+\varepsilon H'(x)+O(\varepsilon^{2})$ and expand the determinant). Applying this for $B_{\sigma,j}$ and the variable $\lambda_{\sigma,j}$, we have to replace each differential $(z-\lambda_{\sigma,j})^{k}\frac{dz}{y_{\chi}}$ by its derivative with respect to $\lambda_{\sigma,j}$, before integrating along the $a_{i}$s and taking the determinant. But by Equation \eqref{ddzlnychi} this differential is (locally) $(z-\lambda_{\sigma,j})^{k-u_{\sigma,\chi}/o(\sigma)}$ times a differential that does not depend on $\lambda_{\sigma,j}$. Hence its derivative with respect to $\lambda_{\sigma,j}$ is $\frac{u_{\sigma,\chi}}{o(\sigma)}-k$ times the differential associated with $k-1$, so that if $k>0$ then we get a multiple of a differential appearing on another column. The corresponding determinant thus vanishes and does not contribute to our expression for $\frac{\partial\det B_{\sigma,j}}{\partial\lambda_{\sigma,j}}$. For $k=0$ we obtain from $\psi_{\overline{\chi}}=\frac{dz}{y_{\chi}}$ the matrix $B_{\sigma,j}^{\chi}$, including the coefficient $\frac{u_{\sigma,\chi}}{o(\sigma)}$. The restriction $t_{\chi}\geq2$ appears in Proposition \ref{holdif} (and is hence required for columns with $\chi$ to appear in $C$ or in $B_{\sigma,j}$), and as for characters $\chi$ with $\chi(\sigma)=1$ we have $u_{\chi,\sigma}=0$ by Equation \eqref{uchisigma}, the restriction $\chi(\sigma)\neq1$ does not change the total expression. This proves the lemma.
\end{proof}

We deduce the following relation to a term appearing in Corollary \ref{firstrel}.
\begin{prop}
For every non-trivial $\sigma \in A$ and $1 \leq j \leq r_{\sigma}$ we have \[\sum_{\{\chi\in\widehat{A}|\chi(\sigma)\neq1\}}\frac{f_{\chi,2}^{\chi}(\lambda_{\sigma,j})}{(f_{\chi,0}^{\chi})'(\lambda_{\sigma,j})}=
\sum_{\{\chi\in\widehat{A}|t_{\chi}\geq2,\chi(\sigma)\neq1\}}\frac{f_{\chi,2}^{\chi}(\lambda_{\sigma,j})}{(f_{\chi,0}^{\chi})'(\lambda_{\sigma,j})}=-\frac{\partial\ln\det C}{\partial\lambda_{\sigma,j}}.\] \label{fchichi2}
\end{prop}

\begin{proof}
Recall from Lemma \ref{intomegaeta} that the integral $\int_{a_{i}}\omega_{\eta}(P,Q)$ (in $P$) vanishes for every $1 \leq i \leq g$ and $Q \in X$. Take $Q$ as in Lemma \ref{expsint}, and substitute the expansion of $\omega_{\eta}(P,Q)$ from the first assertion of that lemma. Take $\eta$ with $\eta(\sigma)\neq1$ (for that expansion to have meaning), and compare the integral of the coefficient of $t^{o(\sigma)-1-u_{\overline{\eta},\sigma}}dt$ to 0. Dividing by $1-\delta_{\eta(\sigma),1}=1$, by $\frac{o(\sigma)}{c_{\eta}}$, and by the non-zero number $\big(f_{\eta,0}^{\eta}\big)'(\lambda_{\sigma,j})$, the resulting equality for $1 \leq i \leq g$ is
\[\sum_{\chi\in\widehat{A}}\sum_{l=2}^{t_{\chi}}\frac{f_{\chi,l}^{\eta}(\lambda_{\sigma,j})}{\big(f_{\eta,0}^{\eta}\big)'(\lambda_{\sigma,j})}\int_{a_{i}}(z-\lambda_{\sigma,j})^{l-2}\psi_{\overline{\chi}}=
-\frac{u_{\eta,\sigma}}{o(\sigma)}\int_{a_{i}}\frac{\psi_{\overline{\eta}}}{z-\lambda_{\sigma,j}}.\] This yields a non-homogenous system of $g$ linear equations, in which the indeterminates are the $g$ expressions $f_{\chi,l}^{\eta}(\lambda_{\sigma,j})\big/\big(f_{\eta,0}^{\eta}\big)'(\lambda_{\sigma,j})$ with $\chi\in\widehat{A}$ and $2 \leq l \leq t_{\chi}$. The matrix of coefficients is the matrix $B_{\sigma,j}$ from Lemma \ref{lndetCder} (which is invertible since its determinant equals $\det C\neq0$), and the vector on the right is minus the column that we used for defining $B_{\sigma,j}^{\eta}$ in that lemma. Solving this system using Cramer's Rule, we see that for $\chi=\eta$ and $l=2$ the matrix in the numerator is $B_{\sigma,j}^{\eta}$ with the sign of the new column inverted. Lemma \ref{lndetCder} therefore establishes the equality between the asserted right hand side and the middle sum, and since $f_{\chi,2}^{\chi}=0$ if $t_{\chi}\leq1$ (because $f_{\chi,l}^{\eta}=0$ if $l \geq t_{\chi}$), the equality with the left hand side also follows. This proves the proposition.
\end{proof}
The relation from Proposition \ref{fchichi2} is one of the reasons why we needed the averaging in Corollary \ref{firstrel}---see Remark \ref{NaKoexp}.

\smallskip

At this point we shall invoke a variational formula due to Rauch. Consider a parametrized holomorphic family $\{X_{\mu}\}_{\{\mu\in\mathbb{C},\ |\mu|<\varepsilon\}}$ of genus $g$ Riemann surfaces, set $X=X_{0}$, choose a canonical homology basis for $X$, and identify $X_{\mu}$ with $X$ as real manifolds (with the homology basis). This defines the dual basis $\{v_{s}\}_{s=1}^{g}$ of differentials of the first kind on $X_{\mu}$, varying holomorphically with $\mu$, and we take a holomorphic map $z:X_{\mu}\to\mathbb{CP}^{1}$ that changes holomorphically with $\mu$ as well. Around every point $P \in X$ we choose a coordinate $t_{P}$ such that $t_{P}^{e_{P}}$ is a function $\varphi_{P}(z)$ of $z$, where $e_{P}$ is the ramification index of $z$ at $P$. If $t_{P,\mu}$ is a coordinate around $P$ but now in $X_{\mu}$, and $t_{P,\mu}$ depends holomorphically on $\mu$, then we expand $t_{P}^{e_{P}}=\varphi_{P}(z)$ around $P$ in $X_{\mu}$. This yields, for every $0\leq\nu<e_{P}$, a holomorphic function $c_{P,\nu}$ of $\mu$ with $c_{P,\nu}(0)=0$ such that \[t_{P}^{e_{P}}=t_{P,\mu}^{e_{P}}+\sum_{\nu=0}^{e_{P}-1}c_{P,\nu}(\mu)t_{P,\mu}^{\nu},\mathrm{\ and\ hence\ }t_{P}=t_{P,\mu}\bigg[1+\frac{\mu}{e_{P}}\sum_{\nu=0}^{e_{P}-1}\frac{c_{P,\nu}'(0)}{t_{P,\mu}^{e_{P}-\nu}}\bigg]+O(\mu^{2})\] by taking the power $\frac{1}{e_{P}}$ and expanding. Then the period matrix $\tau(\mu)$ of $X_{\mu}$ is a holomorphic function of $\mu$, whose derivative is evaluated by the \emph{Rauch variational formula}, as appearing in Equation (29) of \cite{[Ra]}.
\begin{thm}
For any $1 \leq r,s \leq g$ we have the equality
\[\frac{d\tau_{rs}(\mu)}{d\mu}\bigg|_{\mu=0}=\sum_{P \in X}\frac{2\pi i}{e_{P}}\sum_{\nu=0}^{e_{P}-2}\frac{c_{P,\nu}'(0)}{(e_{P}-2-\nu)!}\frac{d^{e_{P}-2-\nu}}{dt^{e_{P}-2-\nu}}\bigg(\frac{v_{r}}{dt_{P}}\cdot\frac{v_{s}}{dt_{P}}\bigg)\bigg|_{t=0}.\] \label{variational}
\end{thm}
This allows us to deduce the following result.
\begin{cor}
If $e$ is as in Corollary \ref{firstrel} then the derivative $\frac{\partial\ln\theta[e](0,\tau)}{\partial\lambda_{\sigma,j}}$ equals \[\frac{1}{2}\frac{\partial\ln\det C}{\partial\lambda_{\sigma,j}}+\frac{n}{2}\sum_{(\rho,i)\neq(\sigma,j)}
\frac{2\phi_{h+d\mathbb{Z}}(\xi)+\phi_{h+d\mathbb{Z}}(0)+\frac{(o(\sigma)-1)(o(\rho)-1)}{4}}{o(\sigma)o(\rho)(\lambda_{\sigma,j}-\lambda_{\rho,i})}.\] \label{difeqtheta}
\end{cor}

\begin{proof}
The Riemann surface $X_{\mu}$ in our holomorphic family is defined by replacing $\lambda_{\sigma,j}$ by $\lambda_{\sigma,j}+\mu$ and leaving the other $\lambda_{\rho,i}$s invariant. Hence if the branch point $P$ satisfies $z(P)=\lambda_{\rho,i}$ then $e_{P}=o(\rho)$ and $t_{P}$ is defined by $t_{P}^{o(\rho)}=z-\lambda_{\rho,i}$, so that $t_{P,\mu}$ satisfies $t_{P,\mu}^{o(\sigma)}=z-\lambda_{\sigma,j}-\mu$ when $\rho=\sigma$ and $i=j$ and $t_{P,\mu}^{o(\rho)}=z-\lambda_{\rho,i}$ otherwise. It follows that $c_{P,\nu}(\mu)=\delta_{\rho,\sigma}\delta_{i,j}\delta_{\nu,0}\cdot\mu$, so that the equality from Theorem \ref{variational} in this case becomes \[\frac{\partial\tau_{rs}}{\partial\lambda_{\sigma,j}}=\frac{2\pi i}{o(\sigma)\big(o(\sigma)-2\big)!}
\sum_{\{R|z(R)=\lambda_{\sigma,j}\}}\frac{d^{o(\sigma)-2}}{dt_{R}^{o(\sigma)-2}}\bigg(\frac{v_{r}}{dt_{R}}\cdot\frac{v_{s}}{dt_{R}}\bigg)\bigg|_{t_{R}=0},\] where the right hand side (divided by $2\pi i$) appears in Equation \eqref{dervsvr}. The other multiplier from that equation transforms, via the proof of Lemma \ref{highordFe}, the vanishing from Corollary \ref{thetader0}, and the heat equation from Equation \eqref{heat}, as
\[\frac{\partial^{2}\ln\theta[e]}{\partial\zeta_{r}\partial\zeta_{s}}\bigg|_{\zeta=0}=\frac{\frac{\partial^{2}\theta[e]}{\partial\zeta_{r}\partial\zeta_{s}}\Big|_{\zeta=0}}{\theta[e](0,\tau)}-
\frac{\frac{\partial\theta[e]}{\partial\zeta_{r}}\Big|_{\zeta=0}\cdot\frac{\partial\theta[e]}{\partial\zeta_{s}}\Big|_{\zeta=0}}{\theta[e](0,\tau)^{2}}=2\pi i(1+\delta_{r,s})\frac{\frac{\partial\theta[e](0,\tau)}{\partial\tau_{rs}}}{\theta[e](0,\tau)},\] so that the full expression in Equation \eqref{dervsvr} is
\[\frac{2\pi i}{2\pi i}\sum_{r=1}^{g}\sum_{s=1}^{g}(1+\delta_{r,s})\frac{\partial\ln\theta[e](0,\tau)}{\partial\tau_{rs}}\cdot\frac{\partial\tau_{rs}}{\partial\lambda_{\sigma,j}}=2\sum_{1 \leq r \leq s \leq g}\frac{\partial\ln\theta[e](0,\tau)}{\partial\tau_{rs}}\cdot\frac{\partial\tau_{rs}}{\partial\lambda_{\sigma,j}}.\] As $\theta[e](0,\tau)$ depends on $\lambda_{\sigma,j}$ only via $\tau$, this reduces to $2\frac{\partial\ln\theta[e](0,\tau)}{\partial\lambda_{\sigma,j}}$, and therefore Corollary \ref{firstrel} compares the required derivative with half of the expression from Equation \eqref{termswithphi}. Expressing the sum on the characters in that equation via Proposition \ref{fchichi2} yields the asserted expression. This proves the corollary.
\end{proof}
Note that while the characteristic $e=\tau\frac{\varepsilon}{2}+I\frac{\delta}{2}$ depends on $\tau$, the definition of the theta function is in terms of $\varepsilon$ and $\delta$. As these are rational vectors with bounded denominators (see Corollary \ref{torsion}) that vary continuously with $\lambda_{\sigma,j}$, they are indeed constant.

We shall need the integrality properties of appropriate multiples of the coefficients from Corollary \ref{difeqtheta}.
\begin{lem}
For $\sigma$, $\rho$, $d$, and $h$ as above and an arbitrary integer $\xi$, the expression $\frac{2mn}{o(\sigma)o(\rho)}\big[2\phi_{h+d\mathbb{Z}}(\xi)+\phi_{h+d\mathbb{Z}}(0)+\frac{(o(\sigma)-1)(o(\rho)-1)}{4}\big]$ is an even integer, unless the 2-Sylow part of $A$ is isomorphic to the Klein 4-group, $o(\sigma)$ and $o(\rho)$ are even, and $\sigma^{o(\sigma)/2}\neq\rho^{o(\rho)/2}$, where this number is an odd integer. \label{4meven}
\end{lem}

\begin{proof}
We claim that $2\phi_{h+d\mathbb{Z}}(\xi)+\phi_{h+d\mathbb{Z}}(0)$ is integral for odd $d$ and lies in $\frac{3}{4}+\mathbb{Z}$ when $d$ is even. Indeed, if 3 divides $d$ then the terms involving $h$ in Lemma \ref{Dedekind} combine to $-2\frac{h}{3}-\frac{h}{3}=-h\in\mathbb{Z}$, yielding the claim for odd $d$. The remaining terms for even $d$ are $2\frac{1+2\xi}{4}+\frac{1}{4}=\frac{3}{4}+\xi$, establishing the claim also for even $d$. The total expression is therefore the product of 2, $\frac{n}{m}$, $\frac{m}{o(\sigma)}$, $\frac{m}{o(\rho)}$, and $\frac{(o(\sigma)-1)(o(\rho)-1)+3\delta}{4}+l$ for an integer $l$, where $\delta=\delta_{d+2\mathbb{Z},2\mathbb{Z}}$ equals 0 for odd $d$ and 1 for even $d$. We thus have to verify that 2 divides the product of the three former quotients with $(o(\sigma)-1)(o(\rho)-1)+3\delta$ at least twice.

We distinguish among several cases. If both $o(\sigma)$ and $o(\rho)$ are odd then so is $d$, and 2 divides both $o(\sigma)-1$ and $o(\rho)-1$. When $o(\sigma)$ is odd and $o(\rho)$ is even, we have that $m$ is even and $d$ is again odd, and 2 divides both $o(\sigma)-1$ and $\frac{m}{o(\sigma)}$. The case where $o(\sigma)$ is even and $o(\rho)$ is odd is identical. When both $o(\sigma)$ and $o(\rho)$ are even, we assume first that $d$ is even as well. Observe that the residue of $(o(\sigma)-1)(o(\rho)-1)$ modulo 4 is 1 when $\frac{o(\sigma)}{2}$ and $\frac{o(\rho)}{2}$ have the same parity and 3 if these parities differ, and that in the latter case 4 must divide $m$. Therefore when $\frac{o(\sigma)}{2}$ and $\frac{o(\rho)}{2}$ have the same parity the number $(o(\sigma)-1)(o(\rho)-1)+3$ is already divisible by 4. On the other hand, if $\frac{o(\sigma)}{2}$ is odd and $\frac{o(\rho)}{2}$ is even then $(o(\sigma)-1)(o(\rho)-1)+3$ is only divisible by 2, but $\frac{m}{o(\sigma)}$ is also even. As the case in which $\frac{o(\sigma)}{2}$ is even and $\frac{o(\rho)}{2}$ is odd is again identical, the case of even $d$ is proved.

In the remaining case $o(\sigma)$ and $o(\rho)$ are even but $d$ is odd, and then the term $(o(\sigma)-1)(o(\rho)-1)+0$ is odd. In this case we consider the structure of the 2-Sylow part $A_{2}$ of $A$, and let $\langle\sigma\rangle_{2}$ (resp. $\langle\rho\rangle_{2}$) be the intersection of $\langle\sigma\rangle$ (resp. $\langle\rho\rangle$) with $A_{2}$. Since $o(\sigma)$ and $o(\rho)$ are even, these subgroups are non-trivial, but as $d$ is odd, they intersect trivially. This implies that $A_{2}$ cannot be cyclic, and therefore $\frac{n}{m}$ is even. Moreover, if $\frac{n}{2m}$ is even then 2 already divides $\frac{n}{m}$ twice, and when $\frac{n}{2m}$ is odd, $A_{2}$ must be the product of a cyclic group with a group of order 2. In addition, in case the index of $\langle\sigma\rangle_{2}$ (resp. $\langle\rho\rangle_{2}$) in $A_{2}$ is larger than 2 then 2 divides both $\frac{n}{m}$ and $\frac{m}{o(\sigma)}$ (resp. $\frac{m}{o(\sigma)}$), and we are done. Now, the only situation that we have not yet covered is when $A_{2}$ is the product of a cyclic group with a group of order 2 and $\langle\sigma\rangle_{2}$ and $\langle\rho\rangle_{2}$ are cyclic subgroups of index 2 in $A_{2}$ that intersect trivially. But if 4 divides $m$ then $A_{2}$ with this structure contains only two cyclic subgroups of index 2, and their intersection is non-trivial. Hence the described situation can occur only when $\frac{m}{2}$ is odd, and then the property of $A_{2}$ determines it as a Klein 4-group. In this case $\langle\sigma\rangle_{2}$ and $\langle\rho\rangle_{2}$ are distinct non-trivial cyclic subgroups of $A_{2}$, and they are clearly generated by $\sigma^{o(\sigma)/2}$ and $\rho^{o(\rho)/2}$ respectively (which are thus distinct elements of $A$). As in this special situation all the numbers $\frac{n}{2m}$, $\frac{m}{o(\sigma)}$, $\frac{m}{o(\rho)}$, and $(o(\sigma)-1)(o(\rho)-1)+0$ are odd, this completes the proof of the lemma.
\end{proof}

\smallskip

Recall that $n$ and $m$ are the order and the exponent of $A$ respectively, that $o(\sigma)$ is the order of $\sigma$ in $A$, and that $C$ is the matrix of integrals of the form $\int_{a_{i}}z^{k}\psi_{\chi}$ with $1 \leq i \leq g$, $\mathbf{1}\neq\chi\in\widehat{A}$, and $0 \leq k \leq t_{\overline{\chi}}-2$. Since we would like to consider each (unordered) couple of pairs $(\sigma,j)$ and $(\rho,i)$ just once, we fix some arbitrary full order on the set of such pairs, and denote it by $<$. We can now prove the main result of this paper.
\begin{thm}
Let $\Delta$ be a divisor given in the form of Equation \eqref{normdiv}, with $h^{\Delta}=1$ and such that the equalities from Theorem \ref{nonspdiv} hold (so that $\Delta$ is of degree $g-1$ and $r(-\Delta)=0$), and set $e=u(\Delta)+K$. Then there exists a complex number $\alpha_{e}$, which is independent of the branching values $\lambda_{\sigma,j}$, such that $\theta[e]^{4m}(0,\tau)$ equals
\[\alpha_{e}(\det C)^{2m}\prod_{(\sigma,j)<(\rho,i)}(\lambda_{\sigma,j}-\lambda_{\rho,i})^{\frac{2mn}{o(\sigma)o(\rho)}\big[2\phi_{h+d\mathbb{Z}}(\xi)+\phi_{h+d\mathbb{Z}}(0)+\frac{(o(\sigma)-1)(o(\rho)-1)}{4}\big]}\] for every choice of the branching values. Here given $\sigma$, $j$, $\rho$, and $i$, the indices $d$ and $h$ are defined in Equation \eqref{dhdef}, the argument $\xi$ is $\beta_{\rho,i}-h\beta_{\sigma,j}$ modulo $d\mathbb{Z}$, and $\phi_{h+d\mathbb{Z}}$ is the function defined in Equation \eqref{phihd}. Moreover, the value of $\alpha_{e}$ does not depend on the ordering of the differentials in the definition of $C$, and if the 2-Sylow subgroup of $A$ is not a Klein 4-group then $\alpha_{e}$ is also independent of the choice of the order $<$. \label{relpert}
\end{thm}

\begin{proof}
Corollary \ref{difeqtheta} implies that $\frac{\partial\ln\theta[e](0,\tau)}{\partial\lambda_{\sigma,j}}$ equals $\frac{\partial}{\partial\lambda_{\sigma,j}}$ of \[\frac{\ln\det C}{2}+\frac{n}{2}\sum_{(\rho,i)\neq(\sigma,j)}\frac{2\phi_{h+d\mathbb{Z}}(\xi)+\phi_{h+d\mathbb{Z}}(0)+\frac{(o(\sigma)-1)(o(\rho)-1)}{4}}{o(\sigma)o(\rho)}\ln(\lambda_{\sigma,j}-\lambda_{\rho,i}).\] We multiply by $4m$, and then the coefficient of $\ln(\lambda_{\sigma,j}-\lambda_{\rho,i})$ becomes integral by Lemma \ref{4meven}. Integrating the resulting differential equation and exponentiating compares $\theta[e]^{4m}(0,\tau)$ to \[\beta_{e,\sigma,j}(\det C)^{2m}\prod_{(\rho,i)\neq(\sigma,j)}(\lambda_{\sigma,j}-\lambda_{\rho,i})^{\frac{2mn}{o(\sigma)o(\rho)}\big[2\phi_{h+d\mathbb{Z}}(\xi)+\phi_{h+d\mathbb{Z}}(0)+\frac{(o(\sigma)-1)(o(\rho)-1)}{4}\big]}\] for some non-zero constant $\beta_{e,\sigma,j}$, where the product in only over the pair $(\rho,i)$. Here the constant $\beta_{e,\sigma,j}$ may depend on all the other branching values (and $e$), but not on $\lambda_{\sigma,j}$. We write $\beta_{e,\sigma,j}$ as another number $\alpha_{e,\sigma,j}$ times the product over all the couples of pairs not involving $(\sigma,j)$ (which is non-zero), and the new coefficient $\alpha_{e,\sigma,j}$ is also independent of $\lambda_{\sigma,j}$. This gives us the desired equality (perhaps up to a sign that we also absorb in the coefficient if necessary), but with $\alpha_{e,\sigma,j}$ instead of $\alpha_{e}$. However, as $\theta[e]^{4m}(0,\tau)$, $(\det C)^{2m}$, and the polynomial are all independent of the choice of $\sigma$ and $j$, the number $\alpha_{e,\sigma,j}$ is the same number for all $\sigma$ and $j$. We thus denote it by $\alpha_{e}$, and deduce that it does not depend on the value of $\lambda_{\sigma,j}$ for any $\sigma$ and $j$. Finally, $\theta[e](0,\tau)$ and the even power $2m$ of $\det C$ are not affected by altering the column order in $C$ or replacing $<$ by another order, and the latter simply replaces some of the differences $\lambda_{\sigma,j}-\lambda_{\rho,i}$ by their additive inverses. As Lemma \ref{4meven} shows that when $A_{2}$ is not a Klein 4-group, every such difference is raised to an even power, this expression is also independent of the choice of $<$, hence so is $\alpha_{e}$. This proves the theorem.
\end{proof}

\begin{rmk}
It can be shown, using the conditions for the irreducibility of $X$, that when $A_{2}$ is the Klein 4-group then some difference $\lambda_{\sigma,j}-\lambda_{\rho,i}$ with an odd exponent must appear in Theorem \ref{relpert}. Hence only $\alpha_{e}^{2}$, and not $\alpha_{e}$, is independent of $<$ in this case. Note that the power $4m$ (or $8m$ in the Klein $A_{2}$ case) is not the minimal one rendering our coefficients even integers in general, but we must take an even power in any case for the final exponent of $\det C$ to become integral. An analysis similar to the proof of Lemma \ref{4meven} discovers that a suitable power is $2m\big/\gcd\big\{m,\frac{n}{m}\big\}$ in case $\mathrm{lcm}\big\{m,\frac{n}{m}\big\}\big/m$ is even and $4m\big/\gcd\big\{m,\frac{n}{m}\big\}$ when it is odd, with very few exceptions that are based on the structure of $A_{2}$. For describing these, we denote by $(\upsilon_{1},\ldots,\upsilon_{q})$ the finite abelian 2-group $\prod_{l=1}^{q}\mathbb{Z}/2^{\upsilon_{l}}\mathbb{Z}$ for a decreasing sequence of positive integers $\{\upsilon_{l}\}_{l=1}^{q}$, and note that if $m$ is $2^{\upsilon}$ times an odd number and $\upsilon\geq1$ then $\upsilon_{1}=\upsilon$ in the sequence describing $A_{2}$. Then the finitely many possibilities of $A_{2}$, with $\upsilon_{1}=\upsilon$, for which we need a larger power in Theorem \ref{relpert} than the general rule from above are as follows. For $(\upsilon,\upsilon,1,1)$, as well as for $(\upsilon,\upsilon,2)$ and $(\upsilon,\upsilon-1,1,1)$ with $\upsilon\geq2$ and for $(\upsilon,\upsilon-1,2)$ with $\upsilon\geq3$, we need $4m\big/\gcd\big\{m,\frac{n}{m}\big\}$ even though $\mathrm{lcm}\big\{m,\frac{n}{m}\big\}\big/m$ is even. When the structure of $A_{2}$ is $(\upsilon,\upsilon,1)$, or when it is $(\upsilon,\upsilon-1)$ or $(\upsilon,\upsilon-1,1)$ with $\upsilon\geq2$, the minimal power is $8m\big/\gcd\big\{m,\frac{n}{m}\big\}$, and if $A_{2}$ has the structure $(\upsilon,\upsilon)$ then we must take $16m\big/\gcd\big\{m,\frac{n}{m}\big\}$ (the latter situation with $\upsilon=1$ is the Klein 4-group excluded in Lemma \ref{4meven} and Theorem \ref{relpert}, and then $8m$ is $16m\big/\gcd\big\{m,\frac{n}{m}\big\}$ times an odd number). In all these cases the resulting constant from Theorem \ref{relpert} does not depend on the order of the differentials in $C$, except when $\mathrm{lcm}\big\{m,\frac{n}{m}\big\}\big/m$ is even and we use the power $2m\big/\gcd\big\{m,\frac{n}{m}\big\}$, since the power $m\big/\gcd\big\{m,\frac{n}{m}\big\}$ of $\det C$ is odd in this case. The power we described here seems to be the minimal one in general, though in some particular situations it can be slightly improved (see, e.g., Equation \eqref{EGZ} below with $n\equiv1(\mathrm{mod\ }4)$). We conclude by remarking that since $\det C$ depends on the arbitrary normalization of the functions $y_{\chi}$ from Proposition \ref{decomgen}, it is difficult to say anything explicit about the value of $\alpha_{e}$. \label{minpow}
\end{rmk}

\section{Well-Definition on Moduli Spaces \label{Moduli}}

The coefficient $\alpha_{e}$ from Theorem \ref{relpert} depends, at this point, on $e$, as well as on the choice of the canonical homology basis on $X$ (via $\tau$), the map $z$, and the choice of the index $1 \leq j \leq r_{\sigma}$ given to any branching value $\lambda_{\sigma,j}$. Now, the references \cite{[Na]}, \cite{[EiF]}, \cite{[EbF]}, and \cite{[FZ]} have all established the independence of $e$ in the cases they considered. This was also done in the more general reference \cite{[Z]}, but depending on a conjecture about the action of certain operators being transitive. In this section we investigate the dependence on all these parameters, namely we prove that $\alpha_{e}$ is well-defined on coarser and coarser moduli spaces.

We begin with some notation. The group $A$ will be considered as fixed throuought, and so is the sequence $\{r_{\sigma}\}_{Id_{X}\neq\sigma \in A}$ (hence also the genus $g$ by Proposition \ref{genus}), which we denote by $\vec{r}$. We make the following definition.
\begin{defn}
Let $\mathcal{M}_{A,\vec{r}}$ be the moduli (or Torelli) space of Riemann surfaces $X$ that can be presented as an $A$-abelian cover of $\mathbb{CP}^{1}$ in such a manner that for every $Id_{X}\neq\sigma \in A$, the number of points that are attached to $\sigma$ via Proposition \ref{liftcirc} is $r_{\sigma}$. We denote by $\mathcal{M}_{A,\vec{r}}^{Tei}$ the moduli (or Teichm\"{u}ller) space of Riemann surfaces $X$ admitting such an $A$-structure, together with a choice of a canonical homology basis for $X$. We define $\mathcal{M}_{A,\vec{r}}^{Tei,z}$ to be the moduli space of $A$-abelian covers $z:X\to\mathbb{CP}^{1}$ with the required parameters being $\{r_{\sigma}\}_{Id_{X}\neq\sigma \in A}$ (with the choice of the map $z$ as part of the structure) with a canonical homology basis for $X$. We set $\mathcal{M}_{A,\vec{r}}^{Tei,z,ord}$ to be the moduli space of maps $z:X\to\mathbb{CP}^{1}$ as before, with a canonical homology basis on $X$, and with markings indicating which branching value that is associated with $\sigma$ has which index $1 \leq j \leq r_{\sigma}$. Finally, we denote by $\mathcal{M}_{A,\vec{r}}^{z,ord}$ the moduli space of the usual maps $z:X\to\mathbb{CP}^{1}$ with the markings, but without the choice of the canonical homology basis. \label{moduli}
\end{defn}
It is evident from Definition \ref{moduli} that there are surjective maps
\[\mathcal{M}_{A,\vec{r}}^{Tei,z,ord}\to\mathcal{M}_{A,\vec{r}}^{Tei,z}\to\mathcal{M}_{A,\vec{r}}^{Tei}\to\mathcal{M}_{A,\vec{r}}\quad\mathrm{as\ well\ as}\quad\mathcal{M}_{A,\vec{r}}^{Tei,z,ord}\to\mathcal{M}_{A,\vec{r}}^{z,ord}\] and that in each step a group acts transitively on the fibers, generically freely and transitively. This group is the product $\prod_{Id_{X}\neq\sigma \in A}S_{r_{\sigma}}$ of symmetric groups for the first map, the automorphism group $PSL_{2}(\mathbb{C})$ of $\mathbb{CP}^{1}$ on the middle one, and the discrete symplectic group $Sp_{2g}(\mathbb{Z})$ on the last arrow of the sequence as well as on the arrow of the additional map. One can verify, e.g., by Proposition \ref{fibprod} and Remark \ref{normFl}, that $\mathcal{M}_{A,\vec{r}}^{z,ord}$ is isomorphic to the open subset of $(\mathbb{CP}^{1})^{\sum_{Id_{X}\neq\sigma \in A}r_{\sigma}}$ consisting of those points all of whose coordinates are distinct. Since $PSL_{2}(\mathbb{C})$ is 3-dimensional, we find that $\mathcal{M}_{A,\vec{r}}^{z,ord}$, $\mathcal{M}_{A,\vec{r}}^{Tei,z,ord}$, and $\mathcal{M}_{A,\vec{r}}^{Tei,z}$ have dimension $\sum_{Id_{X}\neq\sigma \in A}r_{\sigma}$, while the dimension of $\mathcal{M}_{A,\vec{r}}^{Tei}$ and of $\mathcal{M}_{A,\vec{r}}$ is $\sum_{Id_{X}\neq\sigma \in A}r_{\sigma}-3$. This description also shows the existence of the moduli spaces from Definition \ref{moduli} (though perhaps only as orbifolds or stacks).

Recalling from Remark \ref{minpow} that the powers in Theorem \ref{relpert} are even, we deduce from our constructions the following consequence.
\begin{cor}
The number $\alpha_{e}$ is a well-defined constant on the open subset of $\mathcal{M}_{A,\vec{r}}^{Tei,z,ord}$ defined by $z$ not being branched over $\infty$. \label{alphaemod}
\end{cor}
Note that with the choice of a canonical homology basis and of a marking, the parameter $e=u(\Delta)+K$ for $\Delta$ as in Theorem \ref{nonspdiv}, or more precisely its real counterparts $\varepsilon$ and $\delta$ from $\mathbb{Q}^{g}$, are well-defined in $\mathbb{Q}^{g}/(2\mathbb{Z})^{g}$. This is important for the well-definition of the power of $\theta[e](0,\tau)$ in Theorem \ref{relpert}, hence of $\alpha_{e}$ on the open subset of $\mathcal{M}_{A,\vec{r}}^{Tei,z,ord}$ from Corollary \ref{alphaemod}. This interpretation of the value of $e$ makes it not yet possible to simply assert that if $\alpha_{e}$ is constant on $\mathcal{M}_{A,\vec{r}}^{Tei,z,ord}$ then it is also a well-defined constant on quotient spaces, since the meaning of $e$ has to be clarified on these quotient spaces.

The first step in the desired direction uses the following result.
\begin{prop}
Let $e=u(\Delta)+K$ and $\epsilon=u(\Xi)+K$ be characteristics, where $\Delta$ and $\Xi$ are  divisors that are given in terms of Equation \eqref{normdiv} and satisfy the conditions of Theorem \ref{nonspdiv}. Assume that for each $\sigma \in A$ and $0 \leq \beta<o(\sigma)$, the sets $\{j|\beta_{\sigma,j}^{\Delta}=\beta\}$ and $\{j|\beta_{\sigma,j}^{\Xi}=\beta\}$ have the same cardinality. Then $\alpha_{e}=\alpha_{\epsilon}$. \label{invecard}
\end{prop}

\begin{proof}
If $\gamma$ is a continuous path in the open subset of $\mathcal{M}_{A,\vec{r}}^{Tei,z,ord}$ from Corollary \ref{alphaemod} then $\alpha_{e}$ is constant along $\gamma$. On the other hand, $\det C$ is well-defined on $\mathcal{M}_{A,\vec{r}}^{Tei,z}$, so that if the image of $\gamma$ in $\mathcal{M}_{A,\vec{r}}^{Tei,z}$ is a closed path, then $(\det C)^{2m}$ will have the same value at the initial and final points of $\gamma$. Moreover, since all the symmetric groups are generated by transpositions, it suffices to consider the case where $\Xi$ is obtained from $\Delta$ by interchanging two values $\lambda_{\sigma,j}$ and $\lambda_{\sigma,k}$ (i.e., where $\beta_{\sigma,j}^{\Xi}=\beta_{\sigma,k}^{\Delta}$, $\beta_{\sigma,k}^{\Xi}=\beta_{\sigma,j}^{\Delta}$, and $\beta_{\rho,i}^{\Xi}=\beta_{\rho,i}^{\Delta}$ wherever $\rho\neq\sigma$ or $\rho=\sigma$ and $i\not\in\{j,k\}$). We recall that the relations $e=u(\Delta)+K$ and $\epsilon=u(\Xi)+K$ remain constant along our path in $\mathcal{M}_{A,\vec{r}}^{Tei,z,ord}$.

We therefore construct, for such $\Delta$ and $\Xi$, a path $\gamma$ as a concatenation of three continuous paths, which we write explicitly in the space $\mathcal{M}_{A,\vec{r}}^{z,ord}$ (with the coordinates described above), but consider as continuous lifts into $\mathcal{M}_{A,\vec{r}}^{Tei,z,ord}$. Take a point $\mu\in\mathbb{C}$ that is not a branching value, and consider first three continuous paths, $\delta_{s}:[0,1]\to\mathbb{C}$ with $1 \leq s\leq3$, such that \[\delta_{1}(0)=\lambda_{\sigma,j},\ \delta_{1}(1)=\delta_{2}(0)=\mu,\ \delta_{2}(1)=\delta_{3}(0)=\lambda_{\sigma,k},\mathrm{\ and\ }\delta_{3}(1)=\lambda_{\sigma,j},\] and such that $\delta_{s}(t)$ is neither a branch point nor equals $\mu$ for any $0<t<1$ and $1 \leq s\leq3$. Recall that the open subset of $\mathcal{M}_{A,\vec{r}}^{Tei,z,ord}$ on which $z$ is not ramified over $\infty$ maps onto an open subset of $\mathbb{C}^{\sum_{Id_{X}\neq\sigma \in A}r_{\sigma}}$ (with all the coordinates being different), and our initial point in $\mathcal{M}_{A,\vec{r}}^{z,ord}$ has two coordinates $(\lambda_{\sigma,j},\lambda_{\sigma,k})$, together with all the other coordinates $\lambda_{\rho,i}$. We then define, in the two former coordinates, the paths sending $t\in[0,1]$ to
\[\tilde{\gamma}_{1}(t)=\big(\delta_{1}(t),\lambda_{\sigma,k}\big),\quad\tilde{\gamma}_{2}(t)=\big(\mu,\delta_{3}(t)\big),\quad\mathrm{and}\quad\tilde{\gamma}_{3}(t)=\big(\delta_{2}(t),\lambda_{\sigma,j}\big)\] in $\mathcal{M}_{A,\vec{r}}^{z,ord}$, with the other entries $\lambda_{\rho,i}$ being fixed throughout, and define $\gamma_{s}$ to be a continuous lift of $\tilde{\gamma}_{s}$ (such that $\gamma_{2}(0)=\gamma_{1}(1)$ and $\gamma_{3}(0)=\gamma_{2}(1)$, of course). Then the final point $\gamma_{3}(1)$ is obtained from the initial point $\gamma_{1}(0)$ by the transposition interchanging $\lambda_{\sigma,j}$ and $\lambda_{\sigma,k}$, so that the concatenation $\gamma$ of the $\gamma_{s}$s maps onto a closed path in $\mathcal{M}_{A,\vec{r}}^{Tei,z}$.

Now, Corollary \ref{alphaemod} implies that $\alpha_{e}$ is constant along $\gamma$. But as the effect of moving along $\gamma$ interchanges $\lambda_{\sigma,j}$ and $\lambda_{\sigma,k}$, the theta constant and the polynomial that arise from $\Delta$ and $e$ at the point $\gamma_{3}(1)$ are the same as those arising from $\Xi$ and $\epsilon$ at the point $\gamma_{1}(0)$. Since the value of $(\det C)^{2m}$ was seen to be the same on these two points, Theorem \ref{relpert} implies that $\alpha_{e}$ and $\alpha_{\epsilon}$ can be evaluated as the same ratio, whence their equality. This proves the proposition.
\end{proof}

\begin{cor}
The constants $\alpha_{e}$ from Theorem \ref{relpert} are well-defined on the open subset of $\mathcal{M}_{A,\vec{r}}^{Tei,z}$ where $z$ is not branched over $\infty$. \label{noord}
\end{cor}

\begin{proof}
Consider a divisor $\Delta$ as in Theorem \ref{nonspdiv} such that $e=u(\Delta)+K$ at some point of $\mathcal{M}_{A,\vec{r}}^{Tei,z,ord}$. Then the corresponding characteristic at a different point of $\mathcal{M}_{A,\vec{r}}^{Tei,z,ord}$ with the same image in $\mathcal{M}_{A,\vec{r}}^{Tei,z}$ is obtained by permuting the values $\lambda_{\sigma,j}$, $1 \leq j \leq r_{\sigma}$ for every $Id_{X}\neq\sigma \in A$. It therefore coincides with the point arising from a different divisor $\Xi$ at the initial point of $\mathcal{M}_{A,\vec{r}}^{Tei,z,ord}$, and the sets $\{j|\beta_{\sigma,j}^{\Delta}=\beta\}$ and $\{j|\beta_{\sigma,j}^{\Xi}=\beta\}$ have the same cardinality for every such $\sigma$ and $\beta$. As Proposition \ref{invecard} implies that $\alpha_{e}$ equals $\alpha_{\epsilon}$ for $\epsilon=u(\Xi)+K$, these constants attains the same value on both points of $\mathcal{M}_{A,\vec{r}}^{Tei,z,ord}$, and is therefore well-defined on the quotient $\mathcal{M}_{A,\vec{r}}^{Tei,z}$. This proves the corollary.
\end{proof}

\smallskip

The next result that we prove is as follows.
\begin{prop}
The constants $\alpha_{e}$ extend to the full moduli space $\mathcal{M}_{A,\vec{r}}^{Tei,z}$, and reduce to well-defined expressions on $\mathcal{M}_{A,\vec{r}}^{Tei}$. \label{infzinv}
\end{prop}

\begin{proof}
Constants clearly extend from a dense open set to the whole space, but one can also verify that all our evaluations remain valid when one branching value, say $\lambda_{\sigma,j}$, is infinite, provided that all the terms involving this branching value are omitted in the calculations. As for the dependence on $z$, we need to show that replacing $z$ by any of its $\mathrm{PSL}_{2}(\mathbb{C})$-images yields the same constant, and it suffices to consider translations, complex dilations, and the inversion. But the translation to $z+\mu$ and the dilation to $\kappa z$ simply replaces each $\lambda_{\sigma,j}$ by $\lambda_{\sigma,j}+\mu$ and $\kappa\lambda_{\sigma,j}$ respectively (where the first leaves each of the functions $y_{\chi}$ invariant and the latter divides it by $\kappa^{t_{\chi}}$), with the same choice of homology basis on $X$. Consider the paths sending $0 \leq t\leq1$ to $\{\lambda_{\sigma,j}+t\mu\}_{\sigma,j}$ or to $t\mapsto\{e^{t\ln\kappa}\lambda_{\sigma,j}\}_{\sigma,j}$ in $\mathcal{M}_{A,\vec{r}}^{z,ord}$ (for some choice of $\ln\kappa$), project them onto $\mathcal{M}_{A,\vec{r}}^{z}$, and lift them to $\mathcal{M}_{A,\vec{r}}^{Tei,z}$ continuously. Since $\alpha_{e}$ is constant along these paths, this proves the invariance under translations and dilations. For the inversion, note that we have equalities in $\mathrm{PSL}_{2}(\mathbb{C})$ with two inversions on the left hand side and only one on the right hand side, like
\[\textstyle{\big(\begin{smallmatrix}0 & -1 \\ 1 & 0\end{smallmatrix}\big)\big(\begin{smallmatrix}1 & -1 \\ 0 & 1\end{smallmatrix}\big)\big(\begin{smallmatrix}0 & -1 \\ 1 & 0\end{smallmatrix}\big)=\big(\begin{smallmatrix}1 & 0\\ 1 & 1\end{smallmatrix}\big)=\big(\begin{smallmatrix}1 & 1 \\ 0 & 1\end{smallmatrix}\big)\big(\begin{smallmatrix}0 & -1 \\ 1 & 0\end{smallmatrix}\big)\big(\begin{smallmatrix}1 & 1 \\ 0 & 1\end{smallmatrix}\big)}.\] As we can ignore translations, we find that the action of one inversion on $\alpha_{e}$ is the same as the action of two inversions, the latter being trivial as well since the inversion is an involution. This proves the proposition.
\end{proof}

The extension of $\alpha_{e}$ to the case of infinite $\lambda_{\sigma,j}$ by omitting all its occurrences in the calculation can also be obtained by verifying that the terms depending on $\lambda_{\sigma,j}$ in the polynomial from Theorem \ref{relpert} combine to $\lambda_{\sigma,j}^{2m\sum_{\chi\in\widehat{A}}(t_{\chi}-1)u_{\chi,\sigma}/o(\sigma)}$ times $1+O\big(\frac{1}{\lambda_{\sigma,j}}\big)$ as $\lambda_{\sigma,j}\to\infty$, while $(\det C)^{2m}$ expands as a non-zero constant times the inverse power of $\lambda_{\sigma,j}$ with a similar error term. The invariance under translations can also be seen by continuously moving the paths representing the $a_{i}$s and comparing $\det C$ with the determinant of a matrix similar to $B_{\sigma,j}$ from Lemma \ref{lndetCder}. For the dilations, one moves the paths again, the polynomial from Theorem \ref{relpert} is multiplied by $\kappa^{\sum_{\chi\in\widehat{A}}t_{\chi}(t_{\chi}-1)/2}$, and $\det C$ is divided by the same power of $\kappa$ because of the effect on $z$, $dz$, and the branching values. Since the inversion might affect the paths $a_{i}$ also in the homology of $X$, the involutive argument is indeed required in the proof of Proposition \ref{infzinv}.

\smallskip

Recall that the group $Sp_{2g}(\mathbb{Z})$, a typical element $M$ of which can be written as $\big(\begin{smallmatrix}E & F \\ G & H\end{smallmatrix}\big)$, has an almost linear action on (rational or real) characteristics modulo $(2\mathbb{Z})^{2g}$, with respect to which theta functions with characteristics transform. Let $Sp_{2g}^{ev}(\mathbb{Z})$ be the subgroup consisting of those matrices in which the diagonal entries of the symmetric matrices $EF^{t}$ and $GH^{t}$ are even. This group acts linearly on characteristics modulo $2\mathbb{Z}^{2g}$, by $M$ as above sending $\binom{\varepsilon}{\delta}$ to $\binom{H\varepsilon-G\delta}{E\delta-F\varepsilon}$. Now, Corollary \ref{torsion} shows that our characteristics $e$ are of the form $\tau\frac{\varepsilon}{2}+I\frac{\delta}{2}$ with $\varepsilon$ and $\delta$ from $\big(\frac{1}{m}\mathbb{Z}\big)^{g}$, and if $M$ as above stabilizes $\binom{\varepsilon}{\delta}$ modulo $(2\mathbb{Z})^{2g}$ then we have the equality
\begin{equation}
\theta^{4m}[e](0,M\tau)=\det(G\tau+H)^{2m}\theta^{4m}[e](0,\tau), \label{modularity}
\end{equation}
since the root of unity involved becomes trivial when raised to the power $4m$ in this case. We denote by $\Gamma_{e}$ the stabilizer of $e$ modulo $(2\mathbb{Z})^{2g}$ in $Sp_{2g}^{ev}(\mathbb{Z})$, and as $\mathcal{M}_{A,\vec{r}}$ is the image of $\mathcal{M}_{A,\vec{r}}^{Tei}$ modulo $Sp_{2g}(\mathbb{Z})$, we get the following result.
\begin{prop}
The constant $\alpha_{e}$ is well-defined on the finite cover of $\mathcal{M}_{A,\vec{r}}$ obtained by dividing $\mathcal{M}_{A,\vec{r}}^{Tei}$ by $\Gamma_{e}$. \label{alphatau}
\end{prop}

\begin{proof}
By Equation \eqref{modularity}, replacing $\tau$ by $M\tau$, where $M\in\Gamma_{e}$ is written as $\big(\begin{smallmatrix}E & F \\ G & H\end{smallmatrix}\big)$ as above, multiplies the left hand side of Theorem \ref{relpert} by $\det(G\tau+H)^{2m}$. The polynomial from that theorem does not depend on the choice of homology basis. Writing the basis that we used for obtaining $\tau$ as $\vec{a}=(a_{i})_{i=1}^{g}$ and $\vec{b}=(b_{i})_{i=1}^{g}$, the matrix $M\tau$ is obtained by replacing $\binom{\vec{b}}{\vec{a}}$ by $M\binom{\vec{b}}{\vec{a}}$. Recall that for $\tau$ the entries of the matrix $C$ involve the integrals of the differentials from Proposition \ref{holdif} with respect to the homology elements of $\vec{a}$, but the entries of the matrix $\widetilde{C}$ associated with $M\tau$ are integrals of the same differentials but with respect to the homology elements of $G\vec{b}+H\vec{a}$. As $C$ transfers the basis $\{v_{s}\}_{s=1}^{g}$ to our basis and $\tau_{is}=\int_{b_{i}}v_{s}$, we deduce that $\widetilde{C}=C\cdot(G\tau+H)$. Hence the action of $M$ multiplies $(\det C)^{2m}$ by $\det(G\tau+H)^{2m}$ as well, and it total does not affect the value of $\alpha_{e}$. This proves the proposition.
\end{proof}

\smallskip

The question of the dependence of the constants $\alpha_{e}$ on the characteristic $e$ was addressed, for general fully ramified $Z_{n}$ curves, in \cite{[Z]} (with many special cases treated in \cite{[Na]}, \cite{[EiF]}, \cite{[EbF]}, \cite{[EG2]}, and \cite{[FZ]}). For combining our exponents with those from that reference, we need the following property.
\begin{lem}
For every $d$, $h$, and $\xi$ as in Equation \eqref{phihd} we have the equality $\phi_{h+d\mathbb{Z}}(\xi+h)=\phi_{h+d\mathbb{Z}}(\xi)+\xi-\frac{d-1}{2}$, where we assume that $0\leq\xi \leq d-1$. \label{phifcomp}
\end{lem}

\begin{proof}
Substituting into Equation \eqref{phihd} and recalling the proof of Proposition \ref{expval}, we find that $\phi_{h+d\mathbb{Z}}(\xi+h)-\phi_{h+d\mathbb{Z}}(\xi)$ equals
\[\sum_{0 \neq k\in\mathbb{Z}/d\mathbb{Z}}\frac{-\mathbf{e}(k\xi/d)}{1-\mathbf{e}(-k/d)}=\sum_{0 \neq k\in\mathbb{Z}/d\mathbb{Z}}\bigg[\sum_{v=1}^{d-1}\frac{v}{d}\mathbf{e}\big(\tfrac{-kv}{d}\big)-\sum_{v=d-\xi}^{d-1}\mathbf{e}\big(\tfrac{-kv}{d}\big)\bigg]\] (since the term $1-\mathbf{e}(kh/d)$ cancels). But as $\sum_{0 \neq k\in\mathbb{Z}/d\mathbb{Z}}\mathbf{e}\big(\tfrac{-kv}{d}\big)=-1$ for every $v$ here, the difference in question reduces to $\sum_{v=d-\xi}^{d-1}1-\sum_{v=1}^{d-1}\frac{v}{d}=\xi-\frac{d-1}{2}$ as desired. This proves the lemma.
\end{proof}

\begin{rmk}
Lemma \ref{phifcomp} implies that the function $f_{h}^{(d)}$ from \cite{[Z]} must send $0\leq\xi \leq d-1$ to $2\phi_{h+d\mathbb{Z}}(0)-2\phi_{h+d\mathbb{Z}}(\xi)$. But altering the exponent of $\lambda_{\sigma,j}-\lambda_{\rho,i}$ by a constant depending only on $\sigma$ and $\rho$ (but not $i$ and $j$) does not affect the invariance from \cite{[Z]}, and in the fully ramified cyclic case considered in that reference we have $o(\sigma)=m=n$ for every $\sigma$ with $r_{\sigma}>0$. Moreover, the power from Theorem \ref{relpert}, which is $4n$ in the cyclic case, divides the exponent $2en^{2}$ from \cite{[FZ]} and \cite{[Z]}. It follows that the results of \cite{[Z]} can be combined with our Theorem \ref{relpert} and its consequences in this case, which implies that $\alpha_{e}^{n}$ (and even a slightly smaller power of $\alpha_{e}$) is independent of the choice of $e$. \label{compwithZ}
\end{rmk}

\smallskip

%The relation of the functions $\phi_{h+d\mathbb{Z}}$ from Equation \eqref{phihd} and the classical Dedekind sums $s(h,d)$ is given by the formula $\phi_{h+d\mathbb{Z}}(0)=d \cdot s(h,d)+\frac{d-1}{4}$, which one easily verifies, e.g., using the proof of Lemma \ref{Dedekind}. The recursive formula for Dedekind sums (see, e.g., Equation (3) of \cite{[Rd]}) translates to $\phi_{h+d\mathbb{Z}}(0)=\frac{d^{2}+h^{2}+3hd-3d-3h+1}{12h}-\frac{d}{h}\phi_{d+h\mathbb{Z}}(0)$ (when $h$ is considered as lying between 0 and $d$), and combining this with the recursive formula for $f_{h}^{(d)}(\xi)$ appearing in Theorem 6.4 of \cite{[Z]} we obtain \[\phi_{h+d\mathbb{Z}}(\xi)=\frac{d^{2}+h^{2}+3hd-3d-3h+1-6\xi(d+h-1-\xi)}{12h}-\frac{d}{h}\phi_{d+h\mathbb{Z}}(\xi)\] (provided that $0 \leq \xi<d$ as well), where in the rightmost summand we may replace both $d$ and $\xi$ by their residues modulo $h$. This gives a recursive argument for evaluating $\phi_{h+d\mathbb{Z}}(\xi)$ for every $d$, $h$, and $\xi$. Another feature of $\phi_{h+d\mathbb{Z}}(\xi)$ in comparison to $s(h,d)$ is that the equality $\sum_{\xi\in\mathbb{Z}/d\mathbb{Z}}\phi_{h+d\mathbb{Z}}(\xi)=0$ holds for every $d$ and $h$, and one can also verify (directly from Equation \eqref{phihd}, say) that $\phi_{-h+d\mathbb{Z}}(\xi)=-\phi_{h+d\mathbb{Z}}(d-1-\xi)$ for every $d$, $h$, and $\xi$.

Recall from Remark \ref{compwithZ} that the constant $\alpha_{e}^{n}$ does not depend on $e$ in the fully ramified cyclic case (provided that the conjecture from \cite{[Z]} holds). Proposition \ref{invecard} also compares constants arising from different characteristics on general abelian covers of the sphere. In addition, the proof of Proposition \ref{alphatau} shows that elements of $Sp_{2g}(\mathbb{Z})\setminus\Gamma_{e}$ operate on $\det C$ in the same manner, while they take $e$ to a different characteristic. As the root of unity hiding in Equation \eqref{modularity} is of globally bounded order, such elements would take the appropriate power of $\alpha_{e}$ (on the moduli space $\mathcal{M}_{A,\vec{r}}$ itself, say) to an expression arising from a different characteristic, of the same order, and whose theta function does not vanish as well. If one can show that this new characteristic is the characteristic arising from a different divisor $\Delta$, then this extension of Proposition \ref{alphatau} would relate $\alpha_{e}$ to constants associated with even more of our divisors and characteristics. Based on these results we pose the following conjecture.
\begin{conj}
There is some integral power $N$, depending on $n$ and $m$ (and maybe the structure of the 2-Sylow subgroup of $A$ in some special cases), such that the constant $\alpha_{e}^{N}$ takes the same value for every characteristic $e$. \label{InvofDelta}
\end{conj}
This global constant $\alpha^{N}$ from Conjecture \ref{InvofDelta}, namely $\alpha_{e}^{N}$ for every $e$, would then depend only on $A$, $\vec{r}$, and the normalization of the functions $y_{\chi}$. Because of the latter, explicit statements regarding its value are would not be easy to make and prove. On the other hand, restricting attention to maps $z:X\to\mathbb{CP}^{1}$ that are defined over a subfield $\mathbb{F}$ of $\mathbb{C}$ containing enough roots of unity (e.g., a number field), as well as to normalizations of the $y_{\chi}$s that are defined over $\mathbb{F}$, the value of $\alpha^{N}$ in $\mathbb{C}^{\times}/\mathbb{F}^{\times}$ will be an invariant of $A$ and $\vec{r}$ alone (over $\mathbb{F}$). Moreover, the polynomials from Theorem \ref{relpert} lie in $\mathbb{F}^{\times}$ in that case, so that we will have a global constant that relates (powers of) theta constants and periods. It may then be worthwhile to investigate the meaning of this expression in terms of the algebraic geometry of the moduli space $\mathcal{M}_{A,\vec{r}}$. We remark that the tools applied in both \cite{[Na]} and \cite{[Z]} (as well as in the references of the latter) inherently use the possibility to add a single branch point to an invariant divisor and obtain another invariant divisor. Since this happens only in the case where the inequality $r_{\sigma}>0$ implies $\langle\sigma\rangle=A$, these tools do not work outside the fully ramified cyclic case. A proof of Conjecture \ref{InvofDelta} may thus require some new tools, and will therefore be left for future research.

\section{Some Explicit Examples \label{Examples}}

In this section we reproduce some results from the literature as examples of Theorem \ref{relpert}. More explicitly, we consider the fully ramified cyclic case from \cite{[EG2]} or Chapter 5 of \cite{[FZ]} (containing some other cases from earlier references), and the case where the group $A$ has exponent 2. For the cyclic case in question, note that our Dedekind sum $\phi_{h+d\mathbb{Z}}(0)$ equals (e.g., by the proof of Lemma \ref{Dedekind}) the expression $d \cdot s(h,d)+\frac{d-1}{4}$, where $s(h,d)$ is the classical Dedekind sum. Without going into the deeper analysis of the Dedekind sums $\phi_{h+d\mathbb{Z}}(\xi)$, we state that wherever $0\leq\xi \leq d-1$ we have
\begin{equation}
\phi_{1+d\mathbb{Z}}(\xi)=\tfrac{d^{2}-1-6\xi(d-\xi)}{12}\qquad\mathrm{and}\qquad\phi_{-1+d\mathbb{Z}}(\xi)=\tfrac{6d-d^{2}-5+6\xi(d-2-\xi)}{12} \label{phipm1d}
\end{equation}
(the first one follows, e.g., from the recursive formula for Dedekind sums appearing in Equation (3) of \cite{[Rd]} together with Lemma \ref{phifcomp}, and the second one by the complementary equation $\phi_{-h+d\mathbb{Z}}(\xi)=-\phi_{h+d\mathbb{Z}}(d-1-\xi)$).

In our terminology, \cite{[EG2]} and Chapter 5 of \cite{[FZ]} consider the case where $A$ is cyclic, $\sigma$ is a generator, the only positive parameters $r_{\rho}$ are $r_{\sigma}$ and $r_{\sigma^{-1}}$, and they are equal. Equation \eqref{phipm1d} is in correspondence, via Proposition \ref{expval}, with the equalities $\gamma_{\sigma,\sigma}=\frac{(n-1)(2n-1)}{6n^{2}}$ and $\gamma_{\sigma,\sigma^{-1}}=\frac{(n-1)(n+1)}{6n^{2}}$, and with similar evaluations for the terms $q_{e}$. Denoting the parameter $\xi$ by $k$ for pairs $(\lambda_{\sigma,j},\lambda_{\sigma,i})$ or $(\lambda_{\sigma^{-1},j},\lambda_{\sigma^{-1},i})$ but setting it to be $n-1-k$ for $(\lambda_{\sigma,j},\lambda_{\sigma^{-1},i})$, we find that
\[2\phi_{1+n\mathbb{Z}}(k)+\phi_{1+n\mathbb{Z}}(0)+\tfrac{(n-1)^{2}}{4}=\tfrac{n^{2}-1}{4}-k(n-k)+\tfrac{(n-1)^{4}}{4}=\tfrac{n(n-1)}{2}-k(n-k),\] as well as
\[2\phi_{-1+n\mathbb{Z}}(n-1-k)+\phi_{-1+n\mathbb{Z}}(0)+\tfrac{(n-1)^{2}}{4}=\tfrac{6n-n^{2}-5}{4}+(k-1)(n-1-k)+\tfrac{(n-1)^{4}}{4},\] which reduces to $k(n-k)$. The polynomial from Theorem \ref{relpert} thus becomes
\begin{equation}
\prod_{1\leq j<i \leq r_{\sigma}}\!\!\!\big[(\lambda_{\sigma,j}-\lambda_{\sigma,i})(\lambda_{\sigma^{-1},j}-\lambda_{\sigma^{-1},i})\big]^{n(n-1)-2k(n-k)}
\prod_{j=1}^{r_{\sigma}}\prod_{i=1}^{r_{\sigma}}(\lambda_{\sigma,j}-\lambda_{\sigma^{-1},i})^{2k(n-k)}. \label{EGZ}
\end{equation}
Our formula agrees, up to an expression that is independent of $\Delta$ and $e$, with the ones from \cite{[EG2]} and Chapter 5 of \cite{[FZ]}.

Note that this result, with the appropriate products taken up to $r_{\sigma^{-1}}$, holds also without the equality $r_{\sigma^{-1}}=r_{\sigma}$ (see Section A.7 of \cite{[FZ]}). We may therefore consider the special case where $r_{\sigma^{-1}}=0$ as well in this example. Here only the first product in Equation \eqref{EGZ} remains, and Theorem \ref{relpert} reduces to \[\theta[e]^{4n}(0,\tau)=\alpha_{e}(\det C)^{2n}\prod_{i<j}(\lambda_{\sigma,j}-\lambda_{\sigma,i})^{n(n-1)-2k(n-k)},\mathrm{\ with\ }k=|\beta_{\sigma,j}^{\Delta}-\beta_{\sigma,i}^{\Delta}|.\] All the exponents in Equation \eqref{EGZ} are even integers, and in correspondence with the minimality from Remark \ref{minpow} (with $\frac{n}{m}=1$ in the $\gcd$), we see that not all of them will remain so in general if we replace $4n$ by a proper divisor. Note, however, that if $n\equiv1(\mathrm{mod\ }4)$ then the exponent $2n$ will also produce even exponents here (though $\det C$ would then be raised to the odd power $n$). Since Proposition \ref{invecard} implies that $\alpha_{e}$ is the same constant $\alpha$ for every characteristic $e$ in this case, this reproduces (again up to some global expression) the results of \cite{[Na]}, \cite{[EbF]}, and Chapter 4 of \cite{[FZ]} (with $r_{\sigma}$ being a multiple $nq$ of $n$, and $g=\frac{(nq-2)(n-1)}{2}$). When $n=2$ there are $2q=2g+2$ branch points, which are of the form $P_{j}=z^{-1}(\lambda_{\sigma,j})$ for $1 \leq j\leq2g+2$, and the divisor $\Delta$ from Theorem \ref{nonspdiv} is of the form $\sum_{j \in I}P_{j}-z^{-1}(\infty)$ for a subset $I$ of $\{j\}_{j=1}^{2g+2}$ with cardinality $g+1$. If $I^{c}$ is the complement of $I$ in $\{j\}_{j=1}^{2g+2}$, then we obtain the result of Thomae stated in the Introduction, namely
\begin{equation}
\theta[e]^{8}(0,\tau)=\alpha(\det C)^{4}\prod_{i,j \in I,\ i<j}(\lambda_{\sigma,j}-\lambda_{\sigma,i})^{2}\prod_{i,j \in I^{c},\ i<j}(\lambda_{\sigma,j}-\lambda_{\sigma,i})^{2}. \label{origThomae}
\end{equation}
Note that here both $\chi\Delta$ for the non-trivial character $\chi$ of $A$ and $N\Delta$ yield, via Lemma \ref{actAdual} and Corollary \ref{NonAinv} respectively, the divisor $\sum_{j \in I^{c}}P_{j}-z^{-1}(\infty)$. The dihedral group from Corollary \ref{dihedral} (which is a Klein 4-group in this case) thus no longer acts freely, since $N\chi=\chi N$ stabilizes all the divisors.

\medskip

We now turn to the case of an arbitrary group $A$ of exponent 2. Then $m=o(\sigma)=2$ for every $Id_{X}\neq\sigma \in A$ (so that $\beta_{\sigma,j}^{\Delta}\in\{0,1\}$ in Equation \eqref{normdiv}), the function $\phi_{1+2\mathbb{Z}}$ takes even numbers to $\frac{1}{4}$ and odd numbers to $-\frac{1}{4}$, and $\phi_{0+1\mathbb{Z}}$ vanishes identically. For two points $\lambda_{\sigma,j}$ and $\lambda_{\sigma,i}$ (with the same $\sigma$) we have $d=2$ and $h=1$, the parameter $\xi$ is 0 when $\beta_{\sigma,j}^{\Delta}=\beta_{\sigma,i}^{\Delta}$ and $\pm1$ otherwise, and \[2\phi_{1+2\mathbb{Z}}(\xi)+\phi_{1+2\mathbb{Z}}(0)+\tfrac{(2-1)^{2}}{4}=2\cdot\tfrac{1}{4}\cdot(-1)^{\xi}+2\cdot\tfrac{1}{4}=\delta_{\xi+2\mathbb{Z},2\mathbb{Z}}\] (namely 1 when $\beta_{\sigma,j}^{\Delta}=\beta_{\sigma,i}^{\Delta}$ and 0 in case $\beta_{\sigma,j}^{\Delta}\neq\beta_{\sigma,i}^{\Delta}$). On the other hand, for points $\lambda_{\sigma,j}$ and $\lambda_{\rho,i}$ with $\rho\neq\sigma$ the index $d$ equals 1 and only the last term $\frac{(2-1)^{2}}{4}=\frac{1}{4}$ remains. Equivalently, examining the values of characters on elements implies that $\gamma_{\sigma,\rho}$ is $\frac{1}{8}$ when $\sigma=\rho$ and $\frac{1}{16}$ otherwise, the value of $q_{\Delta}(\sigma,j;\rho,i)$ is $\frac{1}{16}$ when $\beta_{\sigma,j}^{\Delta}=\beta_{\rho,i}^{\Delta}$ and $-\frac{1}{16}$ in case $\beta_{\sigma,j}^{\Delta}\neq\beta_{\rho,i}^{\Delta}$, and the sum $q_{e}(\sigma,j;\rho,i)$ becomes $\frac{n}{16}$ if $\sigma=\rho$ and $\beta_{\sigma,j}^{\Delta}=\beta_{\rho,i}^{\Delta}$, $-\frac{n}{16}$ when $\sigma=\rho$ and $\beta_{\sigma,j}^{\Delta}\neq\beta_{\rho,i}^{\Delta}$, and 0 wherever $\sigma\neq\rho$. Thus $2q_{e}(\sigma,j;\rho,i)+n\gamma_{\sigma,\rho}$ is indeed $\frac{n}{o(\sigma)o(\rho)}=\frac{n}{4}$ times $\delta_{\xi+2\mathbb{Z},2\mathbb{Z}}$ when $\rho=\sigma$ and times $\frac{1}{4}$ otherwise. Recall that $n=2^{b}$ for some $b\in\mathbb{N}$ when $A$ has exponent 2, so that Theorem \ref{relpert} takes, in this case, the following form.
\begin{thm}
Let $f:X\mapsto \mathbb{CP}^{1}$ be an abelian cover of $\mathbb{CP}^{1}$ whose Galois group $A$ has exponent 2 and order $n=2^{b}$. Then for every characteristic $e=u(\Delta)+K$ for a non-special invariant divisor $\Delta$ of degree $g-1$ on $X$, expressed as in Equation \eqref{normdiv}, the expression $\theta[e]^{8}(0,\tau)$ equals
\[\alpha_{e}(\det C)^{4}\prod_{\sigma \neq Id_{X}}\prod_{\{1 \leq j<i \leq r_{\sigma}|\beta_{\sigma,j}^{\Delta}=\beta_{\sigma,i}^{\Delta}\}}(\lambda_{\sigma,j}-\lambda_{\sigma,i})^{2^{b}}\prod_{\sigma<\rho}\prod_{j=1}^{r_{\sigma}}\prod_{i=1}^{r_{\rho}}(\lambda_{\sigma,j}-\lambda_{\rho,i})^{2^{b-2}}\] for some constant $\alpha_{e}$ that is independent of the branching values, where in the last product $<$ denotes some arbitrary order on the non-trivial elements of $A$. \label{polexp2}
\end{thm}

Note that Corollary \ref{torsion} implies that the characteristics $e=u(\Delta)+K$ has order dividing 4, and in general it does not have order 2, since the divisor of $dz$ will not be equal the divisor of any function $y_{\chi}$ (plus some multiple of $z^{-1}(\infty)$) in general.

Theorem \ref{polexp2} expresses the main results of \cite{[Ko3]} in explicit terms, and without the restriction that the fibered product structure there is based on hyperelliptic covers with distinct branching values as in Remark \ref{normFl}. We note that when $b=1$ there is only one non-trivial element of $A$, so that the term involving the fractional exponent $2^{1-2}=\frac{1}{2}$ in that theorem does not appear (and we reproduce Equation \eqref{origThomae} once again). The case $b=2$ is the Klein 4-group case, with the odd exponent $2^{2-2}=1$, in correspondence with Lemma \ref{4meven}. On the other hand, Remark \ref{minpow} shows that for $b\geq4$ we can get relations with $\theta[e]^{4}(0,\tau)$ and $(\det C)^{2}$ with a polynomial with even powers in Theorem \ref{polexp2}, and when $b\geq5$ this can even be done with $\theta[e]^{2}(0,\tau)$ and $\det C$ themselves (but then $\alpha_{e}$ will depend on the column order in $C$).

Let us consider the case $b=2$, where many of our expressions can be written very explicitly. Then $A=\{Id_{X},\sigma,\rho,\sigma\rho\}$ and $\widehat{A}=\{\mathbf{1},\chi,\eta,\chi\eta\}$, where \[\chi(\sigma)=\eta(\rho)=-1,\ \chi(\rho)=\eta(\sigma)=1,\ u_{\chi,\sigma}=u_{\eta,\rho}=1,\mathrm{\ and\ }u_{\chi,\rho}=u_{\eta,\sigma}=0.\] Since the definition in Proposition \ref{decomgen} shows that
\begin{equation}
r_{\sigma}+r_{\sigma\rho}=2t_{\chi},\ r_{\rho}+r_{\sigma\rho}=2t_{\eta},\ r_{\sigma}+r_{\rho}=2t_{\chi\eta},\mathrm{\ and\ hence\ }r_{\sigma} \equiv r_{\rho} \equiv r_{\sigma\rho}(\mathrm{mod\ }2), \label{parity}
\end{equation}
we can view $X$ as the fibered product of the hyperelliptic curves defined by
\[y_{\chi}^{2}=\prod_{j=1}^{2t_{\chi}-r_{\sigma\rho}}(z-\lambda_{\sigma,j})\prod_{j=1}^{r_{\sigma\rho}}(z-\lambda_{\sigma\rho,j})\mathrm{\ and\ }y_{\eta}^{2}=\prod_{j=1}^{2t_{\eta}-r_{\sigma\rho}}(z-\lambda_{\rho,j})\prod_{j=1}^{r_{\sigma\rho}}(z-\lambda_{\sigma\rho,j}).\] The genus $g$ from Proposition \ref{genus} is $r_{\sigma}+r_{\rho}+r_{\sigma\rho}-3=t_{\chi}+t_{\eta}+t_{\chi\eta}-3$. The normalized divisor from Equation \eqref{normdiv} with $\beta$-indices in $\{0,1\}$ is
\begin{equation}
\Delta=\sum_{j \in I_{\sigma}}z^{-1}(\lambda_{\sigma,j})+\sum_{j \in I_{\rho}}z^{-1}(\lambda_{\rho,j})+\sum_{j \in I_{\sigma\rho}}z^{-1}(\lambda_{\sigma\rho,j})-h^{\Delta}z^{-1}(\infty) \label{Kleindiv}
\end{equation}
for sets of indices $I_{\sigma}$, $I_{\rho}$, and $I_{\sigma\rho}$, and the conditions of Theorem \ref{nonspdiv} become \[|I_{\sigma}|+|I_{\sigma\rho}|=t_{\chi},\ |I_{\rho}|+|I_{\sigma\rho}|=t_{\eta},\ |I_{\sigma}|+|I_{\rho}|=t_{\chi\eta}=t_{\chi}+t_{\eta}-r_{\sigma\rho},\mathrm{\ and\ }h^{\Delta}=1,\] where $|I|$ denotes the cardinality of the set $I$. Solving these equations yields \[|I_{\sigma}|=t_{\chi}-\tfrac{r_{\sigma\rho}}{2}=\tfrac{r_{\sigma}}{2},\ |I_{\rho}|=t_{\eta}-\tfrac{r_{\sigma\rho}}{2}=\tfrac{r_{\rho}}{2},\mathrm{\ and\ }|I_{\sigma\rho}|=\tfrac{r_{\sigma\rho}}{2},\] which proves the following result.
\begin{prop}
If $X$ is the fibered product of two hyperelliptic curves such that the number of common branching values is odd then $X$ carries no non-special invariant divisors of degree $g-1$. \label{nodivs}
\end{prop}
Indeed, the condition from Proposition \ref{nodivs} is that $r_{\sigma\rho}$ is odd (hence so are $r_{\sigma}$ and $r_{\rho}$ by Equation \eqref{parity}), and then $I_{\sigma}$, $I_{\rho}$, and $I_{\sigma\rho}$ must have non-integral cardinalities for $\Delta$ to be non-special. In fact, since all the terms in Equation \eqref{Kleindiv} have even degrees, and the condition from Proposition \ref{nodivs} implies that $g-1$ is odd, $X$ carries no invariant divisors of degree $g-1$ at all (normalized or not). Recall that \cite{[GDT]} was concerned with finding $Z_{n}$ curves with no Thomae formulae because of the non-existence of the required divisors. Proposition \ref{nodivs} provides another family of abelian covers of $\mathbb{CP}^{1}$ that lack a Thomae formula, and for the same reason.

We thus henceforth assume that $r_{\sigma\rho}$ is even, and then the non-special divisors are obtained by taking $I_{\sigma}$, $I_{\rho}$, and $I_{\sigma\rho}$ be subsets of $\{j\}_{j=1}^{r}$ (with $r$ being the even number $r_{\sigma}$, $r_{\rho}$, or $r_{\sigma\rho}$) of cardinality $\frac{r}{2}$. Then the action of $\chi$ (resp. $\eta$, resp. $\chi\eta$) on $\Delta$ in Lemma \ref{actAdual} replaces $I_{\rho}$ and $I_{\sigma\rho}$ (resp. $I_{\rho}$ and $I_{\sigma\rho}$, resp. $I_{\sigma}$ and $I_{\rho}$) by their complements (denoted again by the superscript $c$), the action of $N$ from Corollary \ref{NonAinv} replaces all three sets by their complements, and the dihedral group from Corollary \ref{dihedral} is abelian of order 8 and exponent 2. This group acts freely when $r_{\sigma}$, $r_{\rho}$, and $r_{\sigma\rho}$ are all positive, but when one of them vanishes, say $r_{\sigma\rho}$ (and only one can vanish since $X$ is irreducible), then one non-trivial element (when $r_{\sigma\rho}=0$ it is $N\chi\eta=\chi\eta N$) acts trivially. By reducing the notation of products like $\prod_{j \in I_{\sigma}}\big(z(P)-\lambda_{\sigma,j}\big)$ to simply $\big(z(P)-\lambda_{\sigma,I_{\sigma}}\big)$, and setting $z=z(P)$ and $w=z(Q)$ in $S[e](P,Q)$, the expression from Theorem \ref{Szegoalg} becomes $\frac{\sqrt{dz}\sqrt{dw}}{4(z-w)}$ times
\[\Bigg[\frac{(z-\lambda_{\sigma,I_{\sigma}})(z-\lambda_{\rho,I_{\rho}})(z-\lambda_{\sigma\rho,I_{\sigma\rho}})(w-\lambda_{\sigma,I_{\sigma}^{c}})(w-\lambda_{\rho,I_{\rho}^{c}})(w-\lambda_{\rho\sigma,I_{\rho\sigma}^{c}})}
{(z-\lambda_{\sigma,I_{\sigma}^{c}})(z-\lambda_{\rho,I_{\rho}^{c}})(z-\lambda_{\rho\sigma,I_{\rho\sigma}^{c}})(w-\lambda_{\sigma,I_{\sigma}})(w-\lambda_{\rho,I_{\rho}})(z-\lambda_{\rho\sigma,I_{\rho\sigma}})}\Bigg]^{1/4}+\]
\[+\Bigg[\frac{(z-\lambda_{\sigma,I_{\sigma}^{c}})(z-\lambda_{\rho,I_{\rho}})(z-\lambda_{\sigma\rho,I_{\sigma\rho}^{c}})(w-\lambda_{\sigma,I_{\sigma}})(w-\lambda_{\rho,I_{\rho}^{c}})(w-\lambda_{\rho\sigma,I_{\rho\sigma}})}
{(z-\lambda_{\sigma,I_{\sigma}})(z-\lambda_{\rho,I_{\rho}^{c}})(z-\lambda_{\rho\sigma,I_{\rho\sigma}})(w-\lambda_{\sigma,I_{\sigma}^{c}})(w-\lambda_{\rho,I_{\rho}})(z-\lambda_{\rho\sigma,I_{\rho\sigma}^{c}})}\Bigg]^{1/4}+\]
\[+\Bigg[\frac{(z-\lambda_{\sigma,I_{\sigma}})(z-\lambda_{\rho,I_{\rho}^{c}})(z-\lambda_{\sigma\rho,I_{\sigma\rho}^{c}})(w-\lambda_{\sigma,I_{\sigma}^{c}})(w-\lambda_{\rho,I_{\rho}})(w-\lambda_{\rho\sigma,I_{\rho\sigma}})}
{(z-\lambda_{\sigma,I_{\sigma}^{c}})(z-\lambda_{\rho,I_{\rho}})(z-\lambda_{\rho\sigma,I_{\rho\sigma}})(w-\lambda_{\sigma,I_{\sigma}})(w-\lambda_{\rho,I_{\rho}^{c}})(z-\lambda_{\rho\sigma,I_{\rho\sigma}^{c}})}\Bigg]^{1/4}+\]
\begin{equation}
+\Bigg[\frac{(z-\lambda_{\sigma,I_{\sigma}^{c}})(z-\lambda_{\rho,I_{\rho}^{c}})(z-\lambda_{\sigma\rho,I_{\sigma\rho}})(w-\lambda_{\sigma,I_{\sigma}})(w-\lambda_{\rho,I_{\rho}})(w-\lambda_{\rho\sigma,I_{\rho\sigma}^{c}})}
{(z-\lambda_{\sigma,I_{\sigma}})(z-\lambda_{\rho,I_{\rho}})(z-\lambda_{\rho\sigma,I_{\rho\sigma}^{c}})(w-\lambda_{\sigma,I_{\sigma}^{c}})(w-\lambda_{\rho,I_{\rho}^{c}})(z-\lambda_{\rho\sigma,I_{\rho\sigma}})}\Bigg]^{1/4} \label{Szegoexp}
\end{equation}
(from which the symmetric condition $S[e](P,Q)=S[-e](Q,P)$ is also evident).

As for the polynomials in the definition of $\omega$, we have $f_{\chi,1}^{\chi}=\frac{1}{2}(f_{\chi,0}^{\chi})'$ and $f_{\chi,0}^{\chi}=y_{\chi}^{2}$ in this case (and the same for $\eta$ and $\chi\eta$), and the formula for $\omega$ depends on the polynomials with $l\geq2$. For obtaining an explicit expression here as well, we consider the case where our hyperelliptic curves are elliptic, namely $t_{\chi}=t_{\rho}=2$, and we also take $r_{\sigma\rho}=2$ hence $t_{\chi\rho}=2$ for symmetry. The genus is 3, and if we set $s_{\chi}=\lambda_{\sigma,1}+\lambda_{\sigma,2}+\lambda_{\sigma\rho,1}+\lambda_{\sigma\rho,2}$ (this is the coefficient of $-w^{3}$ in the degree 4 polynomial $f_{\chi,0}^{\chi}$ and of $-\frac{3}{2}w^{2}$ in $f_{\chi,1}^{\chi}$) then $f_{\chi,2}^{\chi}(w)=w^{2}-\frac{s_{\chi}}{2}w$ (up to a constant). Define $s_{\eta}$ and $s_{\chi\eta}$ similarly, and since all the polynomials appearing in the difference $\omega-\xi$ in Corollary \ref{omegaxi} are constants, we find that with $z$ and $w$ as above, $\frac{\omega(P,Q)}{dzdw}$ equals \[\frac{1}{4(z-w)^{2}}+\frac{f_{\chi,0}^{\chi}(w)\!+\!\frac{1}{2}(f_{\chi,0}^{\chi})'(w)(z-w)\!+\!\big(w^{2}-\frac{s_{\chi}}{2}w\big)(z-w)^{2}}{4y_{\chi}(P)y_{\chi}(Q)(z-w)^{2}}+\frac{c_{\chi}^{\chi}}{y_{\chi}(P)y_{\chi}(Q)}+\] \[+\frac{f_{\eta,0}^{\eta}(w)+\frac{1}{2}(f_{\eta,0}^{\eta})'(w)(z-w)+\big(w^{2}-\frac{s_{\eta}}{2}w\big)(z-w)^{2}}{4y_{\eta}(P)y_{\eta}(Q)(z-w)^{2}}+\frac{c_{\eta}^{\eta}}{y_{\eta}(P)y_{\eta}(Q)}+\]
\[+\frac{f_{\chi\eta,0}^{\chi\eta}(w)+\frac{1}{2}(f_{\chi\eta,0}^{\chi\eta})'(w)(z-w)+\big(w^{2}-\frac{s_{\chi\eta}}{2}w\big)(z-w)^{2}}{4y_{\chi\eta}(P)y_{\chi\eta}(Q)(z-w)^{2}}+
\frac{c_{\chi\eta}^{\chi\eta}}{y_{\chi\eta}(P)y_{\chi\eta}(Q)}+\]
\begin{equation}
+\tfrac{c_{\chi}^{\eta}}{y_{\chi}(P)y_{\eta}(Q)}+\tfrac{c_{\chi}^{\chi\eta}}{y_{\chi}(P)y_{\chi\eta}(Q)}+\tfrac{c_{\eta}^{\chi}}{y_{\eta}(P)y_{\chi}(Q)}+\tfrac{c_{\eta}^{\chi\eta}}{y_{\eta}(P)y_{\chi\eta}(Q)}
\tfrac{c_{\chi\eta}^{\chi}}{y_{\chi\eta}(P)y_{\chi}(Q)}+\tfrac{c_{\chi\eta}^{\eta}}{y_{\chi\eta}(P)y_{\eta}(Q)} \label{difexp}
\end{equation}
for some constants $c_{\chi}^{\chi}$, $c_{\eta}^{\eta}$, $c_{\chi\eta}^{\chi\eta}$, $c_{\chi}^{\eta}$, $c_{\chi}^{\chi\eta}$, $c_{\eta}^{\chi}$, $c_{\eta}^{\chi\eta}$, $c_{\chi\eta}^{\chi}$, and $c_{\chi\eta}^{\eta}$. Note that if $f_{\chi,0}^{\chi}(w)=w^{4}-s_{1,\chi}w^{3}+s_{2,\chi}w^{2}-s_{3,\chi}w+s_{4,\chi}$ (with $s_{1,\chi}$ being $s_{\chi}$ from above and the other coefficients are the next symmetric functions in $\lambda_{\sigma,1}$, $\lambda_{\sigma,2}$, $\lambda_{\sigma\rho,1}$, and $\lambda_{\sigma\rho,2}$) then the numerator over $4y_{\chi}(P)y_{\chi}(Q)(z-w)^{2}$ in Equation \eqref{difexp} is the symmetric expression
\begin{equation}
z^{2}w^{2}-\tfrac{s_{1,\chi}}{2}(z^{2}w+zw^{2})+s_{2,\chi}zw-\tfrac{s_{3,\chi}}{2}(z+w)+s_{4,\chi}. \label{symnum}
\end{equation}
The same happens with $\eta$ and $\chi\eta$, and the symmetry of $\omega$ implies that the constants from above must satisfy $c_{\chi}^{\eta}=c_{\eta}^{\chi}$, $c_{\chi}^{\chi\eta}=c_{\chi\eta}^{\chi}$, and $c_{\eta}^{\chi\eta}=c_{\chi\eta}^{\eta}$.

We conclude by remarking that the description of the non-special divisors in terms of sets $I_{\sigma}$ of cardinality $\frac{r_{\sigma}}{2}$, and the actions of $N$ and $\widehat{A}$, extend to any value of $b$ (i.e., any group $A$ of exponent 2). Hence analogues of Proposition \ref{nodivs} can be proved, though the parity conditions are more complicated for $b\geq3$. The formula for $S[e]=F_{e}$ is similar to Equation \eqref{Szegoexp} when $b\geq3$ (but longer), and the formula $f_{\chi,1}^{\chi}=\frac{1}{2}(f_{\chi,0}^{\chi})'$ remains valid, but the only case with $b\geq3$ where $t_{\chi}=2$ for every $\chi$ is when $b=3$ and $r_{\sigma}=1$ for every $\sigma$ (and then there are no divisors again, like in Proposition \ref{nodivs}). On the other hand, in the cyclic case of $b=1$ we have $A=\{Id_{X},\sigma\}$, $\widehat{A}=\{\mathbf{1},\chi\}$, $r_{\sigma}=2g+2$, and $t_{\chi}=g+1$. Then we have seen the form of the divisor $\Delta$ and the action of $\chi$ and $N$, and $S[e](P,Q)$ is given in the simplification of Equation \eqref{Szegoexp} that looks like \[\frac{\sqrt{dz}\sqrt{dw}}{2(z-w)}\Bigg[\frac{(z-\lambda_{\sigma,I})^{1/4}(w-\lambda_{\sigma,I^{c}})^{1/4}}{(z-\lambda_{\sigma,I^{c}})^{1/4}(w-\lambda_{\sigma,I})^{1/4}}+
\frac{(z-\lambda_{\sigma,I^{c}})^{1/4}(w-\lambda_{\sigma,I})^{1/4}}{(z-\lambda_{\sigma,I})^{1/4}(w-\lambda_{\sigma,I^{c}})^{1/4}}\Bigg].\] Recalling that $y_{\chi}^{2}=\prod_{j=1}^{2g+2}(z-\lambda_{\sigma,j})$, the simplification of Equation \eqref{difexp} shows that when $g=1$ there is a constant $c_{\chi}^{\chi}$ such that $\frac{\omega(P,Q)}{dzdw}$ equals \[\frac{1}{4(z-w)^{2}}+\frac{f_{\chi,0}^{\chi}(w)+\frac{1}{2}(f_{\chi,0}^{\chi})'(w)(z-w)+\!\big(w^{2}-\frac{s_{\chi}}{2}w\big)(z-w)^{2}}{4y_{\chi}(P)y_{\chi}(Q)(z-w)^{2}}+\frac{c_{\chi}^{\chi}}{y_{\chi}(P)y_{\chi}(Q)},\] with the numerator in the middle being the symmetric expression from Equation \eqref{symnum} as well.

\noindent\textsc{Finance Department School of Business 2100, Hillside University of Connecticut, Storrs, CT 06268 \\ Einstein Institute of Mathematics, the Hebrew University of Jerusalem, Edmund Safra Campus, Jerusalem 91904, Israel}

\noindent E-mail address: zemels@math.huji.ac.il, yaacov.kopeliovich@uconn.edu


\begin{thebibliography}{12}

\bibitem[BR1]{[BR1]} Bershadsky, M., Radul, A., \textsc{Conformal Field Theories with Additional $Z_{N}$ Symmetry}, Intern. J. Mod. Phys. A, vol. 2, 165--178 (1987).
\bibitem[BR2]{[BR2]} Bershadsky, M., Radul, A., \textsc{Fermionic Fields on $Z_{n}$ Curves}, Comm. Math. Phys., vol. 116, 689--700 (1988).
\bibitem[dJ]{[dJ]} de Jong, R., \textsc{Explicit Mumford Isomorphism for Hyperelliptic Curves}, J. Pure Appl. Algebr., vol. 208 issue 1, 1--14 (2004).
\bibitem[EbF]{[EbF]} Ebin, D., Farkas, H. M., \textsc{Thomae Formula for $Z_{N}$ Curves}, J. Anal. Math., vol 111, 289--320 (2010).
\bibitem[EiF]{[EiF]} Eisenmann, A., Farkas, H. M., \textsc{An Elementary Proof of Thomae's Formula}. OJAC, issue 3 no. 2, 14pp (2008).
\bibitem[EG1]{[EG1]} Enolskii, V., Grava, T., \textsc{Singular $Z_{N}$ Curves and the Riemann Hilbert Problem}, Int. Math. Res. Not., vol. 32, 1619--1683 (2004).
\bibitem[EG2]{[EG2]} Enolskii, V., Grava, T., \textsc{Thomae Type Formulae for Singular $Z_{N}$ Curves} Lett. Math. Phys., vol 76 no. 2-3, 187--214 (2006).
\bibitem[EKZ]{[EKZ]} Enolskii, V., Kopeliovich Y., Zemel S., \textsc{Thomae's Derivative Formulae for Trigonal Curves}, submitted for publication. arXiv link: https://arxiv.org/abs/1810.06031.
\bibitem[Fa]{[Fa]} Fay, J., \textsc{Theta Functions on Riemann Surfaces}, Lecture Notes in Mathematics 352, Springer--Verlag, v+133pp (1973).
\bibitem[FK]{[FK]} Farkas, H. M., Kra, I., \textsc{Riemann Surfaces}, Graduate Text in Mathematics 71, Springer--Verlag, 354pp (1980).
\bibitem[FZ]{[FZ]} Farkas, H. M., Zemel, S., \textsc{Generalizations of Thomae's Formula for $Z_{n}$ curves}, DEVM 21, Springer--Verlag, xi+354pp (2011).
\bibitem[GDT]{[GDT]} Gonzalez-Diez, G., Torres, D., \textsc{$Z_{N}$-Curves Possessing No Thomae Formulae of Bershadsky–-Radul Type}, Lett. Math. Phys., vol 98 no. 2, 193-205 (2011).
\bibitem[JK]{[JK]} Javanpeykar, A., von K\"{a}nel, R., \textsc{Szpiro’s Small Points Conjecture for Cyclic Covers}, Documenta Math., vol 19, 1085--1103 (2014).
\bibitem[Ko1]{[Ko1]} Kopeliovich, Y., \textsc{Theta Constant Identities at Periods of Covers of Degree 3}, Int. J. Number Theory, vol. 4 issue 5, 1--9 (2008).
\bibitem[Ko2]{[Ko2]} Kopeliovich, Y., \textsc{Thomae Formula for General Cyclic Covers of $\mathbb{CP}^{1}$}, Lett. Math. Phys. 94 issue 3, 313--333 (2010).
\bibitem[Ko3]{[Ko3]} Kopeliovich, Y., \textsc{Thomae Formula for 2-Abelian Covers of $\mathbb{CP}^{1}$}, preprint, https://arxiv.org/abs/1605.01139.
\bibitem[KZ]{[KZ]} Kopeliovich, Y., Zemel, S., \textsc{On Spaces Associated with Invariant Divisors on Galois Covers of Riemann Surfaces and Their Applications}, submitted for publication.
\bibitem[Lo]{[Lo]} Lockhart, P., \textsc{On the Discriminant of a Hyperelliptic Curve}, Trans. Amer. Math. Soc., vol 342 no. 2, 729--752 (1994).
\bibitem[M]{[M]} Matsumoto, K., \textsc{Theta Constants Associated with the Cyclic Triple Coverings of the Complex Projective Line Branching at Six Points}, Pub. Res. Inst. Math. Sci., vol 37, 419--440, (2001).
\bibitem[MT]{[MT]} Matsumoto, K., Tomohide, T., \textsc{Degenerations of Triple Covering and Thomae's Formula}, https://arxiv.org/abs/1001.4950.
\bibitem[Mu]{[Mu]} Mumford, D., \textsc{Tata Lectures on Theta II}, Progress in Mathematics 28, Birkh\"{a}user, Boston--Basel--Stuttgart, xiv+272pp (1984).
\bibitem[Na]{[Na]} Nakayashiki, A., \textsc{On the Thomae Formula for $Z_{N}$ Curves}, Publ. Res. Inst. Math Sci., vol. 33 issue 6, 987--1015 (1997).
\bibitem[Ra]{[Ra]} Rauch, H. E., \textsc{Weierstrass Points, Branch Points, and Moduli of Riemann Surfaces}, Comm. Pure Appl. Math., vol. 12 issue 3, 543--560 (1959).
\bibitem[Rd]{[Rd]} Rademacher, H., \textsc{Zur Theorie der Dedekindschen Summen}, Math. Z., vol. 63, 445--463 (1956).
\bibitem[SB]{[SB]} Shepherd-Barron, N. I., \textsc{Thomae's Formulae for Non-Hyperelliptic Curves and Spinorial Square Roots of Theta-Constants on the Moduli Space of Curves}, pre-print, https://arxiv.org/abs/0802.3014 (2008).
\bibitem[Th1]{[Th1]} Thomae, J., \textsc{Bestimmung von $\mathrm{d}\log\theta(0,\ldots,0)$ durch die Klassmoduln}, J. Reine Angew. Math., vol 66, 92--96 (1866).
\bibitem[Th2]{[Th2]} Thomae, J., \textsc{Beitrag zur Bestimmung von $\theta(0,\ldots,0)$ durch die Klassmoduln Algebraischer Funktionen}, J. Reine Angew. Math., vol 71, 201--222 (1870).
\bibitem[vK]{[vK]} von K\"{a}nel, R., \textsc{On Szpiro’s Discriminant Conjecture}, Int. Math. Res. Not., vol 2014 issue 16, 4457--4491 (2014).
\bibitem[Z]{[Z]} Zemel, S. \textsc{Thomae Formulae for General Fully Ramified $Z_{n}$ Curves}, J. Anal. Math., vol 131, 101--158 (2017).
\end{thebibliography}
\end{document}